\documentclass[10pt,a4paper,twoside,reqno]{amsart}

\usepackage{geometry}

\usepackage{xcolor,soul,lipsum}
\usepackage{hyperref,url}
\usepackage{latexsym}
\usepackage{latexsym, amsmath, amstext, amssymb, amsfonts, amscd, bm, array, multirow, amsbsy, mathrsfs, mathtools}
\usepackage{amsthm}
\usepackage{t1enc}
\usepackage{enumerate}
\usepackage[all]{xy} \CompileMatrices
\usepackage{amscd}
\hypersetup{colorlinks=false,pdfborderstyle={/S/U/W 0}}
\usepackage{slashed} 
\usepackage{tikz-cd}
\usetikzlibrary{matrix}
\usepackage[normalem]{ulem}
\SelectTips{cm}{12} 
\theoremstyle{plain}
\newtheorem{theorem}{Theorem}[section]
\newtheorem{lemma}[theorem]{Lemma}
\newtheorem{corollary}[theorem]{Corollary}
\newtheorem{proposition}[theorem]{Proposition}

\theoremstyle{definition}
\newtheorem{definition}[theorem]{Definition}
\newtheorem{definition-theorem}[theorem]{Definition-Theorem}

\theoremstyle{remark}
\newtheorem{remark}[theorem]{Remark}

\numberwithin{equation}{section}
\usepackage{geometry}
\usepackage{enumitem}

\usepackage[utf8]{inputenc}
\usepackage{color}

\newcommand{\tr}{\operatorname{Tr}}
\newcommand{\Id}{\operatorname{Id}}

\newcommand{\End}{\operatorname{End}}

\newcommand{\Ker}{\operatorname{Ker}}

\newcommand{\Ad}{\operatorname{Ad}}
\newcommand{\Aut}{\operatorname{Aut}}
\newcommand{\Ric}{\operatorname{Ric}}
\newcommand{\ddiv}{\operatorname{div}}

\newcommand{\CC}{{\mathbb C}}

\newcommand{\RR}{{\mathbb R}}

\newcommand{\surj}{\to\kern-1.8ex\to}

\newcommand{\cG}{\mathcal{G}}

\newcommand{\cA}{\mathcal{A}}
\newcommand{\cL}{\mathcal{L}}
\newcommand{\cK}{\mathcal{K}}

\newcommand{\cU}{\mathcal{U}}

\newcommand{\cT}{\mathcal{T}}

\newcommand{\cV}{\mathcal{V}}
\newcommand{\cR}{\mathcal{R}}

\newcommand{\cZ}{\mathcal{Z}}
\newcommand{\Diff}{\mathrm{Diff}}
\newcommand{\cC}{\mathcal{C}}
\newcommand{\Conf}{\mathrm{Conf}}

\newcommand{\Met}{\mathrm{Met}}
\newcommand{\s}{\mathrm{S}}

\newcommand{\G}{\mathrm{G}}

\newcommand{\cE}{\mathcal{E}}
\newcommand{\cO}{\mathcal{O}}
\newcommand{\cN}{\mathcal{N}}

\newcommand{\arxiv}[1]{{\tt
		\href{http://www.arXiv.org/abs/#1}{arXiv:#1}}}
\newcommand{\norm}[1]{\left\lVert#1\right\rVert}
\newcommand{\cS}{\mathcal{S}}

\newcommand{\cI}{\mathcal{I}}
\newcommand{\dd}{\mathrm{d}}

\newcommand{\frg}{\mathfrak{g}}
\newcommand{\frc}{\mathfrak{c}}

\newcommand{\Iso}{\mathrm{Iso}}
\newcommand{\frF}{\mathfrak{F}}
\newcommand{\frG}{\mathfrak{G}}
\newcommand{\mG}{\mathbb{G}}

\newcommand{\frad}{{\mathfrak{Ad}}}

\newcommand{\U}{\mathrm{U}}

\usepackage{enumitem}
\setlist[itemize]{leftmargin=*}

\newcolumntype{P}[1]{>{\centering\arraybackslash}p{#1}}

\usepackage[math]{anttor}
\usepackage[T1]{fontenc}
 
\begin{document}

\title[The local moduli space of the Einstein Yang-Mills system]{The local moduli space of the Einstein-Yang-Mills system}
 
\author[S. Bunk]{Severin Bunk} \address{Mathematical Institute, University of Oxford, United Kingdom}
\email{severin.bunk@maths.ox.ac.uk}

\author[V. Mu\~noz]{Vicente Mu\~noz} \address{Instituto de Matem\'atica Interdisciplinar and Departamento de \'Algebra, Geometr\'{\i}a y Topolog\'{\i}a, Universidad Complutense de Madrid, Spain}
\email{vicente.munoz@ucm.es}

\author[C. S. Shahbazi]{C. S. Shahbazi} \address{Departamento de Matem\'aticas, Universidad UNED - Madrid, Reino de Espa\~na}
\email{cshahbazi@mat.uned.es} 
\address{Fakult\"at f\"ur Mathematik, Universit\"at Hamburg, Bundesrepublik Deutschland.}
\email{carlos.shahbazi@uni-hamburg.de}


\begin{abstract} 
We study the deformation theory of the Einstein-Yang-Mills system on a principal bundle with a compact structure group over a compact manifold. We first construct, as an application of the general slice theorem of Diez and Rudolph \cite{DiezRudolph}, a smooth slice in the tame Fréchet category for the coupled action of bundle automorphisms on metrics and connections. Using this result, together with a careful analysis of the linearization of the Einstein-Yang-Mills system, we realize the moduli space of Einstein-Yang-Mills pairs modulo automorphism as an analytic set in a finite-dimensional tame Fréchet manifold, extending classical results of Koiso for Einstein metrics and Yang-Mills connections to the Einstein-Yang-Mills system. Furthermore, we introduce the notion of \emph{essential deformation} of an Einstein-Yang-Mills pair, which we characterize in full generality and explore in more detail in the four-dimensional case, proving a decoupling result for trace deformations when the underlying Einstein-Yang-Mills pair is a Ricci-flat metric coupled to an anti-self-dual instanton. In particular, we find a novel obstruction that does not occur in the \emph{decoupled} Einstein or Yang-Mills moduli problems. Finally, we prove that every essential deformation of the four-dimensional Einstein-Yang-Mills system based on a Calabi-Yau metric coupled to an instanton is of restricted type. 
\end{abstract}

\maketitle

\setcounter{tocdepth}{1} 
\tableofcontents


\section{Introduction}
\label{sec:intro}



\subsection*{Introduction and motivation}


The main purpose of this article is to investigate the deformation problem of the fully coupled Einstein-Yang-Mills system, presented in Equation \eqref{eq:YMequations}, in terms of the variational problem associated with the Einstein-Yang-Mills functional \eqref{eq:EYMFunctional}, both at the infinitesimal and local levels. Whereas the theory of Einstein metrics and Yang-Mills connections, especially instantons, treated individually, are by now classical \cite{AtiyahBott,Besse,DonaldsonKronheimer,LubkeTeleman}, their interaction through naturally coupled differential systems is, arguably, still in its infancy. Pioneering results in this direction include \cite{Alvarez-Consul:2011tpl,Alvarez-ConsulGFO,Alvarez-ConsulGFOP,ChanSuen,GFThesis,GFTipler,Huang,LiuTu}, where the authors study various natural first-order coupled deformation problems in complex geometry, as well as the recent explosion of activity in the context of coupled moduli problems motivated by supergravity and string theory \cite{Clarke:2020erl,GFR,GFRST,GFRSTII,Garcia-Fernandez:2015hja,GFRTGauge,Moroianu:2021kit,Moroianu:2023jof}. Given its relevance, the latter problem has also been extensively considered in the physics literature, see \cite{Grana:2014rpa,Grana:2014vxa,GranaTriendl,McOrist:2019mxh,Tomasiello} and references therein for more details. 

The Einstein-Yang-Mills system is a natural second-order non-linear system of partial differential equations for a Riemannian metric $g$ on a compact manifold $M$ and a connection $A$ on a principal bundle $P$ over $M$ with a compact structure group $\G$. In particular, this system provides a natural extension, or alternatively, a natural coupling, between the Einstein equations for Riemannian metrics and the Yang-Mills equations for connections on a principal bundle. Despite arising from arguably the most natural combination of the Einstein and Yang-Mills functionals, the Riemannian Einstein-Yang-Mills system has been only scarcely studied in the literature, with some notable exceptions in the abelian case $\G=S^1$ on a complex manifold \cite{ApostolovMaschler,LeBrun1,LeBrun2,LeBrun3}, on Lie algebras \cite{KocaLejmi}, and very recently on complete hyperbolic manifolds \cite{Lopes}. Closely related to the Einstein-Yang-Mills system on a complex manifold, in references \cite{Alvarez-Consul:2011tpl,Alvarez-ConsulGFO,Alvarez-ConsulGFOP,GFThesis,GFTipler} the authors study Kähler metrics on a fixed complex manifold coupled to Yang-Mills connections via various differential systems that admit natural momentum map interpretations. On the other hand, the Einstein-Yang-Mills system in Lorentzian signature has a long history in the mathematics literature, see for instance \cite{ChenYau,SmollerWassermanYauMcLeod,SmollerWasserman,SmollerWassermanII} as well as their references and citations. In fact, we can think of the solutions of the Einstein-Yang-Mills system considered in this article as \emph{static} solutions with trivial \emph{warp factor} of the Lorentzian Einstein-Yang-Mills system in one higher dimension. Additionally, the Riemannian Einstein-Yang-Mills pairs that we consider in this article occur as particular case of the self-similar solitons of the Ricci-Yang-Mills flow \cite{Streets}.

The local moduli spaces of solutions of both the Einstein and Yang-Mills equations have been studied by Koiso \cite{KoisoI,KoisoII,KoisoIII,KoisoIV,KoisoV}, who greatly contributed to developing the foundations for the deformation problems of these systems. In particular, Koiso introduced the notion of essential deformation (these consist of infinitesimal deformations that cannot be eliminated via the infinitesimal action of the symmetry group) and infinitesimal and local rigidity, with the goal of proving various rigidity results for Einstein metrics on symmetric spaces. In turn, Koiso's approach builds on the seminal ideas of Kuranishi \cite{KuranishiI,KuranishiII}, who first constructed the local \emph{versal} moduli of complex structures on a compact manifold via functional analytic methods on infinite-dimensional manifolds. In this article, we will apply the ideas of Koiso and Kuranishi, further explored and expanded by Dondalson and others \cite{DonaldsonKronheimer}, to the more involved case of the Einstein-Yang-Mills system with the goal of obtaining a \emph{Kuranishi model} for the local moduli of Einstein-Yang-Mills pairs modulo isomorphism. Hence, our working framework is that of \cite{MeerssemanNicolau,DiezRudolphII}, where the general theory of Kuranishi models for moduli spaces of appropriately defined geometric structures on manifolds was developed in full generality. As explained in the main text, the deformation problem of the Einstein-Yang-Mills system, being a truly coupled deformation problem, has several genuine properties that cannot be understood from the separate study of the moduli of Einstein metrics and Yang-Mills connections and therefore deserves a proper study by itself. 


\subsection*{Main results}


The layout of this article is as follows. In Section \ref{sec:PrincipalBundles} we describe the automorphism group of a principal bundle $P$ with a compact base and a compact structure group as a closed tame Fréchet Lie subgroup of the tame Fréchet Lie group of all smooth diffeomorphisms of $P$. This is the natural symmetry group of the Einstein-Yang-Mills system. In Section \ref{sec:EinsteinYM} we introduce the Einstein-Yang-Mills system and describe some of its elementary properties, realizing the configuration space of the system as a closed tame Fréchet submanifold of the tame Fréchet manifold of smooth metrics on $P$. In Section \ref{sec:SliceP}, see Theorem \ref{thm:slice}, we construct a smooth Fréchet slice for the action of $\Aut(P)$ on $\Conf(P)$ as an application of the general slice theorem obtained in \cite{DiezRudolph}%
\footnote{Our naming convention for symmetries and gauge transformations of the principal bundle $P$ differs from that in~\cite{BunkMuellerSzabo, BunkShahbazi, Collier}: what we refer to as \textit{automorphisms} here is called \textit{symmetries} in those articles, and what we call \textit{gauge transformations} here is called \textit{automorphisms} in those articles.
The present conventions may be more commonly used in differential geometry.}.
In doing so, we obtain an  \emph{infinitesimal slice condition}, which does not seem to have been explored in the literature, for the infinitesimal deformations of a given pair $(g,A)\in \Conf(P)$ (Cf. \cite{Lopes}). In Section \ref{sec:Kuranishi} we obtain the deformation complex of an Einstein-Yang-Mills pair $(g,A)$ as a self-dual elliptic resolution of the sheaf of infinitesimal automorphisms of $(g,A)$. Furthermore, following Koiso \cite{KoisoII}, we introduce the notion of \emph{essential deformation} of an Einstein-Yang-Mills pair and we characterize it generally, obtaining a seemingly novel pair of obstructions for a pair $(h,a)\in T_{(g,A)}\Conf(P)$ to be an essential deformation of $(g,A)$. Using the aforementioned characterization together with the smooth slice constructed in Section \ref{sec:SliceP}, in Theorem \ref{thm:localKuranishi} we describe the local moduli space of Einstein-Yang-Mills pairs as an analytic subspace of a finite-dimensional tame Fréchet manifold whose tangent space at $(g,A)$ is given by the vector space of essential deformations of the latter. This is the Kuranishi model for the moduli space of Einstein-Yang-Mills. In Section \ref{sec:instantonK3} we fix $M$ to be four-dimensional and we study the essential deformations of a pair $(g,A)$ consisting of a Ricci-flat metric $g$ and an anti-self-dual instanton $A$, proving in Theorem \ref{thm:deformationsEYMASD} that every trace deformation of $(g,A)$ decouples and is determined as the potential of a certain explicit cohomology class that must vanish. Furthermore, when $g$ is a Calabi-Yau metric and $A$ an anti-self-dual instanton, we show in Theorem \ref{thm:CYcase} that every essential Einstein-Yang-Mills deformation of $(g,A)$ is actually an essential deformation of $(g,A)$ as a Calabi-Yau metric coupled to an anti-self-dual instanton. Finally, in Appendix \ref{app:curvatureformulae} we summarize the Riemannian formulae that we use throughout the article with the goal of establishing our notation and conventions.


\subsection*{Open problems}


The present article aims at establishing the foundations of the deformation problem of the Einstein-Yang-Mills system and, as a consequence, it leaves open many possible lines of research. Most of these potential research lines are motivated by extending the theory of Einstien metrics currently available in the literature to the Einstein-Yang-Mills case. We highlight the following potential open problems:
\begin{itemize}

\item Construct examples of Einstein-Yang-Mills pairs whose Yang-Mills connection is not an instanton.

\item Classify homogenous Riemannian manifolds carrying natural Einstein-Yang-Mills pairs consisting of the normal homogeneous Riemannian metric and canonical homogeneous connection. 

\item Study the vector space of essential deformations of an Einstein-Yang-Mills pair and the rigidity of Einstein-Yang-Mills pairs on explicit examples.

\item Study the second variation of the Einstein-Yang-Mills functional and the stability of Einstein-Yang-Mills pairs, see \cite{KoisoII,Kroncke,KronckeII} and their references and citations for analog results in the Einstein case. 

\item Generalize the results of \cite{NagySemmelmann} in order to investigate the second-order deformations of the Einstein-Yang-Mills system.
\end{itemize}

\noindent
We hope to address some of these problems in the future.


\subsection*{Acknowledgements}


CSS would like to thank Tobias Diez for his interesting comments and suggestions. SB's research was funded by the Deutsche Forschungsgemeinschaft (DFG, German Research Foundation) under the project number 468806966. VM has been partially supported by Ministerio de Ciencia e Innovaci\'on Project PID2020-118452GB-I00 (Spain). The work of CSS is funded by the Germany Excellence Strategy \emph{Quantum Universe} - 390833306 and the 2022 Leonardo Grant for Researchers and Cultural Creators of the BBVA Foundation.


\section{Principal bundles and transitive Lie algebroids}
\label{sec:PrincipalBundles}


In this section, we summarize the background on principal bundles and their groups of automorphisms that we shall need throughout the article. In the following, $M$ will denote a connected, compact, and oriented $n$-dimensional manifold without boundary, and $\G$ will denote a compact Lie group with Lie algebra $\frg$.


\subsection{The Atiyah algebroid of a principal bundle}
\label{sec:Atiyahalgebroid}


Let $\pi\colon P\to M$ be a principal bundle over $M$ with structure group $\G$. We denote the principal-bundle right action $\Psi\colon P\times \G \to P$ of $\G$ on $P$ simply by juxtaposition:
\begin{equation*}
\Psi\colon P\times \G \to P \, , \qquad (p,o)\mapsto p o\, , \qquad \Psi_o (p) := p o\, .
\end{equation*}

\noindent
We denote by $\frad(P) = P \times_{\Ad} \frg$ the adjoint bundle of Lie algebras associated with $P$. Taking the differential of the action $\Psi$ at each given element $g\in \G$ we obtain the infinitesimal action $\partial\Psi$ induced naturally by $\Psi$ on $TP$:
\begin{equation*}
\partial\Psi\colon TP\times \G \to TP \, , \qquad (X,o)\mapsto \dd\Psi_o(X)\, .
\end{equation*}

\noindent
In particular, for every $o\in \G$ the map $\dd\Psi_o\colon TP\to TP$ is a vector bundle isomorphism covering the equivariant diffeomorphism $\Psi_o\colon P \to P$. The tangent bundle $\varrho\colon TP \to P$ admits a trivialization atlas which is equivariant with respect to the action $\partial\Psi$. Therefore the corresponding quotient:
\begin{equation*}
\cA_P := \frac{TP}{\G}\, ,
\end{equation*}

\noindent
admits a unique vector bundle structure $\pi_P\colon \cA_P \to M$ over $M$ such that the canonical quotient projection $\beta_P\colon TP\to \cA_{P}$ is a surjective submersion and a fiber bundle map making the following diagram commutative \cite[Appendix A]{KMackenzie}:
\[ \begin{tikzcd}
TP \arrow{r}{\beta_P} \arrow[swap]{d}{\varrho} & \cA_P \arrow{d}{\pi_P} \\%
P \arrow{r}{\pi}& M
\end{tikzcd}
\]

\noindent
where $\pi_P\colon \cA_P \to M$ is the vector bundle projection of $\cA_P$, given by construction as follows:
\begin{equation*}
\pi_P(\beta_P(X)) = \pi(\varrho(X))\, , \qquad X\in TP\, . 
\end{equation*}

\noindent
Note that defining $\pi_P\colon \cA_P \to M$ by the previous equation is consistent since for every fixed $a\in \cA_P$ the elements in $\beta^{-1}_P(a)$ are vectors over the same orbit in $P$. More explicitly, the fiber $\cA_P\vert_m$ of $\cA_P$ at the point $m\in M$ is given by the vector space of all vectors on $\pi^{-1}(m)$ modulo the action $\partial\Psi$, that is:
\begin{equation*}
\cA_P\vert_m = \frac{TP\vert_{\pi^{-1}(m)}}{\sim} \, , \quad T_{p_1}P\ni X_1 \sim X_2 \in T_{p_2}P \,\,\, \mathrm{iff}\,\,\, \exists\, o\in \G:\, \dd\Psi_{o}(X_1) = X_2\, .
\end{equation*}

\noindent
The fact that $\pi\circ \Psi_o = \pi$ for every $o\in \G$ implies that the differential map $\dd\pi \colon TP \to TM$ factors through the canonical projection $\beta_P \colon TP\to \cA_P$ and therefore defines a vector bundle map $\rho_P \colon \cA_P \to TM$ given by $a \mapsto \dd\pi(X)$ for any representative $X\in \beta^{-1}_P(a)  \subset TP$. This map is surjective, and its kernel defines a vector bundle over $M$ that we denote by:
\begin{equation*}
\frG_P = \mathrm{Ker}(\rho_P) = \frac{\mathrm{Ker}(\dd\pi)}{\G}\, .
\end{equation*}

\noindent
This vector bundle is canonically embedded in $\cA_P$. Altogether we obtain the following short exact sequence of vector bundles:
\begin{equation}
\label{eq:Atiyahsequence}
0\to \frG_P \hookrightarrow \cA_P \xrightarrow{\rho_P} TM \to 0\, .
\end{equation}

\noindent
The geometric significance of $\cA_P$ is made apparent by the following result, which is a direct consequence of the definition of $\beta_P \colon TP \to \cA_P$. 

\begin{proposition}
There exists a canonical isomorphism of $C^{\infty}(M)$-modules:
\begin{equation*}
\theta_P\colon\mathfrak{X}(P)^{\G} \xrightarrow{\simeq} \Gamma(\cA_P)\, , \qquad  X \mapsto \beta_P(X)\, ,
\end{equation*}

\noindent
between the $C^{\infty}(M)$ module $\mathfrak{X}(P)^{\G}$ of $\G$-invariant vector fields on $P$ and the $C^{\infty}(M)$ module of sections of $\cA_P$.
\end{proposition}

\begin{remark}
The inverse of an element $a\in \Gamma(\cA_P)$ by the previous bijection is given, at a point $p\in \pi^{-1}(m)$, $m\in M$, by:
\begin{equation*}
\theta_P^{-1}(a)\vert_p = \beta_P^{-1}(a\vert_{\pi(p)})\in T_p P\, ,
\end{equation*}

\noindent
where the right-hand side denotes the unique element in the preimage of $a\vert_{\pi(p)}$ by $\beta_P$ that belongs to $T_p P$. In the following, we will consider $\mathfrak{X}(P)^{\G} = \Gamma(\cA_P)$ to be identified as established by the previous bijection. 
\end{remark}

\noindent
Using the previous identification between $\mathfrak{X}(P)^{\G}$ and $\Gamma(\cA_P)$ we define a bilinear map
\begin{equation*}
    [\cdot,\cdot]_{\cA_P}\colon \Gamma(\cA_P)\times \Gamma(\cA_P) \to \Gamma(\cA_P)
\end{equation*}
on the sections of $\cA_P$ as the bracket of the corresponding $\G$-invariant vector fields on $P$ (recall that the bracket of two $\G$-invariant vector fields is again $\G$-invariant), that is:
\begin{eqnarray*}
[a_1,a_2]_{\cA_P} = \theta_P([\theta_P^{-1}(a_1),\theta_P^{-1}(a_2)])\, , \qquad a_1, a_2 \in \Gamma(\cA_P)\, ,
\end{eqnarray*}

\noindent
where $[\cdot,\cdot]$ denotes the standard bracket of vector fields on $P$. The bracket $[\cdot,\cdot]_{\cA_P}$ is skew and satisfies both the Jacobi identity and a Leibniz-type identity with respect to $\rho_P\colon \cA_P\to TM$. Indeed, given $f\in C^{\infty}(M)$, we have:
\begin{align*}
[a_1, f a_2]_{\cA_P}
&= \theta_P([\theta_P^{-1}(a_1),\theta_P^{-1}( f a_2)])
= \theta_P([\theta_P^{-1}(a_1), (f\circ \pi) \theta_P^{-1}( a_2)])
\\
&= \theta_P( (f\circ \pi)  [\theta_P^{-1}(a_1),  \theta_P^{-1}( a_2)] + \theta_P^{-1}(a_1)(f\circ \pi)\, \theta_P^{-1}( a_2) )
\\
&= f [a_1,a_2]_{\cA} + \rho_P(a_1)(f) a_2\, .
\end{align*} 

\noindent
Therefore, if $a_1,a_2\in \Gamma(\frG_P)$ then $[a_1, f a_2]_{\cA_P} = f [a_1, a_2]_{\cA_P}$ for every function $f\in C^{\infty}(M)$, which immediately implies that $(\frG_P , [\cdot, \cdot]_{\frG_P})$, where $[\cdot, \cdot]_{\frG_P}$ denotes the restriction of $[\cdot, \cdot]_{\cA_P}$ to $\Gamma(\frG_P)$, is a Lie algebra bundle. Furthermore, for every $a_1 , a_2 \in \Gamma(\cA_P)$ we have:
\begin{eqnarray*}
& \rho_P ([a_1,a_2]_{\cA_{P}}) = \rho_{P}(\theta_P([\theta_P^{-1}(a_1),\theta_P^{-1}(a_2)])) = \dd \pi ([\theta_P^{-1}(a_1),\theta_P^{-1}(a_2)]) = [\dd \pi\circ\theta_P^{-1}(a_1),\dd \pi\circ\theta_P^{-1}(a_2)]\\
& = [\rho_{P}(a_1),\rho_{P}(a_2)]
\end{eqnarray*}

\noindent
and thus $\rho_P\colon \cA_P \to TM$ defines a morphism of Lie algebras from $(\Gamma(\cA_P), [\cdot,\cdot]_{\cA_P})$ to $(\mathfrak{X}(M),[\cdot ,\cdot])$.

\begin{definition}
Considering $\frG_P$ to be equipped with $[\cdot,\cdot]_{\frG_P}$ and $\cA_P$ to be equipped with $[\cdot,\cdot]_{\cA_P}$ the short exact sequence of Lie algebroids \eqref{eq:Atiyahsequence} is the Atiyah sequence of $P$. 
\end{definition}

\noindent
For every $p\in P$, define $\Psi_p\colon \G \to P$ by $\Psi_p(o)= p o$. The differential of this map at the identity $e\in \G$ defines an injective linear map of vector spaces $\dd_e \Psi_p\colon \frg\to T_p P$ which can be used to define the infinitesimal action of $\G$ on $P$, namely:
\begin{equation*}
D\Psi \colon P\times \frg \to TP\, , \qquad (p,x)\mapsto \dd_e\Psi_p(x)\, . 
\end{equation*}

\noindent
In particular, the fundamental vector field $v_x\in\mathfrak{X}(P)$ associated to $x\in\frg$ is by definition:
\begin{equation*}
v_x\vert_p = \dd_e\Psi_p(x)\, , \qquad p\in P\, .
\end{equation*}

\noindent
Note that $D\Psi \colon P\times \frg \to TP$ defines by corestriction an isomorphism of vector bundles:
\begin{equation*}
\dd\Psi_{\pi} \colon M\times \frg \to \cV
\end{equation*}

\noindent
that trivializes the vertical bundle $\cV := \Ker(\dd\pi)$. The infinitesimal action $D\Psi \colon P\times \frg \to TP$ can be shown to be equivariant with respect to $\partial\Psi\colon TP\times \G \to TP$ on $TP$ and the right action: 
\begin{equation*}
P\times \frg \times \G \to P\times \frg\, , \qquad (p,x,o)\mapsto (p o,\Ad_o^{-1}(x))\, ,
\end{equation*}

\noindent
on $P\times \frg$, where $\G\ni o\mapsto \Ad_o\in \Aut(\frg)$ denotes the adjoint representation of $\G$. Therefore, $D\Psi \colon P\times \frg \to TP$ descends to a bundle map $\dd\Psi_{\G} $ between the corresponding vector bundle quotients by $\G$:
\begin{equation*}
\dd\Psi_{\G} \colon \mathfrak{Ad}(P) \to \cA_P\, , \qquad [p,x] \mapsto \dd_e\Psi_p(x)\, ,
\end{equation*}

\noindent
for every $p\in P$ and $x\in \frg$. This map is clearly injective and its image is precisely $\frG_P \subset \cA_P$. The following identification is well-known.

\begin{proposition}
\label{prop:isoLiealgebrabundles}
The bundle map $\dd\Psi_{\G} \colon \mathfrak{Ad}(P) \to \frG_P$ defines an isomorphism of Lie algebra bundles between the adjoint bundle of algebras $\mathfrak{Ad}(P)$ and $(\frG_P , [\cdot , \cdot]_{\frG_P})$.
\end{proposition}

\noindent
Connections on $P$ admit a convenient geometric reinterpretation in terms of its associated Atiyah sequence, that is, there exists a one-to-one correspondence between connections on $P$ and vector-bundle splittings of the Atiyah sequence of $P$ as a sequence of vector bundles (splittings which respect the brackets correspond to flat connections). Every splitting $\kappa \colon TM \to \cA_P$ of $\cA_P$, or equivalently, every connection on $P$, defines an isomorphism of vector bundles:
\begin{equation*}
\psi_{\kappa} \colon TM\oplus \frG_P \to \cA_P\, , \qquad v + \tau \mapsto \kappa(v) + \tau\, ,
\end{equation*}

\noindent
where $\frG_P$ is canonically embedded in $\cA_P$ via the Atiyah sequence of $P$. Using this vector bundle isomorphism we can transport the bundle map $\rho_{P}\colon \cA_P \to TM$ and the bracket $[\cdot,\cdot]_{\cA_P}$ on $\cA_P$ to $TM \oplus \frG_P$. In order to do this we define:
\begin{equation*}
\rho_{\kappa}:= \rho_{P}\circ \psi_{\kappa}\colon TM \oplus \frG_P\to TM\, ,\qquad  [v_1 + \tau_1,v_2+\tau_2]_{\kappa} = \psi_{\kappa}^{-1}([\psi_{\kappa}(v_1 + \tau_1),\psi_{\kappa}(v_2 + \tau_2)]_{\cA_P})\, ,
\end{equation*}

\noindent
for every $v_1 + \tau_1, v_2 + \tau_2\in TM\oplus \frG_P$. A direct computation shows that:
\begin{equation}
\label{eq:standardAityah}
\rho_{\kappa}(v + \tau) = v\, , \quad [v_1 + \tau_1,v_2+\tau_2]_{\kappa} = [v_1,v_2] + \nabla^{\kappa}_{v_1}\tau_2 - \nabla^{\kappa}_{v_2}\tau_1 + \cR^{\kappa}(v_1,v_2) + [\tau_1,\tau_2]_{\frG_P}\, ,
\end{equation}

\noindent
where $[\cdot,\cdot]$ denotes the standard bracket of vector fields on $M$, $\nabla^{\kappa}$ is the connection induced on $\frG_P$ by $\kappa\colon TM \to \cA_P$ and $\cR^{\kappa}$ is its curvature.

\begin{remark}
The splitting $\kappa\colon TM \to \cA_P$ is equivalent, as explained above, to a connection $A$ on $P$, which in turn induces a connection $\nabla^{\kappa}$ on the adjoint bundle $\mathfrak{Ad}(P)$. By Proposition \ref{prop:isoLiealgebrabundles} this connection transfers to $\frG_P$ and in addition preserves the Lie algebra bundle structure of $(\frG_P, [\cdot,\cdot]_{\frG_P})$. We have the formula:
\begin{equation*}
\mathfrak{R}^{\kappa}(v_1,v_2) \tau = \nabla^{\kappa}_{v_1}\nabla^{\kappa}_{v_2}\tau - \nabla^{\kappa}_{v_2}\nabla^{\kappa}_{v_1}\tau  -\nabla^{\kappa}_{[v_1,v_2]}\tau =  [\cR^{A}(v_1,v_2),\tau]_{\frG_P}\, ,
\end{equation*}

\noindent
where $\cR^{A}\in \Omega^2(\mathfrak{Ad}(P))$ denotes the curvature of the connection $A$ on $P$.
\end{remark}

\noindent
In the following we will always consider the adjoint bundle $\mathfrak{Ad}(P)$ to be identified with $(\frG_P,[\cdot,\cdot]_{\frG_P})$.


\subsection{The automorphism group of a principal bundle}


Let $\pi\colon P\to M$ be a principal bundle over $M$ with structure group $\G$. Note that since both $M$ and the fibers of $\pi\colon P\to M$ are compact without boundary, the total space $P$ is again a compact manifold without boundary. We denote by $\Diff(P)$ the group of diffeomorphisms of $P$, which we consider as a Fréchet Lie group \cite{Hamilton,Leslie}.

\begin{remark}
The diffeomorphism group of a compact manifold carries a richer structure: it is in fact a \emph{tame} Fréchet Lie group \cite{Hamilton,Subramanian} as well as a strong Inverse Limit Hilbert Lie group \cite{OmoriBook}. In particular, it can be realized as a projective limit of Banach manifolds endowed with the structure of a topological group. We will use this extra structure at our convenience in the sequel. 
\end{remark}

\noindent
We refer to \cite{FrechetProjective,Hamilton} for the basic definitions regarding the theory of Fréchet manifolds and groups and their realizations as projective limits of Hilbert manifolds. We define the automorphism group $\Aut(P)$ of $P$ as the group of equivariant diffeomorphisms of $P$, that is:
\begin{equation*}
\Aut(P) := \left\{ u\in \Diff(P) \,\,\, \vert \,\,\, u(po) = u(p)o \,\,\, \forall\, o\in \G  \right\}\, .
\end{equation*}

\noindent
Every equivariant diffeomorphism $u\in \Aut(P)$ covers a unique diffeomorphism $f_u\colon M\to M$ of $M$ fitting into the following commutative diagram:
\begin{equation}
\begin{tikzcd}[column sep=1.75cm, row sep=1cm]
P \ar[r, "u"] \ar[d]
& P \ar[d]
\\
M \ar[r, "f_u"]
& M
\end{tikzcd}
\end{equation}

\noindent
We introduce the gauge group $\cG(P)\subset \Aut(P)$ of $P$ as those equivariant diffeomorphisms of $P$ that cover the identity diffeomorphism of $M$:
\begin{equation*}
\cG(P) := \left\{ u\in \Aut(P) \,\,\, \vert \,\,\, \pi\circ u = \pi \right\}\, .  
\end{equation*}

\noindent
This is the group usually considered as the symmetry group of the differential systems studied in mathematical gauge theory. We will refer to elements in $\cG(P)$ simply as \emph{gauge transformations}. 

\begin{proposition}
\label{prop:AutPFrechet}
The automorphism group $\Aut(P)\subset \Diff(P)$ is a closed tame Fréchet Lie subgroup of $\Diff(P)$, locally modeled on the tame Fréchet Lie algebra $\mathfrak{X}(P)^{\G}$ with the standard bracket of vector fields or, equivalently, $\Gamma(\cA_{P})$ with Lie bracket $[\cdot,\cdot]_{\cA_{P}}$. Furthermore, $\cG(P)\subset \Aut(P)$ is a closed tame Fréchet Lie subgroup of $\Aut(P)$, locally modeled on the Fréchet Lie algebra $\Gamma(\frG_P)$ with Lie bracket $[\cdot,\cdot]_{\frG_P}$.
\end{proposition}

\begin{proof}
Let $\bar{g}$ be a $\G$-invariant metric on $P$ and consider $\mathfrak{X}(P)$ as a tame Fréchet space with respect to the family of Sobolev norms $\left\{\norm{-}_s\right\}_{s>n+4}$ constructed out of $\bar{g}$. Let $(X_k)\subset \mathfrak{X}(P)^{\G}$ be a sequence of invariant vector fields on $P$ converging to an element $X\in \mathfrak{X}(P)$. By definition, we have:
\begin{equation*}
\lim_{k\to \infty} \norm{ X_k - X}_s = 0 \qquad \forall\,\, s>n+4\,.
\end{equation*}

\noindent
For every $o\in \G$ we have:
\begin{equation*}
0=\lim_{k\to \infty} \norm{ X_k- X}_s = \lim_{n\to \infty} \norm{ \dd\Psi_o (X_k) - \dd\Psi_o\circ \dd\Psi^{-1}_o (X)}_s = \lim_{k\to \infty} \norm{X_k  -  \dd\Psi^{-1}_o (X)}_s \qquad \forall\,\, s>n+4\,,
\end{equation*}

\noindent
which implies:
\begin{equation*}
\lim_{k\to \infty} X_k = \dd\Psi^{-1}_o (X)\,,
\end{equation*}

\noindent
for every $o\in\G$ and $s>n+4$. By uniqueness of limits we conclude $\dd\Psi_o (X) = X$ whence $X\in\mathfrak{X}(P)^{\G}$. Therefore, $\mathfrak{X}(P)^{\G}$ is closed in $\mathfrak{X}(P)$. Consequently, we obtain a short exact sequence of tame Fréchet spaces:
\begin{equation*}
0 \to \mathfrak{X}(P)^{\G} \to \mathfrak{X}(P) \to \mathfrak{X}(P)/\mathfrak{X}(P)^{\G}  \to 0\,.
\end{equation*}

\noindent
Short exact sequences do not necessarily split in the Fréchet category. Equivalently, closed subspaces of a Fréchet space do not necessarily admit a topological complement. Our case is, however, special since the Fréchet space $\mathfrak{X}(P)$ is the projective limit of the Sobolev chain $\left\{ \mathfrak{X}_s(P)\, , \, s > n+4 \right\}$, where $\mathfrak{X}_s(P)$ is the completion of $\mathfrak{X}(P)$ in the Sobolev norm $H_s = L^2_s$. In particular:
\begin{equation*}
\mathfrak{X}(P) = \bigcap_{s>n+4} \mathfrak{X}_s(P)
\end{equation*}

\noindent
with the relative topology of the intersection as a subspace of each factor $\mathfrak{X}_s(P)$. Similarly, we denote by $\mathfrak{X}_s(P)^{\G}$ the completion of $\mathfrak{X}(P)^{\G}$ in the Sobolev norm $H_s$. Let $\mathfrak{X}_s(P)^{\G}_{\perp} \subset \mathfrak{X}_s(P)$ be the orthogonal complement of $\mathfrak{X}_s(P)^{\G}$ in $\mathfrak{X}_s(P)$ with respect to its Hilbert inner product. Then:
\begin{equation*}
f_s \colon \mathfrak{X}_s(P)^{\G}  \oplus \mathfrak{X}_s(P)^{\G}_{\perp} \to \mathfrak{X}_s(P) \, , \qquad (X_1,X_2) \mapsto X_1 + X_2
\end{equation*}

\noindent
is an isomorphism of Hilbert spaces. The projective limit of $\left\{ f_s \right\}_{s> n+4}$ preserves continuity and bijectivity and therefore by the open mapping theorem for Fréchet spaces it defines a homeomorphism of tame Fréchet spaces:
\begin{equation*}
 f \colon \mathfrak{X}(P)^{\G}  \oplus \mathfrak{X}(P)^{\G}_{\perp} \to \mathfrak{X}(P)\, , \qquad \mathfrak{X}(P)^{\G}  = \bigcap_{s>n+4} \mathfrak{X}_s(P)^{\G}\,.
\end{equation*}

\noindent
Hence $\mathfrak{X}(P)^{\G}$ admits a topological complement in $\mathfrak{X}(P)$. Let $\cU \subset \mathfrak{X}(P)$ and $\phi \colon \cU \to \Diff(P)$ be a Fréchet chart of $\Diff(P)$ around the identity element. This chart is equivariant since $\bar{g}$ is $\G$-invariant. Shrinking $\cU$ if necessary, we can write $\cU = \cU_o \times \cV_o$, where $\cU_o \subset \mathfrak{X}(P)^{\G}$ and $\cV_o \subset \mathfrak{X}(P)^{\G}_{\perp}$ are neighborhoods of $0$ in $\mathfrak{X}(P)^{\G}$ and $\mathfrak{X}(P)^{\G}_{\perp}$, respectively. Equivariance of $\phi\colon \cU \to \Diff(P)$ amounts to the identities:
\begin{equation}
\label{eq:EquivarianceExp}
\mathrm{E}^{\bar{g}}_{po}(\dd_p\Psi_o(X_p))  = \Psi_o (\mathrm{E}^{\bar{g}}_{p}(X_p)) \, , \qquad \forall \,\, p\in P \quad \forall\,\, o\in \G \quad \forall\,\, X\in \mathfrak{X}(P) \,,
\end{equation}

\noindent
where $\mathrm{E}^{\bar{g}}\colon TP \to P$ denotes the exponential map associated to $\bar{g}$. The previous equation in turn implies that:
\begin{equation*}
\phi(\cU_o\times \left\{ 0\right\}) \subseteq \phi(\cU) \cap \Aut(P)\,.
\end{equation*}

\noindent
Conversely, if $X\in \cU$ such that $\phi(X)\in \phi(\cU) \cap \Aut(P)$ then we have:
\begin{equation*}
\Psi_o (\mathrm{E}^{\bar{g}}_{p}(X_p)) = \mathrm{E}^{\bar{g}}_{po}(X_{po})\,,
\end{equation*}

\noindent
which, using Equation \eqref{eq:EquivarianceExp}, gives $\dd_p\Psi_o(X_p) = X_{po}$ for every $p\in P$ and every $o\in\G$. Hence $X\in\mathfrak{X}(P)^{\G}$ and $\phi(\cU_o\times \left\{ 0\right\}) = \phi(\cU) \cap \Aut(P)$ whence $\Aut(P)$ is a tame Fréchet submanifold of $\Diff(P)$.

\noindent
It remains to show the closedness. Let now $\left\{u_k\right\}$ be a sequence in $\Aut(P)$ converging to an element $u\in \Diff(P)$. Since the convergence is uniform it implies \emph{pointwise} convergence. Therefore:
\begin{eqnarray*}
\lim_{k\to \infty} u_k(p) = u(p)
\end{eqnarray*}

\noindent
for every $p\in P$. Hence:
\begin{equation*}
u(po) = \lim_{k\to \infty} u_k(p o) = \lim_{k\to \infty} (u_k(p)o) = (\lim_{k\to \infty} u_k(p)) o= u(p) o
\end{equation*}

\noindent
for every $p\in P$ and every $o\in \G$. Therefore, $\Aut(P)$ is closed in $\Diff(P)$. The fact that $\cG(P)\subset \Aut(P)$ is a closed tame Fréchet Lie subgroup of $\Aut(P)$ can be proven similarly after noticing that $\Gamma(\frG_P)$ is closed inside $\Gamma(\cA_P)$ by using the splitting $\mathfrak{X}(P)^{\G} = \mathfrak{X}(M)\oplus \Gamma(\frG_P)$ with bracket \eqref{eq:standardAityah}. In fact, $\cG(P)$ is not only a projective limit of Hilbert manifolds but of Hilbert groups, since the Lie bracket \eqref{eq:standardAityah} restricted to $\Gamma(\frG_P)$ pointwise identifies with the bracket in $\frg$ and involves no \emph{loss of derivatives}. 
\end{proof}

\noindent
Note that, in particular, the connected components of the identity $\Aut^o(P)\subset \Aut(P)$ and $\Diff^o(M)\subset \Diff(M)$ are also tame Fréchet Lie groups. We have the following short exact sequence of tame Fréchet Lie groups: 
\begin{equation*}
1 \to \cG(P) \to \Aut(P) \to \Diff^{\prime}(M)\to 1 \,.
\end{equation*} 
 
\noindent
where $\Diff^{\prime}(M)$ denotes the group of diffeomorphisms that can be covered by elements in $\Aut(P)$ and that includes all those isotopic to the identity. Taking the differential of this sequence at the identity we obtain a corresponding short exact sequence of Fréchet Lie algebras:
\begin{equation}
\label{eq:shortexactLie}
0 \to \Gamma(\frG_P) \to \Gamma(\cA_P) = \mathfrak{X}(P)^{\G} \to \mathfrak{X}(M)\to 0 
\end{equation}

\noindent
fitting in the following commutative diagram with vertical arrows given by the exponential map in the Fréchet Lie group category:
\begin{equation}
\begin{tikzcd}
0 \ar[r] & \Gamma(\frG_P) \ar[d] \ar[r] & \Gamma(\cA_P) \ar[d] \ar[r] & \mathfrak{X}(M) \ar[d] \ar[r] & 0\\ 
1 \ar[r] & \cG(P) \ar[r] & \Aut(P) \ar[r] & \Diff^{\prime}(M) \ar[r] & 1
\end{tikzcd}
\end{equation}

\noindent
Recall that the exponential map $\Gamma(\frG_P) \to \cG(P)$ is fiberwise induced by the standard exponential map of $\G$. On the other hand, the short exact sequence \eqref{eq:shortexactLie} can be obtained by applying the global section functor to the Atiyah sequence of $P$, which therefore can be interpreted as encoding the infinitesimal symmetries of $P$. A choice of smooth connection $A$ on $P$ determines uniquely a smooth splitting of the sequence \eqref{eq:shortexactLie}:
\begin{equation*}
s_A\colon \mathfrak{X}(M) \to \mathfrak{X}(P)^{\G}\, ,
\end{equation*}

\noindent
which maps a vector field $v\in\mathfrak{X}(M)$ to its unique horizontal lifting as determined by the connection $A\in \Omega^1(P,\mathfrak{g})$. Such splitting determines an isomorphism of Lie algebras:
\begin{equation*}
\mathfrak{X}(P)^{\G} = \mathfrak{X}(M) \oplus \Gamma(\frG_P) 
\end{equation*}

\noindent
where $\mathfrak{X}(M) \oplus \Gamma(\frG_P)$ is equipped with the Lie bracket:
\begin{equation}
\label{eq:bracketXP}
[v_1 + \tau_1 , v_2 + \tau_2]_P =  [v_1,v_2] + \nabla^A_{v_1} \tau_2 - \nabla^A_{v_2}\tau_1 + [\cR^{A}(v_1,v_2),\,\cdot\,]_{\frG_P} + [\tau_1,\tau_2]_{\frG_P} 
\end{equation}

\noindent
as explained in Subsection \ref{sec:Atiyahalgebroid}. Given an element $u\in \Aut(P)$, its differential is a bundle map fitting into the following commutative diagram:
\begin{equation}
\begin{tikzcd}[column sep=1.75cm, row sep=1cm]
TP \ar[r, "\dd u"] \ar[d]
& TP \ar[d]
\\
P \ar[r, "u"]
& P
\end{tikzcd}
\end{equation}
 
\noindent
Since $u$ satisfies $\pi \circ u = f_u \circ \pi$ it follows that $\dd u\colon TP  \to TP$ maps the tangent bundle of the orbit $\pi(m)$ isomorphically to the tangent bundle of the orbit $\pi(u(m)) = f_u(\pi(m))$. Furthermore, since $u$ is equivariant, namely it satisfies $u\circ \Psi_o = \Psi_o\circ u $ for every $o\in\G$, it follows that $\dd u\colon TP\to TP$ descends to $\cA_P$ and defines a vector bundle isomorphism:
\begin{equation*}
\dd \underline{u} \colon \cA_P \to \cA_P\, , \qquad  [v] \mapsto [\dd u(v)]
\end{equation*}

\noindent
It can be shown that $\dd\underline{u} \colon \cA_P \to \cA_P$ preserves the bracket $[\cdot , \cdot]_{\cA_P}$ and the anchor map $\rho_P\colon \cA_P \to TM$ and therefore is by definition an automorphism of $(\cA_P ,[\cdot , \cdot]_{\cA_P}, \rho_P)$. We have a morphism of groups:
\begin{equation*}
\Aut(P) \to \Aut(\cA_P)\, , \qquad u \mapsto \dd\underline{u} 
\end{equation*}

\noindent
where $\Aut(\cA_P)$ denotes the automorphism group of $(\cA_P ,[\cdot , \cdot]_{\cA_P}, \rho_P)$. This map is in general neither surjective nor injective, illustrating the fact that the theory of connections on $P$ modulo isomorphism of principal bundles may not be equivalent to the theory of splittings of a transitive Lie algebroid modulo isomorphisms of transitive Lie algebroids. Here we will exclusively consider the principal bundle point of view, although the transitive Lie algebroid point of view can also be of interest (see also the analyses in~\cite{Collier, BunkShahbazi}).


\section{The Einstein-Yang-Mills system on a principal bundle}
\label{sec:EinsteinYM}


In this section, we consider the Einstein-Yang-Mills system on a principal bundle and study its action functional, linearization, and some basic examples. The conventions for the various curvature operators and linear operations occurring below are summarized in Appendix \ref{app:curvatureformulae}.


\subsection{Configuration space and action functional}
\label{subsec:confspaceMaxwell}


Let $P$ be a principal bundle over $M$ with structure group $\G$. Let $\mathfrak{c}$ be a non-degenerate bilinear pairing on $\frG_P$ induced by a positive definite adjoint-invariant inner product on $\frg$, which we denote for simplicity by the same symbol. The configuration space of the Einstein-Yang-Mills system is the following product space:
\begin{equation*}
\Conf(P) = \Met(M)\times \cC(P)\, ,
\end{equation*}

\noindent
where $\Met(M)$ denotes the convex cone of Riemannian metrics on $M$ and $\cC(P)$ denotes the affine space of connections on $P$, both considered as tame Fréchet manifolds. The  tame Fréchet manifold $\Met(M)$ is is locally modeled on $\Gamma(T^{\ast}M\odot T^{\ast}M)$ whereas the tame Fréchet manifold $\cC(P)$ is locally modeled on $\Omega^1(M,\frG_P)$. Hence, we consider $\Conf(P)$ as a product of tame Fréchet manifolds. Given a connection $A\in \cC(P)$, we denote by $F_A\in \Omega^2(M,\frG_P)$ its curvature. The tangent space $T_{(g,A)}\Conf(P)$ of $\Conf(P)$ at $(g,A) \in \Conf(P)$ is given by the following Fréchet space:
\begin{eqnarray*}
T_{(g,A)}\Conf(P) = \Gamma(T^{\ast}M\odot T^{\ast}M)\oplus \Omega^1(M,\frG_P)\, .
\end{eqnarray*}

\noindent
The Einstein-Yang-Mills system determined by $\mathfrak{c}$ on $P$ is defined as the system of partial differential equations obtained through the variational problem defined by the following action functional:
\begin{equation}
\label{eq:EYMFunctional}
\cS_{P, \mathfrak{c}} \colon \Conf(P) \to \mathbb{R}\, ,\quad (g,A)\mapsto \cS_{P, \mathfrak{c}} [g,A]
= \int_M \left( s^g + \kappa \vert F_A \vert^2_{g,\mathfrak{c}}\right) \nu_g\, ,
\end{equation}

\noindent
where $\kappa\in \left\{ -1, 1\right\}$ is a sign, $\nu_g$ denotes the Riemannian volume form associated to $g$ and $\vert {-}\vert^2_{g,\mathfrak{c}}$ denotes the norm induced by $g$ and $\mathfrak{c}$ on the bundle of polyforms taking values in the adjoint bundle $\frG_P$. The action $\cS_{P, \mathfrak{c}}$ is usually referred to as the \emph{Einstein-Yang-Mills functional} in the literature. We denote its associated Lagrangian density by:
\begin{equation*}
\cL_{P, \mathfrak{c}} \colon \Conf(P) \to \Omega^n (M)\, ,\quad (g,A)\mapsto \cL_{P, \mathfrak{c}} [g,A] = ( s^g + \kappa\vert F_A \vert^2_{g,\mathfrak{c}})\nu_g\, .
\end{equation*}

\noindent
which is a smooth map of Fréchet manifolds. The first term in the previous Lagrangian is usually referred to as the \emph{Einstein-Hilbert term} whereas the second term is usually referred to as the \emph{Yang-Mills term}.
 
\begin{remark}
In Lorentzian signature the relative sign of the Einstein-Hilbert and Yang-Mills terms in $\cL_{P, \mathfrak{c}}$ cannot be chosen at will and it is unambiguously determined to be negative by requiring positivity for the kinetic energy of the Yang-Mills \emph{field}. Therefore, for $\kappa = -1$ we can think of the Lagrangian $\cL_{P, \mathfrak{c}}$ as the restriction of the Lorentzian Einstein-Yang-Mills Lagrangian to direct product configurations on a manifold of the type $\mathbb{R}^k\times M$ equipped with the pullback bundle of a principal bundle over $M$. The choice of sign $\kappa = 1$ on the other hand does not seem to have a clear interpretation in relation to the physics origin of Einstein-Yang-Mills theory, although it has been preferred in the mathematical literature \cite{ApostolovMaschler,LeBrun1,LeBrun2,LeBrun3}. In the following, we will simultaneously consider both cases $\kappa \in \left\{ -1, 1\right\}$.
\end{remark}


\subsection{The Einstein-Yang-Mills equations of motion}
\label{subsec:eqsYM}


For further reference we introduce a linear operation $\circ_{g,\frc}$ on pairs of forms $\alpha_1 , \alpha_2 \in \Omega^k(M,\frG_P)$ taking values in $\frG_P$ as follows:
\begin{equation} \label{eq:circgc}
 \alpha_1\circ_{g,\frc} \alpha_2 \in \Gamma(T^{\ast}M\odot T^{\ast}M)\, , \quad 
 (\alpha_1\circ_{g,\frc} \alpha_2)(v_1,v_2) = \frac{1}{2} \Big(\langle \alpha_1(v_1), \alpha_2(v_2)\rangle_{g,\mathfrak{c}} + \langle \alpha_1(v_2), \alpha_2(v_1)\rangle_{g,\mathfrak{c}}\Big),
\end{equation}

\noindent
for every $v_1,v_2\in T^{\ast}M$, where $\langle \cdot , \cdot \rangle_{g,\frc}$ is the positive definite metric on $\wedge T^{\ast}M \otimes \frG_P$ induced naturally by $g$ and $\frc$ and whose associated norm was denoted by $\vert {-} \vert_{g,\frc}$ in Subsection \ref{subsec:confspaceMaxwell}. For standard differential forms not taking values in $\frG_P$ we introduce an analogous operation denoted simply by $\circ_g$. Similarly, for symmetric tensors $\tau_1 , \tau_2 \in \Gamma(T^{\ast}M\odot T^{\ast}M)$ we define:
\begin{equation*}
\tau_1\circ_{g} \tau_2 \in \Gamma(T^{\ast}M\odot T^{\ast}M)\, , \quad (\tau_1\circ_{g} \tau_2)(v_1,v_2) = \frac{1}{2} \Big( g(\tau_1(v_1), \tau_2(v_2)) +  g(\tau_1(v_2), \tau_2(v_1))\Big)
\end{equation*}

\noindent
The metric induced by $g$ on the tensor bundles of $M$ will be denoted again by $g$ for ease of notation.

\begin{lemma}
\label{lemma:YMequations}
A pair $(g,A) \in \Conf(P)$ is a critical point of $\cS_{P, \mathfrak{c}}$ if and only if it satisfies the following equations:
\begin{equation}
\label{eq:YMequations}
\begin{array}{l} \displaystyle
\Ric^g - \frac12 s^g\, g = \frac{\kappa}{2} \vert F_A \vert^2_{g,\mathfrak{c}}\,  g  - \kappa\, F_A\circ_{g,\frc} F_A\, , \\[10pt]
\dd_A^{g\ast} F_{A} = 0\, ,
\end{array}
\end{equation}

\noindent
where $\dd_A^{g\ast}\colon \Omega^2(M,\frG_P) \to \Omega^1(M,\frG_P)$ denotes the formal adjoint of the exterior covariant derivative determined by $A$.
\end{lemma}

\begin{proof}
The differential:
\begin{equation*}
\dd_{(g,A)} \cL_{P, \mathfrak{c}}\colon \Gamma(T^{\ast}M\odot T^{\ast}M)\oplus \Omega^1(M,\frG_P) \to C^{\infty}(M) 
\end{equation*}

\noindent
of $ \cL_{P, \mathfrak{c}} \colon \Conf(P) \to C^{\infty}(M)$ at the point $(g,A)\in \Conf(P)$ evaluated on the tangent vector $(h,a)\in \Gamma(T^{\ast}M\odot T^{\ast}M)\oplus \Omega^1(\frG_P)$ is given by the usual formula as follows:
\begin{equation*}
\dd_{(g,A)} \cL_{P, \mathfrak{c}}(h,a) = \frac{\dd}{\dd t}  \cL_{P, \mathfrak{c}}[g_t,A_t ]\vert_{t=0}\, ,
\end{equation*}

\noindent
where $t\mapsto (g_t,A_t )$ is a smooth curve in $\Conf(P)$ satisfying:
\begin{equation*}
(g_t,A_t )\vert_{t=0} = (g,A)\, , \qquad \frac{\dd}{\dd t}(g_t,A_t )\vert_{t=0} = (h,a)\, .
\end{equation*}

\noindent
By linearity in the arguments of $\dd_{(g,A)} \cL_{P, \mathfrak{c}}$, we can work separately on $h$ and $a$. Take first a vector $(0,a)$ and consider the curve $t\mapsto (g_t,A_t)=(g, A+t a)$. Then $F_{A_t}=F_A+t\,d_A a+t^2 a\wedge a$. Hence, $\frac{\dd }{\dd t}  F_{A_t}\vert_{t=0} = \dd_A a$ and so:
\begin{eqnarray*}
\frac{\dd}{\dd t} \cL_{P, \mathfrak{c}}[g,A_t ]\vert_{t=0} = 
\frac{\dd}{\dd t} \left( s^g + \kappa\langle F_{A_t},F_{A_t} \rangle_{g,\mathfrak{c}}\right)\nu_g\vert_{t=0}
= 2 \kappa\langle F_{A},d_{A} a \rangle_{g,\mathfrak{c}} \nu_g\, .
\end{eqnarray*}

\noindent
Therefore, integrating by parts we obtain: 
\begin{equation*}
 \dd_{(g,A)}\cS_{P, \mathfrak{c}}(0, a)= \int_M 2 \kappa\langle F_{A}, \dd_{A} a \rangle_{g,\mathfrak{c}} \nu_g
= \int_M 2 \kappa\langle \dd_A^{g\ast} F_{A}, a \rangle_{g,\mathfrak{c}} \nu_g\, ,
\end{equation*}

\noindent 
where $\dd_A^{g\ast}$ is the formal adjoint of the exterior covariant derivative $\dd_A\colon \Omega^{r}(M,\frG_P)\to \Omega^{r+1}(M,\frG_P)$ determined by $A$. Imposing $\dd_{(g,A)}\cS_{P, \mathfrak{c}}(0, a) = 0$ for every $a \in \Omega^1(M,\frG_P)$ we conclude that $(g,A)$ is a critical point of $\cS_{P, \mathfrak{c}}$ only if:
\begin{equation*}
\dd_A^{g\ast} F_{A} = 0\, .
\end{equation*}

\noindent
Now we take a vector $(h,0)$, given by the curve $t\mapsto (g+th,A)$ for $\vert t\vert < \varepsilon$ with $\varepsilon > 0$ small enough. By Lemma \ref{lemma:variationRs} we have:
\begin{equation*}
\frac{\dd }{\dd t}( s^{g_t})\vert_{t=0} = \Delta_g \tr_g(h) + \nabla^{g\ast}\nabla^{g\ast}h- g(h , \Ric^g)\, ,
\end{equation*}

\noindent
which, together with Equation \eqref{eq:variationdetg}, implies:
\begin{equation*}
\frac{\dd }{\dd t}( s^{g_t}\,\nu_{g_t})\vert_{t=0} = ( \Delta_g \tr_g(h) + \ddiv_g(\ddiv_g(h))- g(h , \Ric^g) + \frac12 \tr_g(h) s^g) \nu_g\, .
\end{equation*}

\noindent
Thus, we conclude: 
\begin{equation*}
\frac{\dd }{\dd t}\vert_{t=0} \int_M s^{g_t}\,\nu_{g_t} = \int_M g(h ,-\Ric^g+\frac12 s^g g) \nu_g \, ,
\end{equation*}

\noindent
since the integral of the Laplacian and the integral of the divergence vanish on a compact manifold. Regarding the derivative of the Yang-Mills term in the action functional, we compute:
\begin{equation*}
\frac{\dd }{\dd t} \vert F_A \vert^2_{g_t,\mathfrak{c}}\vert_{t=0}  = - g(h , F_A\circ_{g,\frc} F_A)\, , 
\end{equation*}

\noindent
where the operation $F_A\circ_{g,\frc} F_A \in \Gamma(T^{\ast}M\odot T^{\ast}M)$ is defined in 
(\ref{eq:circgc}). Altogether, we obtain: 
\begin{equation*}
\dd_{(g,A)}\cS_{P, \mathfrak{c}}(h,0)= \int_M g\Big(h , -\Ric^g+\frac12 \big( s^g+ \kappa \vert F_A \vert^2_{g,\mathfrak{c}}\big)  g   - \kappa\, (F_A\circ_{g,\frc} F_A) \Big) \nu_g\, .
\end{equation*}
 
\noindent
Therefore, by the discussion above, $(g,A)$ is a critical point of $\cS_{P, \mathfrak{c}}$ if and only if it satisfies the two equations:
\begin{equation*}
\dd_A^{g\ast} F_{A} = 0
\quad \text{and} \quad
\Ric^g - \frac12 s^g\, g = \frac{\kappa}{2} \vert F_A \vert^2_{g,\mathfrak{c}}\,  g  - \kappa  F_A\circ_{g,\frc} F_A
\end{equation*}

\noindent
simultaneously. 
\end{proof}

\noindent
Equations \eqref{eq:YMequations} define the \emph{Einstein-Yang-Mills system}. The first equation in \eqref{eq:YMequations} will be referred to as the \emph{Einstein equation}, whereas the second equation in \eqref{eq:YMequations} will be referred to as the \emph{Yang-Mills} equation. 

\begin{definition}
A pair $(g,A)\in \Conf(P)$ consisting of a Riemannian metric $g$ on $M$ and a connection $A$ on $P$ is an \emph{Einstein-Yang-Mills pair} if it solves the Einstein-Yang-Mills system \eqref{eq:YMequations}. 
\end{definition}

\noindent
If $(g,A)$ is an Einstein-Yang-Mills pair such that $F_A\neq 0$ we will say that $(g,A)$ is \emph{non-trivial}.

\begin{remark}
If $n=2$ then $2\, \mathrm{Ric}^g = s^g$ and the Einstein equation reduces to  $\vert F_A \vert^2_{g,\mathfrak{c}} = 0$. Therefore, every Einstein-Yang-Mills pair $(g,A)$ in this dimension consists of a flat metric on $T^2$ together with a flat connection on $P$. Consequently, in the following, we will assume $n>2$. 
\end{remark}

\noindent
For further reference we introduce the \emph{Einstein tensor} $\mathbb{G}^g$ of a Riemannian metric $g$ on $M$ as:
\begin{equation*}
\mathbb{G}^g = \mathrm{Ric}^g - \frac{1}{2} s^g g\,.
\end{equation*}

\noindent
Note that $\nabla^{g\ast}\mathbb{G}^g = 0$ for every $g\in \Met(M)$.


\subsection{Elementary properties of Einstein-Yang-Mills pairs}
\label{subsec:elementaryEinstein-Yang-Millspairs}


Taking the trace of the Einstein equation with respect to $g$ we obtain:
\begin{equation*}
s^g - \frac12 n s^g  = \frac{\kappa}{2} n \vert F_A \vert^2_{g,\mathfrak{c}}   - 2\kappa \, \vert F_A \vert^2_{g,\mathfrak{c}}  \, ,
\end{equation*}
using that $\mathrm{Tr}(F_A\circ_{g,\frc} F_A)= 2
\vert F_A \vert^2_{g,\mathfrak{c}}$.
Solving for $s^g$ we obtain the following prescription for the scalar curvature of every Einstein-Yang-Mills pair $(g,A)$:
\begin{equation*}
s^g = \frac{\kappa (n-4)}{(2-n)} \vert F_A \vert^2_{g,\mathfrak{c}}\, .
\end{equation*}

\noindent
Note that this equation already shows that the cases $\kappa =-1$ and $\kappa =1$ may be non-equivalent since $\kappa$ prescribes the sign of the scalar curvature and the latter may be obstructed in the positive curvature case \cite{Hitchin,Lichnerowicz}. Plugging the previous equation back into the Einstein equation, we obtain:
\begin{equation*}
\Ric^g = \kappa \left(\frac{1}{n-2}\vert F_A \vert^2_{g,\mathfrak{c}}\,  g  -F_A\circ_{g,\frc} F_A\right)\, .
\end{equation*}

\noindent
This form of the Einstein equation is sometimes more convenient from the analytic point of view since it only involves the Ricci curvature operator instead of both the Ricci and scalar curvature operators. 

\begin{definition}
\label{def:energymomentum}
The \emph{energy-momentum tensor} of the Einstein-Yang-Mills system is the following smooth map of Fréchet manifolds:
\begin{equation*}
\cT\colon \Conf(P) \to \Gamma(T^{\ast}M \odot T^{\ast}M )\, , \qquad (g,A)\mapsto \frac{\kappa}{2} \vert F_A \vert^2_{g,\mathfrak{c}}\,  g  - \kappa\, F_A\circ_{g,\frc} F_A\, ,
\end{equation*} 
	
\noindent
where $\frc$ is the given inner product on $\frG_P$.
\end{definition}

\noindent
In terms of the energy-momentum tensor, the Einstein equation of the Einstein-Yang-Mills system adopts the standard form:
\begin{equation*}
\mG^g = \cT(g,A)\, ,
\end{equation*}

\noindent
where $\mG^g\in \Gamma(T^{\ast}M\odot T^{\ast}M)$ is the Einstein tensor of $g$. 

\begin{lemma}
The following formula holds:
\begin{equation*}
\nabla^{g\ast}(F_A\circ_{g,\frc} F_A)(v) = \langle \dd_A^{g\ast} F_A , \iota_v F_A \rangle_{g,\frc} - \frac{1}{2}\iota_v \dd \vert F_A \vert^2_{g,\frc} 
\end{equation*}

\noindent
for every $(g,A)\in \Conf(P)$ and $v\in \mathfrak{X}(M)$.
\end{lemma}

\begin{proof}
Let $(e_1,\hdots , e_n)$ be a local orthonormal frame with dual local coframe $(e^1,\hdots , e^n)$ around a point $p\in M$, chosen so that $\nabla^g_{e_i} e_j (p)=0$. Given $v\in \mathfrak{X}(M)$ we compute at $p$:
\begin{align*}
-\nabla^{g\ast}(F_A\circ_{g,\frc} F_A)(v) &=   \sum_{i=1}^n\nabla^g_{e_i}(F_A\circ_{g,\frc} F_A) \, (e_i,v)  \\
&=  \sum_{i=1}^n \Big(\nabla^g_{e_i} \langle \iota_{e_i} F_A , \iota_v F_A \rangle_{g,\frc} - \langle \iota_{e_i} F_A , \iota_{\nabla^g_{e_i}v} F_A \rangle_{g,\frc}\Big)   \\
& = \sum_{i=1}^n \Big( \langle \nabla^{g,A}_{e_i} F_A , e^i\wedge \iota_v F_A \rangle_{g,\frc} + \langle  F_A , e^i\wedge\nabla^{g,A}_{e_i} ( \iota_v F_A) \rangle_{g,\frc}  - \langle \iota_{e_i} F_A , \iota_{\nabla^g_{e_i}v} F_A \rangle_{g,\frc}\Big) \\
& = \sum_{i=1}^n  \Big(\langle \iota_{e_i} \nabla^{g,A}_{e_i} F_A ,  \iota_v F_A \rangle_{g,\frc} +   \langle  F_A , e^i\wedge\iota_v\nabla^{g,A}_{e_i} F_A \rangle_{g,\frc}\Big) \\
&= -\langle \dd^{g\ast}_A F_A ,  \iota_v F_A \rangle_{g,\frc} + \frac{1}{2}\iota_v \dd \vert F_A \vert^2_{g,\frc}
\end{align*}

\noindent
where $\nabla^{g,A}$ denotes the tensor product of the Levi-Civita connection and the connection induced by $A$ on $\frG_P$ and we have used the Bianchi identity $\dd_A F_A =0$.
\end{proof}

\noindent
By the previous Lemma, we immediately obtain the following result for the energy-momentum tensor of the Einstein-Yang-Mills theory.
\begin{proposition}
\label{prop:DivergenceT}
The divergence of the energy-momentum tensor evaluated at $(g,A) \in \Conf(P)$ is given by:
\begin{equation*}
-\nabla^{g\ast}(\cT(g,A))(v) = \kappa \langle \dd_A^{g\ast} F_A , \iota_v F_A \rangle_{g,\frc} \, , \qquad v\in \mathfrak{X}(M)\, .
\end{equation*}

\noindent
In particular, $\nabla^{g\ast}(\cT(g,A)) = 0$ whenever $A$ is a Yang-Mills connection.
\end{proposition}

\noindent
Therefore, the Yang-Mills equation that occurs as part of the Einstein-Yang-Mills system guarantees that the energy-momentum tensor is divergence-free when evaluated on pairs $(g,A)$ with $A$ Yang-Mills, as expected by physical consistency. For further reference, we introduce the \emph{reversed energy-momentum tensor} as the following smooth map of tame Fréchet manifolds:
\begin{equation}
\label{eqn:reverse}
\hat{\cT}\colon \Conf(P) \to \Gamma(T^{\ast}M \odot T^{\ast}M )\, , \qquad (g,A)\mapsto \kappa \Big(\frac{1}{n-2}\vert F_A \vert^2_{g,\mathfrak{c}}\,  g  -F_A\circ_{g,\frc} F_A\Big)\, ,
\end{equation}

\noindent
which is sometimes more convenient than the standard energy-momentum tensor for computations.


\subsection{Examples of Einstein-Yang-Mills pairs}


Solutions to the Einstein-Yang-Mills equations on a compact manifold are hard to come by except for some notable exceptions. A particular class of Einstein-Yang-Mills pairs $(g,A)$ can be characterized by requiring that the Ricci tensor of $g$ and the reversed energy-momentum tensor of $(g,A)$ are both equally proportional to $g$, that is:
\begin{equation}
\label{eq:examplecondition}
\Ric^g = \Lambda g = \kappa \Big(\frac{1}{n-2}\vert F_A \vert^2_{g,\mathfrak{c}}\,  g  -F_A\circ_{g,\frc} F_A\Big)\, ,
\end{equation}

\noindent
for a constant $\Lambda \in \mathbb{R}$. In particular, $(M,g)$ is an Einstein manifold. Within the previous ansatz assume that $A$ is not flat. We then have the following distinguished cases:

\begin{itemize}
\item If $n=3$ then Equation \eqref{eq:examplecondition} reduces to:
\begin{equation*}
\kappa \Lambda g =  \vert F_A \vert^2_{g,\mathfrak{c}}\,  g  - F_A\circ_{g,\frc} F_A  = \ast_g F_A\otimes_{\frc}\ast_g F_A\, ,
\end{equation*}

\noindent
where $\ast_g F_A\otimes_{\frc}\ast_g F_A \in \Gamma(T^{\ast}M\odot T^{\ast}M)$ is defined by taking the usual tensor product together with the norm induced by $\frc$ on the adjoint bundle. Therefore, no solution exists in dimension three. In fact, when $n=3$ it is more natural to assume that $(M,g)$ is a metric contact manifold and $g$ is $\eta$-Einstein. This produces a natural ansatz for $\Ric^g$ better suited to fit with the structure of the reversed energy-momentum tensor and can be used to produce Einstein-Yang-Mills pairs, at least when $\G=\U(1)$.

\item If $n=4$ then the trace of \eqref{eq:examplecondition} implies $\Lambda = 0$ and therefore $g$ is Ricci-flat. Furthermore:
\begin{equation*}
F_A\circ_{g,\frc} F_A = \frac{1}{2}\vert F_A \vert^2_{g,\mathfrak{c}}\,  g \, .
\end{equation*}

\noindent
This equation is satisfied automatically by every \emph{instanton}, namely by every self-dual or anti-self-dual connection. Since instantons satisfy the Yang-Mills equation, every pair $(g,A)$ consisting of a Ricci-flat metric $g$ and an instanton $A$ defines an Einstein-Yang-Mills pair. In particular, every poly-stable holomorphic vector bundle over a K3 surface or complex torus canonically defines an associated Einstein-Yang-Mills pair. This will be our main case of study in Section \ref{sec:instantonK3}, where we consider the local moduli space of Einstein-Yang-Mills pairs around a stable holomorphic bundle over K3.

\item Assume $n>4$. By taking the trace of \eqref{eq:examplecondition} we conclude that $\Lambda$ is necessarily non-zero and its sign is opposite to $\kappa$. In this \emph{generic} dimension it is not evident that \eqref{eq:examplecondition} admits solutions. Nonetheless, examples can be found on symmetric spaces equipped with their normal homogeneous Riemannian metric and canonical homogeneous connection \cite{HarnadTafelShnider}.  
\end{itemize}


\subsection{Lift to the total space $P$}


Let $\Met(P)$ denote the Fréchet manifold of Riemannian metrics on $P$. For every Killing form $\frc$ on the Lie algebra $\frg$ of the structure group $\G$ of $P$ we have a canonical map:
\begin{equation*}
\Theta^{\frc} \colon \Conf(P) \to \Met(P)\, , 
\end{equation*}

\noindent
which, to every pair $(g,A) \in \Conf(P)$, associates the following Riemannian metric on $P$:
\begin{equation*}
\Theta^{\frc}_{g,A}(X_1,X_2) = g(\dd \pi (X_1) , \dd \pi (X_2)) + \frc (A(X_1) , A(X_2))\, , \quad X_1 , X_2 \in \mathfrak{X}(P)\, ,
\end{equation*}

\noindent
where $\dd \pi \colon TP\to TM$ is the differential of the principal bundle projection $\pi \colon P\to M$.

\begin{lemma}
For every Killing form $\frc$ on $\frg$ and every pair $(g,A) \in \Conf(P)$ the Riemannian metric is $\G$-invariant and satisfies:
\begin{equation*}
\Theta^{\frc}_{g,A}(D\Psi(x_1),D\Psi(x_2))\vert_p = \frc(x_1,x_2)\, ,
\end{equation*}

\noindent
for every $x_1, x_2 \in \frg$ and $p\in P$.
\end{lemma}
 
\begin{remark}
Conversely, every $\G$-invariant Riemannian metric on $P$ that satisfies the previous equation is of the form $\Theta^{\frc}_{g,A}$ for a certain unique $(g,A)\in \Conf(P)$. 
\end{remark}

\begin{proof}
The result follows directly from the standard properties satisfied by the principal connections of $P$, namely:
\begin{equation*}
\Psi^{\ast}_o A = \Ad_{o^{-1}}(A)\, , \quad A(D\Psi(x)) = x\, , \quad o\in \G\, , \,\, x\in \frg\, ,
\end{equation*}

\noindent
together with the invariance of $\frc$ under adjoint transformations.  
\end{proof}

\noindent
We will refer to elements in the image of $\Theta^{\frc}$ as $\frc$-\emph{adapted} Riemannian metrics on $P$. These metrics are sometimes called \emph{Kaluza-Klein metrics} in the literature \cite{Bourguignon}.

\begin{proposition}
\label{prop:ImThetaFrechet}
For every Killing form $\frc$ on $\frg$ the map $\Theta^{\frc} \colon \Conf(P) \to \Met(P)$ is smooth, tame and injective. Its image $\mathrm{Im}(\Theta^{\frc})\subset \Conf(P)$ is a smooth, closed, tame Fréchet submanifold of $\Met(P)$.
\end{proposition}
 
\begin{proof}
The fact that $\Theta^{\frc}$ is injective follows from the non-degeneracy of $\frc$ together with the fact that $\dd\pi$ is a linear isomorphism when restricted to the horizontal distribution determined by $A$. Since $\Theta^{\frc}$ is locally a polynomial with smooth coefficients in the components of $(g,A)$, it is clearly smooth. The fact that $\Theta^{\frc} \colon \Conf(P) \to \Met(P)$ is tame follows from the fact that partial differential operators between spaces of smooth sections are tame \cite{Subramanian}. The set of smooth metrics $\Met(P)$ is a smooth manifold modeled on the tame Fréchet space $\Gamma(T^{\ast}P\odot T^{\ast}P)$ and it is actually a contractible open subset of the latter. The vector space $\Gamma(T^{\ast}P\odot T^{\ast}P)^{\G}$ of all $\G$-invariant symmetric tensors on $P$ is a closed vector subspace of $\Gamma(T^{\ast}P\odot T^{\ast}P)$ and hence a tame Fréchet vector space itself. Consider the following continuous map:
\begin{equation*}
f\colon \Gamma(T^{\ast}P\odot T^{\ast}P)^{\G} \to \Gamma (V^{\ast}\odot V^{\ast})\, , \qquad H \mapsto H\vert_{\cV\times \cV} - \frc(\dd\Psi^{-1}_{\pi}(-) ,  \dd\Psi^{-1}_{\pi}(-))
\end{equation*}

\noindent 
where $V^{\ast}$ is the vector bundle dual to the vertical distribution $V$ in $TP$. We have:
\begin{equation*}
\mathrm{Im}(\Theta^{\frc}) = \Met(P) \cap f^{-1}(0)
\end{equation*}

\noindent
which implies, since $\Met(P)$ is an open subset of a tame Fréchet space and $f^{-1}(0)$ is a closed affine subspace, that $\mathrm{Im}(\Theta^{\frc})$ is a tame Fréchet manifold modeled on the tame Fréchet space of $\G$-invariant symmetric two-tensors on $P$ that vanish on $\cV\times \cV$.
\end{proof}

\noindent
Therefore, the map $\Theta^{\frc}$ gives a natural, smooth, and tame correspondence between $\frc$-adapted metrics on $P$ and elements of the configuration space of the Einstein-Yang-Mills system on $(P,\frc)$.


\section{A slice theorem for Einstein-Yang-Mills pairs}
\label{sec:SliceP}



\subsection{Preliminaries}


The Fréchet Lie group of automorphisms $\Aut(P)$ acts smoothly on the tame Fréchet manifold $\Conf(P)$ as follows:
\begin{equation*}
\Phi\colon \Conf(P)\times \Aut(P) \to \Conf(P)\, , \qquad (g,A,u) \mapsto (f^{\ast}_{u}g , u^{\ast}(A))
\end{equation*}

\noindent
where $f^{\ast}_{u}g$ denotes the pullback of $g$ along the diffeomorphism $f_u \colon M\to M$ covered by $u\in \Aut(P)$ and $ u^{\ast}(A)\in \Omega^1(P,\frg)$ is the pullback of $A$ by the diffeomorphism $u\colon P\to P$. Hence, we have that $\Phi$ is a smooth tame action.

\noindent
The goal of this section is to prove a slice theorem for the aforementioned action of $\Aut(P)$ on $ \Conf(P)$. In order to achieve this, the correspondence between elements $(g,A) \in \Conf(P)$ and $\frc$-adapted metrics on $P$ established by $\Theta^{\frc} \colon \Conf(P) \to \Met(P)$ will be particularly useful. For every $u\in \Aut(P)$ we have a commutative diagram of Fréchet manifolds:
\begin{equation}
\begin{tikzcd}[column sep=1.75cm, row sep=1cm]
\Conf(P)  \ar[r, "\Theta^{\frc}"] \ar[d, "\Phi_u"']
& \Met(P) \ar[d, "u^{\ast}"]
\\
\Conf(P) \ar[r, "\Theta^{\frc}"']
& \Met(P)
\end{tikzcd}
\end{equation}

\noindent
where:
\begin{equation*}
\Phi_u \colon \Conf(P) \to \Conf(P) \, , \qquad (g,A) \mapsto (f^{\ast}_u g, u^{\ast}A)\, ,
\end{equation*}

\noindent
and $u^{\ast}\colon \Met(P) \to \Met(P)$ denotes the action of the diffeomorphism $u\in \Diff(P)$ via pullback. Therefore, the smooth action $\Phi$ of $\Aut(P)$ on $\Conf(P)$ can be equivalently studied through the pullback action of $\Aut(P)$ on the metrics lying in the image of $\Theta^{\frc}$. In particular, $u^{\ast}\Theta^{\frc}_{g,A} \in 
\mathrm{Im}(\Theta^{\frc})$ for every $(g,A) \in \Conf(P)$ and $u\in \Aut(P)$.


\subsection{Isotropy subgroups in $\Aut(P)$ and $\cG(P)$}
\label{subsec:IsotropyP}


The \emph{isotropy group} $\cI_{(g,A)} \subset \Aut(P)$ of an element $(g,A)\in\Conf(P)$ is by definition its group of symmetries, namely:  
\begin{equation*}
\cI_{(g,A)} := \left\{ u\in \Aut(P) \,\,\, \vert\,\,\, (f^{\ast}_{u}g , u^{\ast}(A)) = (g , A) \right\}\, .
\end{equation*}

\noindent
For further reference, we denote by $\cI_{A} \subset \cG(P)$ the isotropy group of the connection $A$ on $P$, namely the subgroup of $\cG(P)$ that preserves $A$ under the natural action by pullback.

\begin{lemma}
Let $A$ be a connection on $P$ and fix a point $m\in M$. Then, the map:
\begin{equation*}
\cI_{A} \to C[\mathrm{Hol}_m(A),\Aut(P_m)]\, ,
\quad
u \mapsto u_m \coloneqq u_{|P_m}
\end{equation*}
	
\noindent
is an isomorphism of groups, where $C[\mathrm{Hol}_m(A),\Aut(P_m)]$ denotes the centralizer of the holonomy of $A$ at $m\in M$ inside the automorphism group $\Aut(P_m) \cong\G$ of the fiber $P_m$. 
\end{lemma}

\begin{proof}
We first show that the map is well-defined.
Consider the holonomy group $\mathrm{Hol}_m(A)$ of $A$ based at a fixed point $m\in M$. A gauge transformation $u\in \cG(P)$ preserves $A$ if and only if it preserves its associated horizontal distribution, namely if and only if it commutes with the parallel transport prescribed by $A$. Therefore, considering loops based at $m \in M$ it follows that $u_m\in \Aut(P_m)$ commutes with $\mathrm{Hol}_m(A)$ for every $u\in \cI_{A}$.
Next, we show that the map is injective. Since $M$ is connected, if $u\in\cI_{A}$ is trivial at $m\in M$ then $u$ acts trivially on $P$. In particular, the evaluation map:
\begin{equation*}
\cI_{A}\ni u \mapsto u_m \in \Aut(P_m)\, ,
\end{equation*}
	
\noindent
is an injective homomorphism. Consequently, the image of $\cI_{A}$ in $\Aut(P_m)$ commutes with $\mathrm{Hol}_m(A)$.
Finally, we check that the maps is also surjective. Every element $u_0 \in C[\mathrm{Hol}_m(A),\Aut(P_m)]$ can be extended by the parallel transport prescribed by $A$ to a unique gauge transformation $u\in \cG(P)$ such that $u_m = u_0$. This extension is independent of the paths used to connect any two given points precisely because $u_0$ belongs to $C[\mathrm{Hol}_m(A),\Aut(P_m)]$.
\end{proof}

\noindent
Given a fixed point $m\in M$, by the previous lemma we obtain, c.f.\ Equation \eqref{eq:shortexactLie}, that $\cI_{(g,A)}$ fits non-canonically into the following short exact sequence:
\begin{equation}
1 \to C[\mathrm{Hol}_m(A),\G] \to \cI_{(g,A)} \to \Iso(M,g)^{\prime}\to 1\, ,
\end{equation}

\noindent
where $\Iso(M,g)^{\prime}$ denotes the Lie subgroup of the isometry group of $(M,g)$ that can be covered by elements in $\cI_{(g,A)}$. Recall that, since both $C[\mathrm{Hol}_m(A),\G]$ and $\Iso(M,g)^{\prime}$ are finite-dimensional Lie groups it follows that $\cI_{(g,A)}$ is a finite-dimensional Lie group. Furthermore, $\Theta^{\frc} \colon \Conf(P) \to \Met(P)$ identifies $\cI_{(g,A)}$ with the intersection of $\Aut(P)$ and the isometry group of $\Theta^{\frc}_{(g,A)}$ in $\Diff(P)$. Since the latter is compact and $\Aut(P)$ is closed in $\Diff(P)$ we conclude that $\cI_{(g,A)}$ is a compact Lie subgroup of $\Aut(P)$. Using this, we obtain the following result.

\begin{proposition}\cite[Page 44]{Subramanian}
\label{prop:quotientorbit}
Let $(g,A)\in \Conf(P)$. The quotient $\cI_{(g,A)}\backslash\Aut(P)$ has a unique tame manifold structure such that:
\begin{itemize}
    \item The natural projection $\Upsilon\colon \Aut(P)\to \cI_{(g,A)}\backslash\Aut(P)$ satisfies $\Ker(\dd_u \Upsilon) = (\dd_e \mathrm{R}_u) (T_e \cI_{(g,A)})$, where $\mathrm{R}_u \colon \Aut(P)\to \Aut(P)$ denotes right multiplication. 
    
    \item The projection  $\Upsilon\colon \Aut(P)\to \cI_{(g,A)}\backslash\Aut(P)$ admits local smooth tame sections.
\end{itemize}

\noindent
Furthermore, a map $f\colon \cI_{(g,A)}\backslash\Aut(P) \to Y$, where $Y$ is a smooth tame Fréchet manifold is smooth tame if and only if $f\circ\Upsilon \colon \Aut(P) \to Y$ is smooth tame.
\end{proposition}
 
\noindent
Therefore, in the terminology of \cite{DiezRudolph} the subgroup $\cI_{(g,A)}\subset \Aut(P)$ is a principal Lie subgroup of $\Aut(P)$.


\subsection{The infinitesimal action and its adjoint map}


Given $(g,A)\in \Conf(P)$, we introduce the \emph{orbit map} of $(g,A)$ as the following smooth map of Fréchet manifolds:
\begin{equation*}
\Phi_{(g,A)}\colon\Aut(P) \to \Conf(P)\, , \qquad u \mapsto (f^{\ast}_u g , u^{\ast}(A))\, .	
\end{equation*}

\noindent
With this definition, the orbit $\cO_{(g,A)}\subset \Conf(P)$ of $\Aut(P)$ passing through $(g,A)\in \Conf(P)$ is simply given by $\cO_{(g,A)} = \mathrm{Im}(\Phi_{(g,A)})$. 
Recall from Proposition~\ref{prop:AutPFrechet} that the Lie algebra of $\Aut(P)$ is canonically identified with $\mathfrak{X}(P)^G$.

\begin{lemma}
The differential $\dd_{e}\Phi_{(g,A)}\colon \mathfrak{X}(P)^{\G} \to T_{(g,A)}\Conf(P)$ of $\Phi_{(g,A)}$ at the identity in $\Aut(P)$ is given by:
\begin{equation*}
\dd_{e}\Phi_{(g,A)}(v,\tau) = (\cL_v g , \dd_A\tau + \iota_v F_A)
\end{equation*}
	
\noindent
where we identify $\mathfrak{X}(P)^{\G}  = \mathfrak{X}(M) \oplus \Gamma(\frG_P)$ through the splitting of the Atiyah sequence \eqref{eq:Atiyahsequence} determined by the horizontal distribution associated to $A$. In particular, $T_{(g,A)}\cO_{(g,A)}\subset T_{(g,A)}\Conf(P)$ is a closed vector subspace of $T_{(g,A)}\Conf(P)$.
\end{lemma}

\begin{proof}
Let $(g,A)\in\Conf(P)$ and consider a smooth curve $u_t\in \Aut(P)$ such that $u_0 = \Id$ and:
\begin{equation*}
\frac{\dd}{\dd t} u_t \vert_{t=0} = X \in \mathfrak{aut}(P) = \mathfrak{X}(P)^{\G}\, .
\end{equation*}
	
\noindent
Using the connection $A$ we obtain, as explained in Section \ref{sec:Atiyahalgebroid}, a canonical isomorphism of vector bundles $\mathfrak{X}(P)^{\G} = \mathfrak{X}(M) \oplus \Gamma(\frG_P)$ identifying $X=v\oplus \tau$. Here $v$ is the unique vector field on $M$ lifting through $A$ to the horizontal projection of $X$, and $\tau$ is the vertical projection understood as a section of the adjoint bundle of $P$. Hence:
\begin{equation*}
\dd_{e}\Phi_{(g,A)}(v,\tau) = \frac{\dd}{\dd t}\Phi_{(g,A)}(u_t)\vert_{t=0} = (\cL_v g,\cL_X A)\, .
\end{equation*}
	
\noindent
Writing $X = X_H \oplus X_V$ in terms of its horizontal $X_H$ and vertical $X_V$ components with respect the decomposition of $TP$ determined by $A$ we compute:
\begin{equation*}
\cL_X A = \dd\iota_X A + \iota_X\dd A =\dd(A(X_V)) + \Omega_A(X) - \frac{1}{2} [A,A](X) = \dd\bar{\tau} + [A,\bar{\tau}] + \Omega_A(X)\, , 
\end{equation*}
	
\noindent
where $\Omega_A$ is the curvature of $A$ as a Lie algebra valued two-form on $P$ and $\bar{\tau} = A(X_V) = \dd \Psi^{-1}(X_V)$ is the $\frg$-valued function on $P$ corresponding to $\tau\in \Gamma(\frG_P)$. Projecting the previous equation to $M$ we conclude.
\end{proof}

\begin{remark}
By the previous Lemma, it follows that the Lie algebra $\mathfrak{i}_{(g,A)} $ of $\cI_{(g,A)}$ for a pair $(g,A) \in \Conf(P)$ can be given as follows:
\begin{eqnarray*}
\mathfrak{i}_{(g,A)} = \left\{ (v,\tau) \in \mathfrak{X}(M) \oplus \Gamma(\frG_P)\,\, \vert\,\, \cL_v g = 0\, , \,\, \dd_A\tau + \iota_v F_A = 0 \right\}
\end{eqnarray*}

\noindent
Condition $\dd_A\tau = - \iota_v F_A$ is reminiscent of the definition of momentum map in symplectic geometry, with the curvature $F_A$ playing the role of a symplectic form. This intriguing idea is proposed and explored in \cite{Elgood:2020svt,Elgood:2020nls}, \emph{crystalizing} in several physics applications. To the best of our knowledge, this proposal has not been studied in the mathematics literature.
\end{remark}

\noindent
In order to proceed further, we endow $\Conf(P)$ with the $L^2$ metric $[\![ \cdot,\cdot]\!]_{g,\frc}$ determined by $g$ and $\frc$, which in our situation is explicitly given at the point $(g,A)\in \Conf(P)$ by the following expression:
\begin{equation*}
[\![ (h,a),(h,a) ]\!]_{g,\frc}  := \int_M ( g(h,h) +   \vert a\vert_{g,\frc}^2 ) \nu_g\, ,
\end{equation*}

\noindent
for every $(h,a) \in T_{(g,A)}\Conf(P)$. This defines a \emph{weak} Riemannian metric on $\Conf(P)$, invariant under $\Aut(P)$ transformations. The symbol of $\dd_{e}\Phi_{(g,A)}\colon \mathfrak{X}(P)^{\G} \to T_{(g,A)}\Conf(P)$ is injective, which implies the following $L^2$ orthogonal decomposition with closed factors:
\begin{equation*}
T_{(g,A)}\Conf(P) = \mathrm{Im}(\dd_{e}\Phi_{(g,A)}) \oplus \Ker(\dd_{e}\Phi_{(g,A)})^{\ast} \, ,
\end{equation*}

\noindent
where $(\dd_{e}\Phi_{(g,A)})^{\ast}\colon T_{(g,A)}\Conf(P) \to \mathfrak{X}(P)^{\G}$ denotes the formal $L^2$-adjoint of $\dd_{e}\Phi_{(g,A)}$.

\begin{lemma}
\label{lemma:adjointdPhi}
The adjoint differential operator $(\dd_{e}\Phi_{(g,A)})^{\ast}\colon T_{(g,A)}\Conf(P) \to \mathfrak{X}(P)^{\G}$ on $M$ is given by:
\begin{equation}
\label{eq:adjointdPhi}
(\dd_{e}\Phi_{(g,A)})^{\ast}(h,a) = \big( 2 (\nabla^{g\ast}h)^{\sharp_g} - (a \lrcorner_g^{\frc} F_A)^{\sharp_g} , \dd^{g\ast}_A a \big)\, .
\end{equation}

\noindent
In particular, the orthogonal complement of $T_{(g,A)}\cO_{(g,A)}$ in $T_{(g,A)}\Conf(P)$ is given by:
\begin{equation*}
(T_{(g,A)}\cO_{(g,A)})^{\perp_g} = \left\{ (g,a) \in T_{(g,A)}\Conf(P) \,\, \vert \,\,  2 (\nabla^{g\ast}h)^{\sharp_g} = ( a \lrcorner_g^{\frc} F_A)^{\sharp_g} \, , \, \, \dd^{g\ast}_A a = 0\right\}\, .
\end{equation*}

\noindent
Here we have defined $a \lrcorner_g^{\frc} F_A \in \Omega^1(M)$ by $(a \lrcorner_g^{\frc} F_A)(v) = - \langle a , \iota_v F_A \rangle_{g,\frc}$ for every $v\in \mathfrak{X}(M)$. 
\end{lemma}

\begin{proof}
Let $(g,A)\in \Conf(P)$. Consider $(h,a)\in T_{(g,A)}\Conf(P)$ and $X\in \mathfrak{X}(P)^G$ in the presentation $X=v\oplus \tau$ determined by $A$, we compute:
\begin{align*}
 [\![ (\dd_{e}\Phi_{(g,A)})(v,\tau) , (h, a)  ]\!]_{g,\frc} &= [\![ (\cL_v g, \dd_A\tau + \iota_v F_A) , (h, a)  ]\!]_{g,\frc} \\ &= \int_M ( g(\cL_v g,h) +   \langle \dd_A\tau + \iota_v F_A , a\rangle_{g,\frc} ) \nu_g \\
& = \int_M ( g(v, 2\nabla^{g\ast} h) +   \langle  \tau, \dd^{g\ast}_A a\rangle_{g,\frc} + \langle  F_A , v\wedge a\rangle_{g,\frc} ) \nu_g  \\ &= \int_M ( g(v, 2\nabla^{g\ast} h) +   \langle  \tau, \dd^{g\ast}_A a\rangle_{g,\frc} -  \langle  a \lrcorner_g^{\frc} F_A , v \rangle_{g,\frc} ) \nu_g\\
& = \int_M ( g(v, 2\nabla^{g\ast} h -  a \lrcorner_g^{\frc} F_A) +   \langle  \tau, \dd^{g\ast}_A a\rangle_{g,\frc} ) \nu_g \\ &= [\![ (v,\tau) , (2\nabla^{g\ast} h -  a \lrcorner_g^{\frc} F_A, \dd^{g\ast}_A a)  ]\!]_{g,\frc}
\end{align*}

\noindent
and hence we conclude.
\end{proof}

\noindent
The tame Fréchet space $(T_{(g,A)}\cO_{(g,A)})^{\perp_g}\subset T_{(g,A)}\Conf(P)$ is a natural candidate of \emph{infinitesimal slice} for the action of $\Aut(P)$ on $\Conf(P)$. We will verify that this is indeed the case in the following subsection.


\subsection{The slice theorem}
\label{subsec:slicetheorem}


In this subsection, we prove a smooth local slice theorem for the action of $\Aut(P)$ on $\Conf(P)$ modelled on the tame Fréchet space $(T_{(g,A)}\cO_{(g,A)})^{\perp_g}$. We take \cite[Definition 2.2]{DiezRudolph} as our definition of a slice in the tame Fréchet category. The existence of a slice around every $(g,A) \in \Conf(P)$ can be obtained as an application of the general theorem proved in \cite[Theorem 3.28]{DiezRudolph} for tame Fréchet smooth actions. The hypothesis of \cite[Theorem 3.28]{DiezRudolph} as well as an explicit proof of the existence of a slice follows by using the celebrated slice theorem for $\Diff(P)$ on $\Met(P)$ in the Fréchet category \cite{CorroKordass,Ebin,Hamilton,Subramanian}, together with the Fréchet tame identification of $\Conf(P)$ with the image of $\Theta^{\frc}$ in $\Met(P)$. 

\begin{theorem}
\label{thm:slice}
Let $(g,A)\in \Conf(P)$. Then, there exists a slice $\mathbb{S}  \subset \Conf(P)$ around $(g,A)\in \Conf(P)$ which is the image of an equivariant diffeomorphism $\mathrm{E}\colon \cU \to \mathbb{S} $, where $\cU \subset (T_{(g,A)}\cO_{(g,A)})^{\perp_g}$ is an open neighbourhood of $0$ in $(T_{(g,A)}\cO_{(g,A)})^{\perp_g}$.
\end{theorem}

\begin{proof}
Recall that $\Diff(P)$ is a tame Fréchet Lie group acting tamely and properly on the tame Fréchet manifold $\Met(P)$, see \cite{Hamilton} and \cite[Page 68]{Subramanian}. Consider $\Aut(P) \subset \Diff(P)$ acting on $\mathrm{Im} (\Theta^{\frc}) \subset \Met(P)$ via diffeomorphisms. By Proposition \ref{prop:AutPFrechet}, $\Aut(P)$ is a closed tame Fréchet Lie subgroup of $\Diff(P)$ and by Proposition  \ref{prop:ImThetaFrechet}, $\mathrm{Im} (\Theta^{\frc})$ is a closed tame Fréchet submanifold of $\Met(P)$ preserved by $\Aut(P)$. Hence, the action of $\Aut(P)$ on $\mathrm{Im} (\Theta^{\frc})$ is also tame and proper. This proves point $(i)$ in \cite[Theorem 3.28]{DiezRudolph}. In particular, given $\bar{g}\in \mathrm{Im} (\Theta^{\frc})$, its orbit $\cO_{\bar{g}}$ under the action of $\Aut(P)$ is homeomorphic to $\cI_{(g,A)}\backslash \Aut(P)$, which by Proposition \ref{prop:quotientorbit} is a tame Fréchet manifold. Again by Proposition \ref{prop:quotientorbit}, we conclude that $\cI_{(g,A)}$ is a principal tame Fréchet Lie subgroup of $\Aut(P)$, which proves point $(ii)$ in \cite[Theorem 3.28]{DiezRudolph}. The fact that the differential of the orbit map $\Phi_{(g,A)}\colon \Aut(P)\to \Conf(P)$ has injective symbol, proven explicitly in Lemma \ref{lemma:deformationelliptic} below, implies that the action of $\Aut(P)$ on $\mathrm{Im} (\Theta^{\frc})$ is a tame regular map, which yields point $(iii)$ in \cite[Theorem 3.28]{DiezRudolph}. By \cite[Proposition 3.17]{DiezRudolph} this shows, in turn, that the orbit $\cO_{\bar{g}}$ is a smooth closed submanifold of $\Aut(P)$ and that the $L^2$ orthogonal complements of $T\cO_{\bar{g}}$ inside $T\mathrm{Im} (\Theta^{\frc})\vert_{\cO_{\bar{g}}}$ assemble into a smooth split normal subbundle $\cN\cO_{\bar{g}}$ of the latter. In particular, $\cN\cO_{\bar{g}}$ is a smooth tame subbundle of $T\mathrm{Im} (\Theta^{\frc})\vert_{\cO_{\bar{g}}}$ complementary to $T\cO_{\bar{g}}$. Furthermore, $\mathrm{Im} (\Theta^{\frc})$ inherits a smooth exponential map from that of $\Met(P)$, whose restriction to a neighbourhood of the zero section of the normal bundle $\cN\cO_{\bar{g}}$ is an equivariant local diffeomorphism onto its image. This shows that point $(iv)$ in \cite[Theorem 3.28]{DiezRudolph} holds since this exponential map is associated with the restriction to $\mathrm{Im} (\Theta^{\frc})$ of the Levi-Civita connection of the weak $L^2$ Riemannian metric on $\Met(P)$. Therefore, we conclude that the action of $\Aut(P)$ on $\mathrm{Im} (\Theta^{\frc})$ admits a slice $\bar{\mathbb{S}}_{\bar{g}}$ at every $\bar{g}\in \mathrm{Im} (\Theta^{\frc})$, which is realized as the restriction of the $L^2$ exponential map of $\mathrm{Im} (\Theta^{\frc})$ to the normal bundle $\cN\cO_{\bar{g}}$ of $T\cO_{\bar{g}}$ in $T\mathrm{Im} (\Theta^{\frc})$. Since $\Theta^{\frc}$ is smooth, tame, and equivariant, defining:
\begin{equation*}
\mathbb{S}_{\bar{g}} :=    (\Theta^{\frc})^{-1}(\bar{\mathbb{S}}_{\bar{g}})
\end{equation*}

\noindent
we obtain a slice for the action of $\Aut(P)$ on $\Conf(P)$ at the unique $(g,A)\in \Conf(P)$ such that $\bar{g} = \Theta^{\frc}_{(g,A)}$. Since $\Theta^{\frc}\colon \Conf(P) \to \Met(P)$ preserves orthogonality with respect to the corresponding $L^2$ metrics, the normal bundle $\cN\cO_{\bar{g}}$ of $\cO_{\bar{g}}$ corresponds to the $L^2$ normal bundle of $\cO_{(g,A)}$ in $\Conf(P)$ and hence we conclude.
\end{proof}

\begin{remark}
The slice $\mathbb{S} $ is tame Fréchet submanifold of $\Conf(P)$ since it is locally modeled on $(T_{(g,A)}\cO_{(g,A)})^{\perp_g}$, which, being a closed vector subspace of $T_{(g,A)}\Conf(P)$, is a tame Fréchet space.
\end{remark}

\begin{remark}
Theorem 3.28 \cite{DiezRudolph} implies not only the existence of a slice but of a \emph{linear slice} as defined in \cite[Definition 2.8]{DiezRudolph}. This is important in order to construct Kuranishi models of moduli spaces of geometric structures \cite{MeerssemanNicolau,DiezRudolphII}, since it identifies the slice with the quotient by the isotropy group of an open neighborhood of $0$ in the tangent space of the slice at the given base point.  
\end{remark}

\noindent
As a direct consequence of the slice theorem we obtain the following corollaries, c.f. \cite[Proposition 4.1]{CorroKordass} and \cite[Corollary 4.2]{CorroKordass}, where analog statements are proved in the case of Riemannian metrics on a compact manifold. The proofs in our case are formally identical and are therefore omitted.

\begin{corollary}
Let $(g,A)\in \Conf(P)$. There exists an open neighborhood  $\cU$ of the identity in $\Aut(P)$ and an open neighborhood $\cV$ of $(g,A)$ in $\Conf(P)$ such that for any $(g^{\prime},A^{\prime})\in \cV$ the symmetry group $\cI_{(g^{\prime},A^{\prime})}$ of $(g^{\prime},A^{\prime})$ is conjugate to a subgroup of $\cI_{(g,A)}$ via an element in $\cU$.
\end{corollary}

\begin{corollary}
The subset of elements of $\Conf(P)$ with a trivial symmetry group is open in $\Conf(P)$.
\end{corollary}

\noindent
Furthermore, the fact that the set of metrics with trivial isometry group and connections with trivial symmetry group is open and dense in $\Met(M)$ and $\cC(P)$, respectively, together with the fact that $\Conf(P)$ is the direct product of $\Met(M)$ and $\cC(P)$ equipped with the corresponding product Fréchet structures implies, in addition, the following result.

\begin{corollary}
The subset of elements of $\Conf(P)$ with a trivial symmetry group is open and dense in $\Conf(P)$.
\end{corollary}


\section{The local Kuranishi model}
\label{sec:Kuranishi}



\subsection{Preliminaries}


In this section, we construct the local Kuranishi moduli of Einstein-Yang-Mills pairs on $(P,\frc)$ modulo automorphisms in $\Aut(P)$. To do this, we consider the following smooth map of Fréchet manifolds defined using the Einstein-Yang-Mills system as follows:
\begin{equation} \label{eqn:E}
\begin{array}{rcl}
\cE := (\cE_1, \cE_2) \colon \Conf(P) &\to& \Gamma(T^{\ast}M\odot T^{\ast}M) \times \Omega^1(M,\frG_P)\\[5pt]
 (g,A) &\mapsto& \left(\cE_1(g,A) := -(\mG^g  - \cT(g,A))\, ,
 \cE_2(g,A) := 2\kappa\,\dd_A^{g\ast} F_{A} \right) 
\end{array}
\end{equation}

\noindent
where $\frG_P$ denotes the adjoint bundle of $P$, $\mG^g = \Ric^g - \frac{1}{2} s^g g$ denotes the Einstein tensor of $g$ and $\cT$ denotes the energy-momentum tensor introduced in Definition \ref{def:energymomentum}. The space of solutions to the Einstein-Yang-Mills system on $(P,\frc)$ is given by $\cE^{-1}(0)\subset \Conf(P)$, which we consider to be endowed with the subspace topology. The choice of signs in (\ref{eqn:E}) is needed later for the self-adjointness of the differential.

\begin{remark}
We could have chosen to define the map $\cE_1 \colon \Met(M) \to \Gamma(T^{\ast}M\odot T^{\ast}M)$ by using the Einstein equation written in terms of the Ricci curvature and the reverse energy-momentum tensor as $\mathrm{Ric}^g = \hat{\cT}(g,A)$. This form of the Einstein equation seems to be more convenient to study the linearization of the Einstein-Yang-Mills system since it only involves the Ricci curvature operator instead of the Ricci and scalar curvature operators that are contained in the Einstein tensor. However, the linearization problem in this form turns out to be less transparent since the associated deformation complex is not self-adjoint.  
\end{remark}

\noindent
The automorphism group $\Aut(P)$ acts on $\Conf(P)$ equivariantly with respect to $\cE$, that is, we have:
\begin{equation*}
\cE(f_u^{\ast} g , u^{\ast}A ) = f_u^{\ast} \cE(g,A)\, ,
\end{equation*}

\noindent
for every $u\in \Aut(P)$ and $(g,A) \in \Conf(P)$, whence $\Aut(P)$ preserves the solution space $\cE^{-1}(0)\subset \Conf(P)$. The quotient:
\begin{equation*}
\mathfrak{M}(P,\frc) =  \cE^{-1}(0)/\Aut(P)\, ,
\end{equation*}

\noindent
equipped with quotient topology is the moduli space of Einstein-Yang-Mills pairs. 

\begin{remark}
Note that if $[g,A]\in \mathfrak{M}(P,\frc)$ then $[\lambda g,A]\in \mathfrak{M}(P,\lambda\frc)$ for every real positive constant $\lambda >0$. This is in sharp contrast to the situation occurring for the moduli space of Einstein metrics, which is preserved by homotheties. 
\end{remark}

\noindent
For further reference, we also introduce the smooth map of Fréchet manifolds corresponding to the trace of the Einstein equation of the Einstein-Yang-Mills system:
\begin{align}
\label{eq:tracEinstein-Yang-Mills}
\s \colon \Conf(P) &\to \cC^{\infty}(M)\,,\\
\notag
(g,A) &\mapsto \s (g,A) := s^g - \frac{\kappa (n-4)}{(2-n)} \vert F_A \vert^2_{g,\mathfrak{c}}\, .
\end{align}

\noindent
Clearly we have $\cE^{-1}(0)\subset \s^{-1}(0)$ and $\mathfrak{M}(P,\frc)\subset \s^{-1}(0)/\Aut(P)$. We expect the latter to be \emph{infinite-dimensional}, exactly as it happens with the moduli space of constant scalar curvature metrics in more than two dimensions.


\subsection{Infinitesimal theory}
\label{subsec:infinitesimal}


Let $(g,A)\in \Conf(P)$. The differentials of $\cE_1 \colon \Conf(P) \to \Gamma(T^{\ast}M\odot T^{\ast}M)$ and $\cE_2 \colon \Conf(P) \to\Omega^1(M)$ evaluated at $(g,A)$ define linear maps of Fréchet spaces:
\begin{equation*}
\dd_{(g,A)}\cE_1\colon T_{(g,A)}\Conf(P)\to \Gamma(T^{\ast}M\odot T^{\ast}M)\, , \qquad \dd_{(g,A)}\cE_2\colon T_{(g,A)}\Conf(P)\to \Omega^1(M)\,,
\end{equation*}

\noindent
which we compute in the following lemma.

\begin{lemma}
\label{lemma:infinitesimalE1}
For every $(h,a) \in T_{(g,A)}\Conf(P)$ the following formulas hold: 
\begin{align*}
&\dd_{(g,A)}\cE_1(h,a)
\\
&= - \Big( \frac{1}{2} \Delta_L h - \delta^{g}\nabla^{g\ast} h - \frac{1}{2}\nabla^g \dd\mathrm{Tr}_g(h)  - \frac{1}{2} s^g h - \frac{1}{2} \big(\Delta_g \mathrm{Tr}_g (h) - g(\Ric^g , h) + \nabla^{g\ast}\nabla^{g\ast} h \big) g  \\ 
& \qquad - \frac{\kappa}{2}\vert F_A\vert^2_{g,\frc} h - \kappa F_A \circ_{h,\frc} F_A + 2\kappa F_A \circ_{g,\frc} \dd_A a - \kappa \big(\langle F_A , \dd_A a\rangle_{g,\frc} - \frac{1}{2} g(F_A\circ_{g,\frc} F_A , h)\big) g \Big)\,, \\
&\dd_{(g,A)}\cE_2(h,a) = 2\kappa \Big( \dd_A^{g\ast}\dd_A a - a\lrcorner_g^{\frg} F_A + h(\dd_A^{g\ast} F_A) - \frac{1}{2} \dd\mathrm{Tr}_g(h)\lrcorner_g  F_A + \dd^{g\ast}_A (F_A)^g_h \Big) \,,
\end{align*}

\noindent
where $\Delta^g_L$ denotes the Lichnerowicz Laplacian in the conventions introduced in Appendix \ref{app:curvatureformulae}.
\end{lemma}

\begin{remark}
The symbol $F_A\circ_{h,\frc} F_A \in \Gamma(T^{\ast} M \odot T^{\ast}M)$ denotes a symmetric tensor which is defined similarly to $F_A\circ_{g,\frc} F_A \in \Gamma(T^{\ast}M\odot T^{\ast}M)$ but in terms of $h\in \Gamma(T^{\ast}M\odot T^{\ast}M)$ instead of $g$. More explicitly, we define:
\begin{equation*}
(F_A\circ_{h,\frc} F_A)(v_1,v_2) := g(h, \frc(\iota_{v_1}F_A \otimes \iota_{v_2}F_A)),
\end{equation*}

\noindent
whence in local coordinates we have $(F_A\circ_{h,\frc} F_A)_{ij} = \sum_a (F^a_A)_{ik}\, (F^a_A)_{jl}\, h^{kl}$. On the other hand, the symbol $(\,\cdot\,)_h^g$ denotes the linear operation defined in Appendix \ref{app:curvatureformulae}, see \eqref{eq:lineagh}, trivially extended to differential forms taking values on a vector bundle.
\end{remark}

\begin{proof}
Let $g_t = g + t h$ and $A_t = A + t a$ be smooth curves of metrics and connections for $\vert t\vert$ small enough. Then, to first order in $t$, a direct computation shows that:
\begin{align*}
\kappa \cT(g_t,A_t) = & \kappa \cT(g,A) +   t \Big(\frac{1}{2}\vert F_A\vert^2_{g,\frc} h + F_A \circ_{h,\frc} F_A - 2 F_A \circ_{g,\frc} \dd_A a\Big) \\
& +   t \Big(\langle F_A , \dd_A a\rangle_{g,\frc} - \frac{1}{2} g(F_A\circ_{g,\frc} F_A , h)\Big)g   + o(t^2)\, .
\end{align*}

\noindent
The expression for the differential of $\cE_1 \colon \Conf(P) \to \Gamma(T^{\ast}M\odot T^{\ast}M)$ at $(g,A)$ follows now from the previous equation together with Lemma \ref{lemma:variationRs} and Corollary \ref{cor:variationEinstein}. To compute the differential of $\cE_2$ at $(g,A)$ we use that
$\dd_A^{g\ast} F_A= -\sum_{i=1}^n \iota_{e_i} \nabla_{e_i}^{g,A} F_A$ on an orthonormal basis $(e_1,\hdots,e_n)$, where $\nabla^{g,A}$ denotes the connection induced by $\nabla^g$ and $A$ on $\Omega^r(M,\frG_P)$. Then the variation of $\cE_2$ with respect to $A$ is: 
\begin{equation*}
\dd_{(g,A)}\cE_2(0,a) = 2\kappa \Big(\dd_A^{g\ast}\dd_A a - \sum_{i=1}^n [a(e_i) , \iota_{e_i} F_A]_{\frG_P} \Big) \in \Omega^1(M,\frG_P)
\end{equation*}

\noindent
where $[\cdot,\cdot]_{\frG_P}$ is the Lie bracket on the adjoint bundle $\frG_P$ of $P$. For ease of notation, we define:
\begin{equation*}
a\lrcorner_g^{\frg} F_A = \sum_{i=1}^n [a(e_i) , \iota_{e_i} F_A]_{\frG_P} \in \Omega^1(M,\frG_P)\, .
\end{equation*}

\noindent
To compute the variation of $\cE_2$ with respect to the metric we use the formula \eqref{eq:variationvolumeform} together with the identity $\dd_A^{g\ast}F_A = - (-1)^n \ast_g \dd_A \ast_g F_A$, obtaining:
\begin{align*}
&\dd_A^{(g + t h)\ast}F_A
\\
&=  \dd_A^{g\ast}F_A - (-1)^n t \Big( (\dd_g\star_{\dd_A \ast_g F_A}) (h)  +\ast_g \dd_A (\dd_g\star_{F_A}(h))\Big) + o(t^2)\\
& = \dd_A^{g\ast}F_A +  t \Big( \frac{1}{2} \mathrm{Tr}_g(h) \dd_A^{g\ast}F_A - (-1)^n(\ast_g (\dd_A \ast_g F_A)^g_h  + \ast_g \dd_A \Big(\frac{1}{2} \mathrm{Tr}_g(h) \ast_g F_A + \ast_g (F_A)^g_h)\Big)\Big) + o(t^2)\\
&  = \dd_A^{g\ast}F_A +  t \Big(  \mathrm{Tr}_g(h) \dd_A^{g\ast}F_A - (-1)^n\ast_g (\dd_A \ast_g F_A)^g_h  - \frac{1}{2} \dd\mathrm{Tr}_g(h)\lrcorner_g  F_A + \dd^{g\ast}_A (F_A)^g_h\Big) + o(t^2)\, .
\end{align*}

\noindent
Here we have used the identity:
\begin{equation*}
(-1)^n \ast_g ( \dd \mathrm{Tr}_g(h)\wedge \ast_g F_A) =  \dd\mathrm{Tr}_g(h)\lrcorner_g  F_A\, .
\end{equation*}

\noindent
Furthermore, we compute:
\begin{equation*}
(-1)^{n-1}\ast_g (\dd_A \ast_g F_A)^g_h = \ast_g (\ast_g^2\dd_A \ast_g F_A)^g_h = (-1)^{n-1} \ast_g (\ast_g \dd_A^{g\ast} F_A)^g_h = -  \tr_g(h) \dd_A^{g\ast} F_A +   h(\dd_A^{g\ast} F_A)\,,
\end{equation*}

\noindent
where we have used the identity $(-1)^n \ast_g (\ast_g u)_h^g = \tr_g(h) u - h(u)$ that holds for every one-form $u\in \Omega^1(M)$. Therefore, we obtain:
\begin{eqnarray*}
\dd_A^{(g + t h)\ast}F_A  = \dd_A^{g\ast}F_A +  t  \Big( h(\dd_A^{g\ast} F_A) - \frac{1}{2} \dd\mathrm{Tr}_g(h)\lrcorner_g  F_A + \dd^{g\ast}_A (F_A)^g_h \Big) + o(t^2)\, ,
\end{eqnarray*}
	
\noindent 
Altogether, we have:
\begin{eqnarray*}
\dd_{(g,A)}\cE_2(h,a) =	2\kappa  \Big( \dd_A^{g\ast}\dd_A a - a\lrcorner_g^{\frg} F_A + h(\dd_A^{g\ast} F_A) - \frac{1}{2} \dd\mathrm{Tr}_g(h)\lrcorner_g  F_A + \dd^{g\ast}_A (F_A)^g_h \Big)\,,
\end{eqnarray*}

\noindent
and therefore we conclude.
\end{proof}

\noindent
For ease of notation in the following we set $\dd_{(g,A)}\cE := (\dd_{(g,A)}\cE_1 ,\dd_{(g,A)}\cE_2 )$. At an Einstein-Yang-Mills pair $(g,A) \in \cE^{-1}(0)$ the differential of $\cE$ simplifies as follows.

\begin{proposition}
\label{prop:infinitesimalE2}
Let $(g,A)$ be an Einstein-Yang-Mills pair. For every $(h,a) \in T_{(g,A)}\Conf(P)$ the following formulas hold:
\begin{align*}
\dd_{(g,A)}\cE_1(h,a) & = - \Big( \frac{1}{2} \nabla^{g\ast} \nabla^g h   - \mathrm{R}^g_o (h) - \delta^{g}\nabla^{g\ast} h - \frac{1}{2}\nabla^g \dd\mathrm{Tr}_g(h)  - \frac{1}{2} \Delta_g \mathrm{Tr}_g (h) g      -\frac{1}{2} \big(\nabla^{g\ast}\nabla^{g\ast} h\big)\, g  \\ 
& - \kappa\,  h\circ_g (F_A\circ_{g,\frc} F_A) + \frac{\kappa \mathrm{Tr}_g(h)}{2(n-2)}\vert F_A \vert^2_{g,\mathfrak{c}} g  - \kappa F_A \circ_{h,\frc} F_A + 2\kappa F_A \circ_{g,\frc} \dd_A a - \kappa \langle F_A , \dd_A a\rangle_{g,\frc} g \Big)\,,\\
\dd_{(g,A)}\cE_2(h,a) & = 2\kappa\Big(  \dd_A^{g\ast}\dd_A a - a\lrcorner_g^{\frg} F_A - \frac{1}{2} \dd\mathrm{Tr}_g(h)\lrcorner_g  F_A + \dd^{g\ast}_A (F_A)^g_h \Big) .
\end{align*}
\end{proposition}

\begin{proof}
Since by assumption $(g,A)$ is an Einstein-Yang-Mills pair it satisfies:
\begin{equation}
\label{eq:RicciT}
\Ric^g  = \kappa \Big(\frac{1}{n-2}\vert F_A \vert^2_{g,\mathfrak{c}}\,  g  -F_A\circ_{g,\frc} F_A\Big)\, , \qquad s^g = \frac{\kappa (n-4)}{(2-n)} \vert F_A \vert^2_{g,\mathfrak{c}}\,,
\end{equation}

\noindent
which we use to compute:
\begin{align*}
h\circ_g \Ric^g  &= \kappa \Big(\frac{1}{n-2}\vert F_A \vert^2_{g,\mathfrak{c}}\, h  - h\circ_g (F_A\circ_{g,\frc} F_A)  \Big)\, , \quad \text{and}
\\
g(\Ric^g , h) &= \frac{\kappa }{n-2}\vert F_A \vert^2_{g,\mathfrak{c}}\,\mathrm{Tr}_g(h)  - \kappa \, g(F_A\circ_{g,\frc} F_A, h)  \,.
\end{align*}

\noindent
Hence, we obtain that:
\begin{equation*}
\frac{1}{2} \Delta_L^g h = \frac{1}{2} \nabla^{g\ast} \nabla^g h + \frac{\kappa}{n-2} \vert F_A \vert^2_{g,\frc} h - \kappa\,  h\circ_g (F_A\circ_{g,\frc} F_A) - \mathrm{R}^g_o (h)\,.
\end{equation*}

\noindent
Inserting these equations, together with the Yang-Mills equation $\dd^{g\ast}_A F_A =0$, into Lemma \ref{lemma:infinitesimalE1} and rearranging terms, the claim follows.
\end{proof}

\noindent
Associated to the smooth maps $\cE\colon \Conf(P) \to \Conf(P)$ and $\Phi_{(g,A)}\colon\Aut(P) \to \Conf(P)$ we obtain the deformation complex given in the following result. 

\begin{proposition}
For every $(g,A)\in \Conf(P)$ the following equation holds:
\begin{equation*}
(\dd_e \Phi_{(g,A)})^{\ast}(\cE(g,A))=0\, .
\end{equation*}

\noindent
Furthermore, if $(g,A)$ is an Einstein-Yang-Mills pair then the following sequence:
\begin{equation}
\label{eq:deformationcomplex}
0 \to \Gamma(\cA_P) \xrightarrow{\dd_e\Phi_{(g,A)}} T_{(g,A)}\Conf(P) \xrightarrow{\dd_{(g,A)}\cE} T_{(g,A)}\Conf(P) \xrightarrow{(\dd_e\Phi_{(g,A)})^{\ast}} \Gamma(\cA_P)\to 0
\end{equation}

\noindent
is a complex.
\end{proposition}
 
\begin{proof}
Let $u_t\in \Aut(P)$ a smooth curve such that $u_0 = e$ and:
\begin{equation*}
\frac{\dd u_t}{\dd t}\vert_{t=0} = v\oplus \tau \in \mathfrak{X}(M) \oplus \Gamma(\frG_P)\, .
\end{equation*}

\noindent 
Since the Einstein-Yang-Mills functional is invariant under $\Aut(P)$-transformations we have:
\begin{equation*}
0 = \frac{\dd \cS_{P, \mathfrak{c}} (f^{\ast}_{u_t} g,u^{\ast}_t A)}{\dd t} \vert_{t=0} = \dd_{(g,A)}\cS_{P, \mathfrak{c}} (\cL_v g, \dd_A\tau + \iota_v F_A)\,,
\end{equation*}

\noindent
for every $(g,A)\in \Conf(P)$. Using the explicit expression for $\dd_{(g,A)}\cS_{P, \mathfrak{c}}\colon T_{(g,A)}\Conf(P) \to \mathbb{R}$ computed in the proof of Lemma \ref{lemma:YMequations}, this implies the following equation:
\begin{align*}
 &\dd_{(g,A)}\cS_{P, \mathfrak{c}}(\cL_v g, \dd_A\tau + \iota_v F_A)
 \\
 & = \int_M \langle (\cL_v g, \dd_A\tau + \iota_v F_A)  , ( s^g+ \kappa \vert F_A \vert^2_{g,\mathfrak{c}}) \frac{g}{2} -\Ric^g  - \kappa\, (F_A\circ_{g,\frc} F_A) , 2 \kappa \dd_A^{g\ast} F_{A}\rangle_{g,\mathfrak{c}}  \nu_g\\
& = \int_M \langle (v,  \tau  )  , (\dd_{e}\Phi_{(g,A)})^{\ast}(( s^g+ \kappa \vert F_A \vert^2_{g,\mathfrak{c}})  \frac{g}{2} -\Ric^g  - \kappa\, (F_A\circ_{g,\frc} F_A) , 2 \kappa \dd_A^{g\ast} F_{A}\rangle_{g,\mathfrak{c}}  \nu_g\,,
\end{align*}

\noindent
for every $(v,\tau) \in \mathfrak{X}(M) \oplus \Gamma(\frG_P)$. Hence:
\begin{equation*}
(\dd_{e}\Phi_{(g,A)})^{\ast}\Big(-\Ric^g+\frac12 ( s^g+ \kappa \vert F_A \vert^2_{g,\mathfrak{c}})  g   - \kappa\, (F_A\circ_{g,\frc} F_A) , 2 \kappa \dd_A^{g\ast} F_{A}\Big) = 0\, .
\end{equation*}

\noindent
The map $(\dd_{e}\Phi_{(g,A)})^{\ast}\colon T_{(g,A)}\Conf(P) \to \Gamma(\cA_P)$ was computed in \eqref{eq:adjointdPhi} and can be used to explicitly verify the previous identity, which in turn is equivalent to the following equation:
\begin{eqnarray*}
& (\dd_{e}\Phi_{(g,A)})^{\ast}(\cE_1(g,A) , \cE_2(g,A)) = 0.
\end{eqnarray*}

\noindent
Hence, we conclude that $(\dd_{e}\Phi_{(g,A)})^{\ast}(\cE(g,A))=0$ for every pair $(g,A)$, as claimed. Assuming now that $(g,A)\in \cE^{-1}(0)$ is an Einstein-Yang-Mills pair and differentiating the previous equation, we obtain:
\begin{eqnarray*}
(\dd_{e}\Phi_{(g,A)})^{\ast}  (\dd_{(g,A)}\cE_1 (h,a) , \dd_{(g,A)}\cE_2(h,a)) = 0\,,
\end{eqnarray*}

\noindent
for every $(h,a)\in T_{(g,A)}\Conf(P)$ and every Einstein-Yang-Mills pair $(g,A)$. This proves the fact that $\mathrm{Im}(\dd_{(g,A)}\cE)\subseteq \mathrm{Ker}((\dd_{e}\Phi_{(g,A)})^{\ast})$. That $\mathrm{Im}(\dd_e\Phi_{(g,A)})\subseteq \mathrm{Ker}(\dd_{(g,A)}\cE)$ follows by differentiating the identity:
\begin{equation*}
\cE\circ\Phi_{(g,A)}(u_t) = \cE(f_{u_t}^{\ast}g,u^{\ast}_t A) = 0\,,
\end{equation*}

\noindent
for every smooth family $u_t\in \Aut(P)$ such that $u_0 = e$ is the identity.
\end{proof}

\noindent
Elements $(h,a)\in T_{(g,A)}\Conf(P)$ in the kernel of $\dd_{(g,A)}\cE$ are the so-called \emph{infinitesimal deformations} of the Einstein-Yang-Mills pair $(g,A)$. We consider the complex \eqref{eq:deformationcomplex} as the deformation complex of Einstein-Yang-Mills pairs and we introduce its associated cohomology groups:
\begin{equation*}
\mathbb{H}^0_{(g,A)} = \mathrm{Ker}(\dd_e\Phi_{(g,A)})\, , \quad \mathbb{H}^1_{(g,A)} = \frac{\mathrm{Ker}(\dd_{(g,A)}\cE)}{\mathrm{Im}(\dd_e\Phi_{(g,A)})}\, , \quad \mathbb{H}^2_{(g,A)} = \frac{\mathrm{Ker}((\dd_{e}\Phi_{(g,A)})^{\ast})}{\mathrm{Im}(\dd_{(g,A)}\cE)}\, , \quad \mathbb{H}^3_{(g,A)} = \frac{\Gamma(\cA_P)}{\mathrm{Im}((\dd_e\Phi_{(g,A)})^{\ast})}
\end{equation*}

\noindent
Elements in the vector space $\mathbb{H}^0_{g,A}$ correspond to the \emph{infinitesimal symmetries} of $(g,A)$, whereas elements in the vector space $\mathbb{H}^1_{g,A}$ correspond to the \emph{infinitesimal deformations} of $(g,A)$ modulo the infinitesimal action of $\Aut(P)$. On the other hand, $\mathbb{H}^2_{g,A}$ corresponds to the \emph{obstruction space} for the infinitesimal deformations of $(g,A)$, as will become apparent in subsection \ref{subsec:Kuranishitheory}.  

\begin{lemma}
\label{lemma:deformationelliptic}
The complex \eqref{eq:deformationcomplex} is elliptic. 
\end{lemma}

\begin{proof}
The symbol complex associated with \eqref{eq:deformationcomplex}  is of the form:
\begin{equation}
\label{eq:symbolcomplex}
0\to \cA_{P,m} \xrightarrow{\sigma(\dd_e\Phi_{(g,A)})} T^{\ast}_m M^{\odot 2}  \oplus (T^{\ast}_m M\otimes \frG_{P,m}) \xrightarrow{\sigma(\dd_{(g,A)}\cE)} T^{\ast}_m M^{\odot 2} \oplus (T^{\ast}_m M\otimes \frG_{P,m}) \xrightarrow{\sigma((\dd_e\Phi_{(g,A)})^{\ast})} \cA_{P,m}\to 0\,,
\end{equation}

\noindent
where $\sigma(-)$ denotes the principal symbol of the corresponding differential operator at the point $(m,\xi)$ with $m\in M$ and $\xi\in T^{\ast}_m M - \left\{ 0 \right\}$. By the standard theory of elliptic operators and complexes, the complex \eqref{eq:deformationcomplex} is elliptic if and only if \eqref{eq:symbolcomplex} is exact for every $m\in M$ and $\xi\in T^{\ast}_m M - \left\{ 0 \right\}$. A direct calculation gives:
\begin{align*}
\sigma_{m,\xi}(\dd_e\Phi_{(g,A)}) (v,\tau) = & (v\otimes \xi + \xi\otimes v , \xi\otimes\tau)\,,\\
\sigma_{m,\xi}(\dd_{(g,A)}\cE) (h,a) = & 
\big(-\big(- \vert\xi\vert^2_g h + \xi\otimes h(\xi)+ h(\xi)\otimes \xi   - \mathrm{Tr}_g (h)\,\xi\otimes \xi + (\vert\xi\vert^2_g \mathrm{Tr}_g(h)  - h(\xi,\xi))g \big), \\
& \quad 2\kappa \big(-2 \vert\xi\vert_g^2 \, a + 2 \xi \otimes \iota_{\xi}a\big)\big)\,, \\
\sigma_{m,\xi}((\dd_e\Phi_{(g,A)})^{\ast})(h,a) = & (- 2 h(\xi) , - \iota_{\xi}a)\,,
\end{align*}

\noindent
where we momentarily consider $(h,a)\in (T^{\ast}_m M\odot T^{\ast}_m M)\oplus (T^{\ast}_m M\otimes \frG_{P,m})$. A quick inspection of the previous expressions shows that the symbol complex \eqref{eq:symbolcomplex} splits as a direct sum of the following complexes:
\begin{equation}
\label{eq:exact1}
0\to T_m M \xrightarrow{\sigma^1_{m,\xi}(\dd_e\Phi_{(g,A)})} T^{\ast}_m M\odot T^{\ast}_m M \xrightarrow{\sigma^1_{m,\xi}(\dd_{(g,A)}\cE)} T^{\ast}_m M\odot T^{\ast}_m M \xrightarrow{\sigma^1_{m,\xi}((\dd_{e}\Phi_{(g,A)})^{\ast})} T_m M \to 0\\
\end{equation}
and
\begin{equation}
\label{eq:exact2}
0 \to \frG_{P,m} \xrightarrow{\sigma^2_{m,\xi}(\dd_e\Phi_{(g,A)})}  T^{\ast}_m M\otimes \frG_{P,m} \xrightarrow{\sigma^2_{m,\xi}(\dd_{(g,A)}\cE)} T^{\ast}_m M\otimes \frG_{P,m} \xrightarrow{\sigma^2_{m,\xi}((\dd_{e}\Phi_{(g,A)})^{\ast})} \frG_{P,m}  \to 0\,,
\end{equation} 

\noindent
where:
\begin{align*}
\sigma^1_{m,\xi}(\dd_e\Phi_{(g,A)}) (v) &= v\otimes \xi + \xi\otimes v \, ,\\
\sigma^2_{m,\xi}(\dd_e\Phi_{(g,A)}) (\tau) &= \xi\otimes\tau\,,\\
\sigma^1_{m,\xi}(\dd_{(g,A)}\cE) (h) &= 
-\big(- \vert\xi\vert^2_g h + \xi\otimes h(\xi)+ h(\xi)\otimes \xi   - \mathrm{Tr}_g (h)\,\xi\otimes \xi + (\vert\xi\vert^2_g \mathrm{Tr}_g(h)  - h(\xi,\xi))g \big)\,,\\
\sigma^2_{m,\xi}(\dd_{(g,A)}\cE) (a) &= 2\kappa \big(  -2 \vert\xi\vert_g^2 \, a + 2 \xi \otimes \iota_{\xi}a \big)\,,\\
\sigma^1_{m,\xi}((\dd_{e}\Phi_{(g,A)})^{\ast}) (h) &= - 2 h(\xi)   \, ,\\
\sigma^2_{m,\xi}((\dd_{e}\Phi_{(g,A)})^{\ast}) (a) &=   - \iota_{\xi}a\,.
\end{align*}

\noindent
The complex \eqref{eq:symbolcomplex} is exact if and only if the complexes \eqref{eq:exact1} and \eqref{eq:exact2} are exact. We begin with~\eqref{eq:exact2}. Note that the only solution to $\xi\otimes \tau = 0$ is $\tau = 0$ whence $\sigma^2_{m,\xi}(\dd_e\Phi_{(g,A)})$ is injective. An element $a\in T^{\ast}_m M\otimes \frG_{P,m}$ lies in the kernel of $\sigma^2_{m,\xi}(\dd_{(g,A)}\cE)$ if and only if:
\begin{equation*}
a = \frac{1}{\vert\xi\vert_g^2}\xi \otimes \iota_{\xi}a\,,
\end{equation*}

\noindent
namely if and only if $a$ is equal to its projection along $\xi$. Hence, $a\in \Ker(\sigma^2_{m,\xi}(\dd_{(g,A)}\cE))$ if and only if there exists an element $\tau\in \frG_{P,m}$ such that $a = \xi\otimes \tau$. Therefore:
\begin{equation*}
\mathrm{Im}(\sigma^2_{m,\xi}(\dd_e\Phi_{(g,A)}) ) = \mathrm{Ker}(\sigma^2_{m,\xi}(\dd_{(g,A)}\cE_{(g,A)}))\,.
\end{equation*}

\noindent
On the other hand, the kernel of $\sigma^2_{m,\xi}((\dd_{e}\Phi_{(g,A)})^{\ast})$ consists of all elements $a\in T^{\ast}_m M\otimes \frG_{P,m}$ such that $\iota_{\xi}a = 0$. Equivalently, $\Ker(\sigma^2_{m,\xi}((\dd_{e}\Phi_{(g,A)})^{\ast}))$ is given by all elements in $T^{\ast}_m M\otimes \frG_{P,m}$ whose projection along $\xi$ vanishes. Hence, every element $x\in\Ker(\sigma^2_{m,\xi}((\dd_{e}\Phi_{(g,A)})^{\ast}))$  can be obtained from an appropriate element $a\in T^{\ast}_m M\otimes \frG_{P,m}$ by subtracting its projection along $\xi$, namely:
\begin{equation*}
x = a - \frac{1}{\vert\xi\vert_g^2}\xi \otimes \iota_{\xi}a = - \frac{1}{2\vert\xi\vert_g^2} ( - 2 \vert\xi\vert_g^2 a + 2\xi\otimes \iota_{\xi}a) = \sigma^2_{m,\xi}(\dd_{(g,A)}\cE) \Big(- \frac{1}{2\vert\xi\vert_g^2} a\Big).
\end{equation*}

\noindent
This implies that $\mathrm{Im}(\sigma^2_{m,\xi}(\dd_{(g,A)}\cE) = \Ker(\sigma^2_{m,\xi}((\dd_{e}\Phi_{(g,A)})^{\ast}))$, and since $\sigma^2_{m,\xi}((\dd_{e}\Phi_{(g,A)})^{\ast})$ is clearly surjective, we conclude that the sequence \eqref{eq:exact2} is exact.

\noindent
Now we consider the sequence \eqref{eq:exact1}. The equation $v\otimes \xi + \xi \otimes v = 0$ has the only solution $v=0$, hence $\sigma^1_{m,\xi}(\dd_{(g,A)}\cE)$ is injective. Suppose that $h\in \mathrm{Ker}(\sigma^1_{m,\xi}(\dd_{(g,A)}\cE))$. Then:
\begin{equation*}
 h = \frac{1}{\vert\xi\vert^2_g} \big(\xi\otimes h(\xi) + h(\xi) \otimes \xi - \mathrm{Tr}_g(h)\, \xi\otimes \xi +  (\vert\xi\vert^2_g \mathrm{Tr}_g(h)  - h(\xi,\xi))g \big)\,.
\end{equation*}

\noindent
Choose an orthonormal frame of the form $(e_1, \hdots, e_{n-1}, \xi/\vert\xi\vert_g)$, in terms of which we write:
\begin{equation}
\label{eq:hexpression}
h = \frac{1}{\vert\xi\vert^4_g} h(\xi,\xi)\, \xi\otimes\xi + \frac{1}{\vert\xi\vert^2_g}  h(\xi,e_i)\, ( \xi\otimes e_i + e_i \otimes \xi) + \sum^{n-1}_{i,j}  h(e_1,e_j)\, e_i\otimes e_j\,.
\end{equation}

\noindent
Plugging this expression into the previous equation and combining all terms we obtain $h(e_i,e_j) = 0$ and therefore:
\begin{equation*}
h = \frac{1}{\vert\xi\vert^4_g} h(\xi,\xi)\, \xi\otimes\xi + \frac{1}{\vert\xi\vert^2_g}  h(\xi,e_i) \, ( \xi\otimes e_i +  e_i \otimes \xi) = v\otimes \xi + \xi \otimes v\,,
\end{equation*}

\noindent
where:
\begin{equation*}
v =\frac{1}{2\vert\xi\vert^4_g} h(\xi,\xi)\, \xi + \frac{1}{\vert\xi\vert^2_g} h(\xi,e_i) \, e_i \,.
\end{equation*}

\noindent
Hence, $\mathrm{Im}(\sigma^1_{m,\xi}(\dd_e\Phi_{(g,A)}))=\mathrm{Ker}(\sigma^1_{m,\xi}(\dd_{(g,A)}\cE))$. Suppose now that $h\in \Ker(\sigma^1_{m,\xi}((\dd_{e}\Phi_{(g,A)})^{\ast}))$, that is, $h$ satisfies:
\begin{equation*}
 h(\xi) = 0\, .
\end{equation*}

\noindent
Plugging the previous equation into \eqref{eq:hexpression} we obtain:
\begin{equation*}
h =   \sum_{i,j}^{n-1} h(e_i,e_j)\, e_i\otimes e_j =  - \frac{1}{\vert\xi\vert^2_g} \sigma^1_{m,\xi}(\dd_{(g,A)}\cE)\Big(\frac{1}{n-2} \sum_k h(e_k,e_k) \sum_{j}  e_j\otimes e_j  -   \sum_{i,j}^{n-1} h(e_i,e_j)\, e_i\otimes e_j  \Big) \, .
\end{equation*}


\noindent
Therefore, $\mathrm{Im}(\sigma^1_{m,\xi}(\dd_{(g,A)}\cE)) = \Ker(\sigma_{m,\xi}((\dd_{e}\Phi_{(g,A)})^{\ast}))$ and the sequence \eqref{eq:exact1} is exact. Since the map:
\begin{equation*}
\sigma^1_{m,\xi}((\dd_{e}\Phi_{(g,A)})^{\ast})\colon T^{\ast}_m M\odot T^{\ast}_m M \to T_m M 
\end{equation*}

\noindent
is clearly surjective we conclude.
\end{proof}

\noindent
Following N. Koiso, who considered the deformation problem of Einstein metrics \cite{KoisoI,KoisoII,KoisoIII}, we introduce the following terminology.

\begin{definition}
Given an Einstein-Yang-Mills pair $(g,A)$, an element $(h,a) \in T_{(g,A)}\Conf(P)$ is an \emph{essential deformation} of $(g,A)$ if $(h,a)\in \Ker(\dd_{(g,A)}\cE)\cap \Ker((\dd_{e}\Phi_{(g,A)})^{\ast})$.
\end{definition}

\noindent
Intuitively speaking, the essential deformations of an Einstein-Yang-Mills $(g,A)$ pair correspond precisely those infinitesimal deformations of $(g,A)$ that cannot be eliminated or \emph{gauged away} through an \emph{infinitesimal gauge transformation}, namely through the infinitesimal action of $\Aut^o(P)$ on $\Conf(P)$.  Applying standard Hodge theory of elliptic complexes to \eqref{eq:deformationcomplex} we immediately obtain the following orthogonal decompositions with respect to the $L^2$ metric determined by $g$ and $\frc$:
\begin{align}
\Gamma(\cA_P) &= \Ker(\Delta^{(0)}_{(g,A)}) \oplus \mathrm{Im}((\dd_e\Phi_{(g,A)})^{\ast})\, ,\\
\Gamma(\cA_P) &= \Ker(\Delta^{(3)}_{(g,A)}) \oplus \mathrm{Im}((\dd_{e}\Phi_{(g,A)})^{\ast})\,,\nonumber\\
T_{(g,A)}\Conf(P) &= \Ker(\Delta^{(1)}_{(g,A)}) \oplus \mathrm{Im}(\dd_e\Phi_{(g,A)}) \oplus \mathrm{Im}((\dd_{(g,A)}\cE)^{\ast})\,, \label{eq:decompositionTConf}\\
T_{(g,A)}\Conf(P) &= \Ker(\Delta^{(2)}_{(g,A)}) \oplus \mathrm{Im}(\dd_{(g,A)}\cE) \oplus \mathrm{Im}(\dd_{e}\Phi_{(g,A)})\,, \nonumber  
\end{align}

\noindent
where $\Delta^{(i)}_{(g,A)}$, $i=0,\hdots, 3$ are the Laplacians of the elliptic complex \eqref{eq:deformationcomplex}, which are explicitly given by:
\begin{align*}
\Delta^{(0)}_{(g,A)} &= (\dd_e\Phi_{(g,A)})^{\ast}\circ \dd_e\Phi_{(g,A)}\, , \\
\Delta^{(1)}_{(g,A)} &= (\dd_{(g,A)}\cE)^{\ast}\circ\dd_{(g,A)}\cE + \dd_e\Phi_{(g,A)} \circ (\dd_e\Phi_{(g,A)})^{\ast}\,,\\
\Delta^{(2)}_{(g,A)} &= \dd_{e}\Phi_{(g,A)} \circ (\dd_{e}\Phi_{(g,A)})^{\ast} + \dd_{(g,A)}\cE \circ(\dd_{(g,A)}\cE)^{\ast}\ \, ,\\
\Delta^{(3)}_{(g,A)} &= (\dd_{e}\Phi_{(g,A)})^{\ast} \circ \dd_{e}\Phi_{(g,A)}\, .
\end{align*}

\noindent
From these decompositions, we obtain the natural isomorphisms:
\begin{equation}
\mathbb{H}^i_{(g,A)} = \Ker(\Delta^{(i)}_{(g,A)})\, , \qquad i=0,\hdots,3\,,
\end{equation}

\noindent
which identify the cohomology groups of the elliptic complex \eqref{eq:deformationcomplex} with the kernels of its associated Laplacians. In particular:
\begin{equation}
\mathbb{H}^1_{(g,A)} = \Ker(\dd_{(g,A)}\cE)\cap \Ker((\dd_{e}\Phi_{(g,A)})^{\ast})
\end{equation}

\noindent
is identified with the vector space of essential deformations of $(g,A)$, in agreement with the expectation that $\mathbb{H}^1_{(g,A)}$ encodes the \emph{non-trivial} infinitesimal deformations of $(g,A)$. Hence, we will refer to $\mathbb{H}^1_{(g,A)}$ as the vector space of essential deformations of the Einstein-Yang-Mills pairs $(g,A)$.

\begin{proposition}
\label{prop:dEselfadjoint}
Let $(g,A)$ be an Einstein-Yang-Mills pair. Then, the differential operator
\begin{equation*}
\dd_{(g,A)}{\cE} := (\dd_{(g,A)}\cE_1 ,  \dd_{(g,A)}\cE_2) \colon T_{(g,A)}\Conf(P)\to T_{(g,A)}\Conf(P) 
\end{equation*}

\noindent
is formally self-adjoint with respect to the $L^2$-metric induced by $g$ and $\frc$. 
\end{proposition}
 
\begin{proof}
The formal adjoint $\dd_{(g,A)}{\cE}^{\ast}$ of $\dd_{(g,A)}{\cE}$ is determined by the usual relation:
\begin{equation*}
\int_M \langle \dd_{(g,A)}{\cE}(h,a), (\hat{h},\hat{a}) \rangle_{g,\frc} \nu_g = \int_M \langle (h,a), \dd_{(g,A)} {\cE}^{\ast}(\hat{h},\hat{a}) \rangle_{g,\frc} \nu_g
\end{equation*}

\noindent
holds for every $(h,a),(\hat{h},\hat{a}) \in T_{(g,A)}\Conf(P)$. We consider the terms in $\dd_{(g,A)} {\cE}(h,a)$ that are not manifestly self-adjoint. First of all, using the compatibility of the structure group $\G$ with the inner product $\frc$ we compute:
\begin{align*}
\int_M \langle a\lrcorner_g^{\frg} F_A , \hat{a}\rangle_{g,\frc} \nu_g & = \int_M \langle (\alpha_{\Lambda} \otimes \tau_{\Lambda}) \lrcorner_g^{\frg} (F_{\Sigma} \otimes \tau_{\Sigma} ), \hat{\alpha}_{\Gamma}\otimes \tau_{\Gamma} \rangle_{g,\frc} \nu_g = \int_M \langle \alpha_{\Lambda} \lrcorner_g F_{\Sigma} , \hat{\alpha}_{\Gamma} \rangle_{g}  \frc([\tau_{\Lambda},\tau_{\Sigma}]_{\frG_P},\tau_{\Gamma}) \nu_g\\
& = - \int_M \langle  \hat{\alpha}_{\Gamma}\lrcorner_g F_{\Sigma} , \alpha_{\Lambda}  \rangle_{g} \frc(\tau_{\Lambda} , [\tau_{\Sigma},\tau_{\Gamma}]_{\frG_P}) \nu_g = \int_M \langle \hat{a}\lrcorner_g^{\frg} F_A , a\rangle_{g,\frc} \nu_g
\end{align*}

\noindent
whence this term does in fact provide a self-adjoint contribution to $\dd_{(g,A)}{\cE}(h,a)$. Here $(\tau_{\Lambda})$ denotes a local frame of $\frG_P$, where $\Lambda, \Sigma, \Gamma = 1, \hdots , \dim(\mathfrak{g})$ are Lie algebra indices, and we have written locally:
\begin{equation*}
a = \sum_{\Lambda} \alpha_{\Lambda} \otimes \tau_{\Lambda}
\end{equation*}

\noindent
For the remaining terms of $\dd_{(g,A)}{\cE}^{\ast}$ of $\dd_{(g,A)} {\cE}$ that are not evidently self-adjoint, we compute as follows:
\begin{align*}
\int_M g(\nabla^g\dd \mathrm{Tr}_g(h), \hat{h})\nu_g & = \int_M g(\dd \mathrm{Tr}_g(h), \nabla^{g\ast} \hat{h}) \nu_g  = \int_M g(h,  (\nabla^{g\ast}\nabla^{g\ast} \hat{h})\, g) \nu_g \,, \\
\int_M g((\nabla^{g\ast}\nabla^{g\ast}h)\, g, \hat{h})\nu_g & =   \int_M g(h,\nabla^g\dd \mathrm{Tr}_g(\hat{h}))\nu_g \,,\\
\int_M g(\langle F_A , \dd_A a\rangle_{g,\frc} g , \hat{h}) \nu_g & = \int_M  \langle \mathrm{Tr}_g(\hat{h}) F_A , \dd_A a\rangle_{g,\frc} \nu_g = - \int_M  \langle \dd\mathrm{Tr}_g(\hat{h})\lrcorner_g F_A , a\rangle_{g,\frc} \nu_g\,, \\
\int_M \langle \dd\mathrm{Tr}_g(h)\lrcorner_g  F_A , \hat{a}\rangle_{g,\frc} \nu_g &  = - \int_M \langle \dd_A^{g\ast}(\mathrm{Tr}_g(h)   F_A) , \hat{a}\rangle_{g,\frc} \nu_g  =   - \int_M g(\langle  F_A , \dd_A \hat{a}\rangle_{g,\frc} g , h)\nu_g\,, \\ 
\int_M g(F_A \circ_{g,\frc} \dd_A a , \hat{h}) \nu_g & = \sum_{i,j} \int_M (F_A \circ_{g,\frc} \dd_A a) (e_i,e_j) \hat{h}(e_i,e_j) \nu_g = \sum_{i,j} \int_M \langle F_A(e_i) , \dd_A a (e_j)\rangle_{g,\frc} \hat{h}(e_i,e_j) \nu_g \\
& = \sum_{i,j} \int_M \langle \dd_A^{g\ast} ( \hat{h}(e_i,e_j) \, e_j\wedge F_A(e_i) ),   a  \rangle_{g,\frc} \nu_g = - \int_M \langle \dd^{g\ast}_A (F_A)^g_{\hat{h}} , a\rangle_{g,\frc} \nu_g \,,\\
\int_M \langle \dd^{g\ast}_A (F_A)^g_h , \hat{a}\rangle_{g,\frc} \nu_g & = \frac{1}{2} \sum_{i,j}\int_M  \frc((F_A)^g_h(e_i,e_j) , \dd^{g}_A \hat{a}(e_i,e_j)) \nu_g \\
&=   \sum_{i,j}\int_M  \frc(g(h(e_i) , F_A(e_j)) , \dd^{g}_A \hat{a}(e_i,e_j)) \nu_g\\
& = - \int_M g(F_A \circ_{g,\frc} \dd_A \hat{a} , h) \nu_g \,.
\end{align*}

\noindent
The terms in the previous equation are adjoint of each other when considered with the factors with which they occur in $\dd_{(g,A)}{\cE}$, and hence we conclude. Note that we have used the following identity:
\begin{equation*}
(F_A)^g_h(e_i,e_j) = g(h(e_i) , F_A(e_j)) - g(h(e_j) , F_A(e_i))
\end{equation*}

\noindent
which follows from \eqref{eq:lineagh}.  
\end{proof}

\begin{theorem}
\label{thm:infinitesimalfinite}
For every Einstein-Yang-Mills pair $(g,A)$ the vector space of essential deformations $\mathbb{H}^1_{(g,A)}$ is finite-dimensional and isomorphic to the vector space of obstructions $\mathbb{H}^2_{(g,A)}$.
 \end{theorem}

\begin{proof}
The fact that $\mathbb{H}^1_{(g,A)}$ is finite-dimensional follows from the ellipticity of the complex \eqref{eq:symbolcomplex}, which is proven in Lemma \ref{lemma:deformationelliptic}. By Proposition \ref{prop:dEselfadjoint}, the differential operator $\dd_{(g,A)} {\cE} \colon T_{(g,A)}\Conf(P)\to T_{(g,A)}\Conf(P)$ is formally self-adjoint, and therefore we obtain:
\begin{eqnarray*}
& T_{(g,A)}\Conf(P) = \Ker(\Delta^{(1)}_{(g,A)}) \oplus \mathrm{Im}(\dd_{(g,A)}{\cE}) \oplus \mathrm{Im}(\dd_e\Phi_{(g,A)}) \\
& T_{(g,A)}\Conf(P) = \Ker(\Delta^{(2)}_{(g,A)}) \oplus \mathrm{Im}(\dd_{(g,A)}{\cE}) \oplus \mathrm{Im}(\dd_{e}\Phi_{(g,A)})  
\end{eqnarray*}

\noindent
Hence, the identity map on $T_{(g,A)}\Conf(P)$ induces a canonical isomorphism between $\Ker(\Delta^{(1)}_{(g,A)})$ and $\Ker(\Delta^{(2)}_{(g,A)})$ which immediately implies that $\mathbb{H}^1_{(g,A)}$ and $\mathbb{H}^2_{(g,A)}$ are isomorphic as finite-dimensional vector spaces.
\end{proof}

\noindent
By the previous result, if the moduli space of Einstein-Yang-Mills pairs around an Einstein-Yang-Mills pair $(g,A)$ is to be of positive dimension then it is necessarily obstructed. That is, the condition that $\mathbb{H}^2_{(g,A)} = 0$ trivializes the moduli problem of Einstein-Yang-Mills pairs, exactly as it happens for the moduli problems of Einstein metrics or Yang-Mills connections \cite{KoisoI,KoisoV}. We proceed now to examine in more detail the conditions for a pair $(h,a)$ to be an essential deformation.

\begin{lemma}
\label{lemma:infinitesimaltrace}
Let $(h,a) \in \Ker(\dd_{(g,A)}\cE) \subset T_{(g,A)}\Conf(P)$ be an infinitesimal deformation of $(g,A)$. Then, the following equations hold, depending on the dimension $n$:
\begin{align*}
&  \Delta_g \tr(h) +\nabla^{g\ast}\nabla^{g\ast}h + \kappa\, g( (F_A \circ_{g,\frc} F_A)^o , h^o ) = 0\, , \qquad \text{if }\ n=4 \,, \\
& \Delta_g \tr(h) +\nabla^{g\ast}\nabla^{g\ast} h +  \frac{s^g}{n} \tr_g(h) = 2\kappa \frac{n-4}{2-n} \langle \dd_A a , F_A \rangle_{g,\frc} + \frac{2\kappa}{2-n} g( (F_A \circ_{g,\frc} F_A)^o , h^o )\, , \qquad \text{if }\ n\neq 4 \,,
\end{align*}
	
\noindent
where $(F_A \circ_{g,\frc} F_A)^o$ and $h^o$ denote the trace-less projections of $F_A \circ_{g,\frc}, F_A$ and $h$, respectively.
\end{lemma}

\begin{proof}
Since by assumption $(h,a) \in \Ker(\dd_{(g,A)}\cE) \subset T_{(g,A)}\Conf(P)$ and $(g,A)$ is Einstein-Yang-Mills, $(g,A)$ and $(h,a)$ satisfy:
\begin{eqnarray*}
\s(g,A) = 0\, , \qquad (h,a) \in \Ker(\dd_{(g,A)}\s)
\end{eqnarray*}
	
\noindent
where $\s \colon \Conf(P) \to \cC^{\infty}(M)$ is defined in \eqref{eq:tracEinstein-Yang-Mills}. We compute:
\begin{eqnarray*}
\dd_{(g,A)}\s  (h,a) = \Delta_g \tr_g(h) + \nabla^{g\ast}\nabla^{g\ast} h- g(h,\Ric^g) + \kappa \frac{4-n}{2-n} \big(2 \langle \dd_A a, F_A \rangle_{g,\frc} - g(F_A\circ_{g,\frc} F_A,h)\big).
\end{eqnarray*}
	
\noindent
Substituting the first equation in \eqref{eq:RicciT} into the previous expression and manipulating, we obtain:
\begin{align*}
\dd_{(g,A)}\s (h,a) = & \, \Delta_g \tr_g(h) + \nabla^{g\ast}\nabla^{g\ast} h - \frac{\kappa}{n-2} \vert F_A\vert^2_{g,\frc} \tr_g(h) \\
& + \frac{2\kappa}{n-2} g(F_A\circ_{g,\frc}F_A , h) + 2\kappa \frac{n-4}{n-2} \langle \dd_A a , F_A\rangle_{g,\frc}
\end{align*}
	
\noindent
Writing:
\begin{equation*}
h = h^o + \frac{1}{n} \tr_g(h) g\, , \qquad F_A\circ_{g,\frc}F_A = (F_A\circ_{g,\frc}F_A)^o + \frac{2}{n} \vert F_A\vert^2_{g,\frc} g 
\end{equation*}

\noindent
we obtain:
\begin{align*}
\dd_{(g,A)}\s (h,a) = & \, \Delta_g \tr_g(h) + \nabla^{g\ast}\nabla^{g\ast} h + \kappa \frac{4-n}{n(n-2)} \vert F_A\vert^2_{g,\frc} \tr_g(h) \\
& + \frac{2\kappa}{n-2} g((F_A\circ_{g,\frc}F_A)^o , h^o) + 2\kappa \frac{n-4}{n-2} \langle \dd_A a , F_A\rangle_{g,\frc}
\end{align*}
	
\noindent
Hence, if $n=4$, we obtain the first equation in the statement of the Lemma. On the other hand, if $n\neq 4$, then the equation $\s(g,A) = 0$ implies:
\begin{eqnarray*}
\vert F_A \vert^2_{g,\mathfrak{c}} = \kappa \frac{2-n}{n-4}s^g\, .
\end{eqnarray*}
	
\noindent
Substituting this equation into the previous expression we obtain the second equation in the statement of the Lemma.
\end{proof} 

\noindent
We now consider in more detail the essential deformations of a given Einstein-Yang-Mills pair.
\begin{lemma}
\label{lemma:nablanablah}
The following formula holds:
\begin{eqnarray*}
\nabla^{g\ast}\nabla^{g\ast} h = \frac{1}{2} \langle \dd_A a , F_A \rangle_{g,\frc} 
\end{eqnarray*}
	
\noindent
for every pair $(h,a)\in \Ker((\dd_{e}\Phi_{(g,A)})^{\ast})$.
\end{lemma}

\begin{proof}
Since $(h,a)$ is an essential deformation, it satisfies the slice condition $2 \nabla^{g\ast} h = a\lrcorner_g^{\frc}F_A$, see \eqref{eq:adjointdPhi}. Hence:
\begin{equation*}
2 \nabla^{g\ast}\nabla^{g\ast} h = \nabla^{g\ast}(a\lrcorner_g^{\frc}F_A).
\end{equation*} 

\noindent
We chose a local frame $( e_i )$ of $TM$ adapted at a given basepoint and a local frame $(\tau_{\Lambda})$ of $\frG_P$. We compute:
\begin{align*}
\nabla^{g\ast}(a\lrcorner_g^{\frc}F_A)
&= - \sum_{i} \nabla^{g}_{e_i}(a_{\Lambda} \otimes \tau_{\Lambda}\lrcorner_g^{\frc} (F^{\Sigma}_A\otimes \tau_{\Sigma}))(e_i)
= - \sum_{i} \nabla^{g}_{e_i}  (F_A^{\Sigma}(a_{\Lambda},e_i) \frc(\tau_{\Lambda},\tau_{\Sigma}))\\
& = - \sum_{i} (F_A^{\Sigma}(\nabla^{g}_{e_i}  a_{\Lambda},e_i) \frc(\tau_{\Lambda},\tau_{\Sigma}) + F_A^{\Sigma}(a_{\Lambda},e_i) \frc(\nabla^{A}_{e_i} \tau_{\Lambda},\tau_{\Sigma})) \\
& = \sum_{i} \langle F_A^{\Sigma} , e_i \wedge \nabla^{g}_{e_i}  a_{\Lambda}\rangle_g \frc(\tau_{\Lambda},\tau_{\Sigma}) + \langle F_A^{\Sigma} , e_i \wedge  a_{\Lambda}\rangle_g \frc(\nabla^{A}_{e_i} \tau_{\Lambda},\tau_{\Sigma}))\\
&= \langle \dd_A a , F_A \rangle_{g,\frc}\,,
\end{align*}

\noindent
from which the equation follows.
\end{proof}

\noindent
We give now the final general characterization of the essential deformations of an Einstein-Yang-Mills pair $(g,A)\in \cE^{-1}(0)$ that we will use in the following.  

\begin{proposition}
\label{prop:infinitesimal}
Let $(g,A)$ be an Einstein-Yang-Mills pair. A pair $(h,a) \in T_{(g,A)}\Conf(P)$ is an essential deformation of $(g,A)$ if and only if it satisfies the following four equations:
\begin{align*}
&\frac{1}{2} \nabla^{g\ast}\nabla^g h - \mathrm{R}^g_o (h) - \delta^{g}\nabla^{g\ast} h - \frac{1}{2}\nabla^g \dd\mathrm{Tr}_g(h) \\ 
&-  \kappa \Big(F_A \circ_{h,\frc} F_A + h\circ_g (F_A\circ_{g,\frc} F_A) - 2 F_A \circ_{g,\frc} \dd_A a\Big)  - \frac{\kappa}{n-2} \Big( 2 \langle F_A , \dd_A a\rangle_{g,\frc} -  g(F_A\circ_{g,\frc} F_A , h) \Big) g = 0\,, \\
&\dd_A^{g\ast}\dd_A a - a\lrcorner_g^{\frg} F_A - \frac{1}{2} \dd\mathrm{Tr}_g(h)\lrcorner_g  F_A + \dd^{g\ast}_A (F_A)^g_h  = 0\, , \\
&2 \nabla^{g\ast}h = a \lrcorner_g^{\frc} F_A\, ,\\
&\dd^{g\ast}_A a = 0\,.
\end{align*}
\end{proposition}

\begin{proof}
This follows from Proposition \ref{prop:infinitesimalE2} after imposing the equation $\dd_{(g,A)}\s (h,a) = 0$ as computed in the proof of Lemma \ref{lemma:infinitesimaltrace}.  
\end{proof}

\noindent
Given an essential deformation $(h,a)$, the previous proposition immediately implies that $2 \nabla^{g\ast}h^o - a \lrcorner_g^{\frc} F_A \in \Omega^1(M)$ is exact and therefore defines the trivial class in $H^1(M,\mathbb{R})$. Given an Einstein-Yang-Mills pair $(g,A)$, an arbitrary trace-less symmetric two-tensor $h^o$ and an element $a\in\Omega^1(M)$, we will say that $(h^o,a)$ is \emph{completable} or can be completed into an essential deformation of $(g,A)$ if there exists a function $f\in C^{\infty}(M)$ such that $(h = f\, g + h^o , a)$ is an essential deformation. Thus, an obstruction for $(h^o,a)$ to be \emph{completable} is that the cohomology class $[2 \nabla^{g\ast}h^o - a \lrcorner_g^{\frc} F_A] \in H^1(M,\mathbb{R})$ vanishes. When necessary we will denote by $T^o_{(g,A)}\Conf(P)$ the subspace of $T_{(g,A)}\Conf(P)$  consisting of elements $(h^o,a)$ with $h^o$ traceless. If $[2 \nabla^{g\ast}h^o - a \lrcorner_g^{\frc} F_A] = 0$ then we say that $(h^o,a)$ is \emph{unobstructed}. In particular, we have the following result.

\begin{corollary}
\label{cor:obstruction}
Let $(h,a) \in \Ker(\dd_{(g,A)}\cE) \subset T_{(g,A)}\Conf(P)$ be an infinitesimal deformation of an Einstein-Yang-Mills pair $(g,A)$. Then, the following equations hold:
\begin{align*}
\int_M \frac{s^g}{n} \tr_g(h) \nu_g = \frac{2\kappa}{2-n} \int_M  g( (F_A \circ_{g,\frc} F_A)^o , h^o )  \nu_g \, , \qquad  0 = [2 \nabla^{g\ast}h^o - a \lrcorner_g^{\frc} F_A] \in H^{1}(M,\mathbb{R})
\end{align*}
	
\noindent
In particular, if $n=4$ we have $\int_M g( (F_A \circ_{g,\frc} F_A)^o , h^o) \nu_g = 0$.
\end{corollary}

\begin{proof}
Integrate the equations of Lemma \ref{lemma:infinitesimaltrace}.
\end{proof}

\noindent
Hence, infinitesimal deformations of Einstein-Yang-Mills pairs $(g,A)$ are \emph{obstructed} in terms of a cohomology class and of the bilinear $F_A\circ_{g,\frc} F_A \in \Gamma(T^{\ast}M\odot T^{\ast}M)$ constructed out of $g$ and $A$. These are genuine obstructions to the deformation problem of the Einstein-Yang-Mills system. 

\begin{remark}
In fact, deformations of the Einstein-Yang-Mills system do not decouple even for \emph{pure} metric or Yang-Mills deformations. More specifically, an element $a\in\Omega^1(M,\frG_P)$ defines an essential deformation $(0,a)\in T_{(g,A)}\Conf(P)$ of $(g,A)$ if and only if:
\begin{equation}
\label{eq:(0,a)}
\dd_A^{g\ast}\dd_A a = a\lrcorner_g^{\frg} F_A\, , \qquad \dd_A^{g\ast}a = 0\, , \qquad F_A \circ_{g,\frc}\dd_A a = \frac{1}{n-2} \langle F_A , \dd_A a\rangle_{g,\frc} g\, , \qquad  a\lrcorner_g^{\frc} F_A = 0 
\end{equation}

\noindent
Solutions to the first and second equations above correspond to essential deformations of $A$ as a Yang-Mills connection \cite{KoisoV}. Hence, essential deformations of the form $(0,a)$ correspond to the subset of the essential deformations of $A$ as a Yang-Mills connection that satisfies the third and fourth equations above. On the other hand, if we consider deformations of the form $(h,0)\in T_{(g,A)}\Conf(P)$ we obtain a constrained coupled system that does not reduce to the differential system that characterizes infinitesimal deformations of a Ricci-flat metric. 
\end{remark}

\noindent
The previous remark together with Corollary \ref{cor:obstruction} shows that the metrics and connections in the moduli space of Einstein-Yang-Mills pairs are truly coupled. The fact that we are considering simultaneous deformations of $g$ and $A$ allows for a number of special types of deformations by choosing $h$ and $a$ of a particular type. An infinitesimal deformation $(h,a) \in T_{(g,A)}\Conf(P)$ is generated by a smooth family of automorphisms of $P$ if and only if it can be written as: 
\begin{equation*}
(h,a) = (2\delta^g v , \dd_A \tau + F_A(v))  
\end{equation*}

\noindent
in terms of a vector field $v\in\mathfrak{X}(M)$ and a section $\tau \in \Gamma(\frG_P)$. Pairs of this type automatically belong to the kernel of $\dd_{(g,A)}\cE\colon T_{(g,A)}\Conf(P) \to T_{(g,A)}\Conf(P)$ and are essential deformations of the given $(g,A)$ if and only if $v\in\mathfrak{X}(M)$ is Killing and $\tau\in \Gamma(\frG_P)$ is parallel with respect to the connection induced by $A$ on $\frG_P$. We consider in the following one of the most natural deformations of $(g,A)$, which consists of simply multiplying $g$ by a smooth family of positive functions and results in an infinitesimal deformation of the form $(h,a) = (f g, a)$ for a function $f\in C^{\infty}(M)$ and an arbitrary infinitesimal deformation $a\in \Omega^1(M,\frG_P)$ of $A$. 

\begin{corollary}
Let $(f\,g,a) \in T_{(g,A)}\Conf(P)$ be an essential deformation of $(g,A)$. If:
\begin{equation*}
\kappa \big(n (n-2) + 2 (n-2)^2 - 16 \kappa (n-2) + 8\kappa n \big)(4-n) \leq 0
\end{equation*}

\noindent
then $f = 0$ and $a\in \Omega^1(M,\frG_P)$ satisfies \eqref{eq:(0,a)}.
\end{corollary}
 
\begin{proof}
Let $(f\,g,a)$ be an essential deformation of $(g,A)$. Using Lemmas \ref{lemma:infinitesimaltrace} and \ref{lemma:nablanablah} it follows that $f$ satisfies:
\begin{eqnarray*}
\Delta_g f + \frac{2\kappa (n-4)}{n (n-2) + 2 (n-2)^2 - 16 \kappa (n-2) + 8\kappa n } \vert F_A \vert_{g,\frc}^2 f = 0
\end{eqnarray*}

\noindent
Since $\Delta_g$ is a positive operator the conclusion follows.
\end{proof}


\subsection{The Kuranishi model}
\label{subsec:Kuranishitheory}


Let $(g,A)\in \cE^{-1}(0)$ be an Einstein-Yang-Mills pair and fix a slice $\mathbb{S} $ around $(g,A)$. Consider the restriction of $\cE$ to $\mathbb{S} $:
\begin{equation*}
\cE_{\mathbb{S} } \colon \mathbb{S}  \to T_{(g,A)}\Conf(P).
\end{equation*}

\noindent
Then, the differential of $\cE_{\mathbb{S} }$ defines a linear map:
\begin{equation*}
\dd_{(g,A)}\cE_{\mathbb{S} }\colon \mathrm{Ker}(\dd_e\Phi_{(g,A)})^{\ast} \to  T_{(g,A)}\Conf(P)
\end{equation*}

\noindent
given by evaluation of $\dd_{(g,A)}\cE(h,a)$ on $\mathrm{Ker}(\dd_e\Phi_{(g,A)})^{\ast}$. The local moduli space of Einstein-Yang-Mills pairs around an Einstein-Yang-Mills pair can be characterized similarly to the case of Einstein metrics or Yang-Mills connections, as proven by Koiso in \cite{KoisoI,KoisoII,KoisoIII,KoisoIV,KoisoV}.

\begin{theorem}
\label{thm:localKuranishi}
Let $(g,A)$ be an Einstein-Yang-Mills pair and $\mathbb{S} \subset T_{(g,A)}\Conf(P)$ a slice around $(g,A)$. Then, there exists an analytic closed submanifold $\cZ_{(g,A)}\subset \mathbb{S}$ of $\mathbb{S}$ such that:
\begin{equation*}
T_{(g,A)} \cZ  = \Ker(\dd_{(g,A)}\cE)\cap \Ker((\dd_{e}\Phi_{(g,A)})^{\ast})
\end{equation*}

\noindent
and $\cE^{-1}(0)\cap \mathbb{S} $ is an analytic subset of $\cZ_{(g,A)}$.
\end{theorem}

\begin{proof}
By the Hodge decomposition \eqref{eq:decompositionTConf} we have:
\begin{equation*}
\mathrm{Im} (\dd_{(g,A)}\cE_{\mathbb{S} }) = \mathrm{Im} (\dd_{(g,A)}\cE) \subset T_{(g,A)}\Conf(P).
\end{equation*}

\noindent
Denote by:
\begin{equation*}
\mathrm{P}\colon T_{(g,A)}\Conf(P)\to \mathrm{Im} (\dd_{(g,A)}\cE) 
\end{equation*}

\noindent
the natural projection in the orthogonal decomposition given in \eqref{eq:decompositionTConf}. Then, the map:
\begin{equation*}
\mathrm{P}\circ \cE_{\mathbb{S} } \colon \mathbb{S}  \to \mathrm{Im} (\dd_{(g,A)}\cE)
\end{equation*}

\noindent
is a smooth map that has a surjective derivative at $(g,A)\in \Conf(P)$. On the other hand, the slice $\mathbb{S} $ as well as the map $ \cE_{\mathbb{S} }  \colon \mathbb{S}  \to T_{(g,A)}\Conf(P)$ can be obtained as the projective limit of their canonical extensions $\mathbb{S}^s$ and $\cE^s_{\mathbb{S} }$ to the corresponding Sobolev completions in the Sobolev norm $H_s =L^2_s$ with $s>n+4$. Extensions to Sobolev spaces will be denoted with the superscript $s>n+4$. Then:
\begin{equation*}
\mathrm{P}^s\circ \cE^s_{\mathbb{S} } \colon \mathbb{S}^s \to \mathrm{Im} (\dd_{(g,A)}\cE^s)
\end{equation*}

\noindent
has a surjective derivative at $(g,A)$ and thus by the inverse function theorem there exists an open neighborhood $\cK^s\subset \mathbb{S}^s$ of $(g,A)$ such that:
\begin{equation*}
\cZ^s =  (\mathrm{P}^s\circ \cE^s_{\mathbb{S} })^{-1}(0)\cap \cK^s
\end{equation*}

\noindent
is a smooth submanifold of $\mathbb{S}^s$. By the implicit function theorem, it follows that: 
\begin{equation*}
T_{(g,A)} \cZ^s  = \Ker(\dd_{(g,A)}\cE^s)\cap \Ker(\dd_{e}\Phi^s_{(g,A)})^{\ast}\,.
\end{equation*}

\noindent
Furthermore, the zero set of $\cE^s_{\mathbb{S} }$ restricted to $\cK^s$ belongs to the zero set of the restriction of $\mathrm{P}^s\circ \cE^s_{\mathbb{S} }$ to $\cK^s$ and therefore defines an analytic subset of $\cZ^s$. Fix an integer $s$ satisfying $s>n+4$ and consider $q\geq s$. Define:
\begin{equation*}
\cZ^q  = \cZ^s  \cap \mathbb{S}^q \, , \qquad \cK^q = \cK^s\cap \mathbb{S}^q 
\end{equation*}

\noindent
Every $(g^{\prime},A^{\prime}) \in \cZ^q_{(g,A)}$ satisfies:
\begin{equation*}
(\dd_{e}\Phi_{(g,A)}^{q})^{ \ast}\circ \mathrm{E}^{-1}_q ((g^{\prime},A^{\prime})) = 0\, , \qquad  \cE(g^{\prime},A^{\prime}) = 0
\end{equation*}

\noindent
where $\mathrm{E}_q\colon \cU_q \to \cE^q_{\mathbb{S} }$ is the exponential map. This system is elliptic, hence $(g^{\prime},A^{\prime})$ is smooth. Therefore, $\cZ^q= \cZ^s$ for every $q\geq s$ and, in particular, $\cZ^q$ consists only of smooth elements. Since $\mathrm{P}^s\circ \cE^s_{\mathbb{S} } \colon \mathbb{S}^s \to \mathrm{Im} (\dd_{(g,A)}\cE^s)$ has surjective derivative on $\cZ^s$, hence also on $\cZ^q$, it follows that for every element:
\begin{equation*}
(h_o,a_o) \in \mathrm{Im} (\dd_{(g^{\prime},A^{\prime})}(\mathrm{P}^s\circ \cE^s_{\mathbb{S}}\vert_{\cK^q}))
\end{equation*}

\noindent
there exist elements $(h_1,a_1) \in \Ker((\dd_{e}\Phi_{(g,A)}^s)^{ \ast})$ and $(h_2,a_2) \in \Ker (\dd_{(g,A)}\cE^s)$ such that:
\begin{equation*}
 \dd_{(g^{\prime},A^{\prime})}(\mathrm{P}^s\circ \cE^s_{\mathbb{S}})(h_1,a_1) = (h_0,a_0)  + (h_2,a_2) .
\end{equation*}

\noindent
Hence:
\begin{equation*}
 \dd_{(g,A)}\cE^s \circ \dd_{(g^{\prime},A^{\prime})}(\mathrm{P}^s\circ \cE^s_{\mathbb{S}})(h_1,a_1) = \dd_{(g,A)}\cE^s(h_0,a_0) \, , \qquad   \dd_{e}\Phi_{(g,A)}^{s \ast}(h_1,a_1) = 0.
\end{equation*}

\noindent
Therefore, $(h_1,a_1)\in \cK^q$ and consequentely $\mathrm{P}^s\circ \cE^s_{\mathbb{S}}\vert_{\cK^q}$ has surjective derivative at every point in $\cZ^q$, implying that the latter is a closed analytic submanifold of $\cK^q$. Hence, we end up with a system:
\begin{equation*}
\left\{ \cK^q, \cZ^q , \mathrm{P}^s\circ \cE^s_{\mathbb{S} }\vert_{\cZ^q}\right\}_{q\geq s}
\end{equation*}

\noindent
of open sets, closed analytic submanifolds of the latter, and smooth maps. Since $\cZ^q$ does not depend on $q$ and contains only smooth connections, the projective limit gives the desired smooth Fréchet closed submanifold $\cZ\subset \cK$ of the open set $\cK\subset \mathbb{S}$ of the slice $\mathbb{S}$.
\end{proof}  

\begin{remark}
By the proof of the previous theorem, we conclude that \cite[Theorem 5.3]{DiezRudolphII} holds and therefore we obtain a Kuranishi chart around every Einstein-Yang-Mills pair $(g,A)$ in the sense of \cite{DiezRudolphII}, see Definition \cite[Definition 5.1]{DiezRudolphII}.
\end{remark}

\noindent
Following N. Koiso \cite{KoisoI,KoisoII,KoisoIII} we introduce the following terminology.

\begin{definition}
Let $(g,A)\in \cE^{-1}(0)$ be an Einstein-Yang-Mills pair. If the vector space of infinitesimal deformations of $(g,A)$ is of dimension zero, then $(g,A)$ is \emph{infinitesimally rigid}. 
\end{definition}

\noindent
There is also a natural notion of \emph{local} rigidity of Einstein-Yang-Mills pairs, which is the natural generalization of the rigidity notion introduced by Koiso for Einstein metrics \cite{KoisoI,KoisoII,KoisoIII}. 

\begin{definition}
Let $(g,A)\in \cE^{-1}(0)$ be an Einstein-Yang-Mills pair. If there exists an $\Aut(P)$-invariant open set $\mathfrak{V}$ of $(g,A)$ containing $\cO_{(g,A)}$ and such that every Einstein-Yang-Mills pair in $\mathfrak{V}$ is an element of $\cO_{(g,A)}\subset \mathfrak{V}$ then $(g,A)$ is \emph{rigid}.
\end{definition}

\begin{corollary}
If an Einstein-Yang-Mills pair is infinitesimally rigid then it is rigid. 
\end{corollary}

\begin{remark}
Note that we cannot expect the converse to the previous proposition to hold, since it already fails in the Einstein case \cite{KoisoIII}.
\end{remark}


\section{Einstein-Yang-Mills deformations of instantons on Calabi-Yau two-folds}
\label{sec:instantonK3}



\subsection{Preliminaries}


Throughout this section, we take $M$ to be four-dimensional, hence $n=4$. In this dimension the Riemannian Hodge dual squares to the identity. We say that a connection $A$ on $P$ is anti-self-dual if:
\begin{equation*}
\ast_g F_A = - F_A\,.
\end{equation*}

\noindent
Self-dual connections are defined similarly in terms of the opposite sign. 

\begin{proposition}
\label{prop:RicciFlatASD}
Let $g$ be a Ricci-flat metric on $M$ and $A$  an anti-self-dual connection on $P$. Then $(g,A)$ is an Einstein-Yang-Mills pair.
\end{proposition}

\begin{proof}
Since $g$ is Ricci flat its Einstein tensor vanishes. Furthermore, using the identity:
\begin{equation*}
F_A\circ_{g,\frc} F_A = \frac{1}{2} \vert F_A \vert^2_{g,\mathfrak{c}}
\end{equation*}

\noindent
it follows that
\begin{equation*}
\cT(g,A) = \frac{\kappa}{2} \vert F_A \vert^2_{g,\mathfrak{c}}\,  g  - \kappa\, F_A\circ_{g,\frc} F_A
\end{equation*}

\noindent
and we conclude.
\end{proof}

\noindent
Therefore, anti-self-dual instantons on Ricci-flat four-manifolds provide, in case they exist, a distinguished class of Einstein-Yang-Mills pairs. We will refer to such pairs simply as \emph{anti-self-dual Einstein-Yang-Mills pairs}.  


\subsection{Infinitesimal deformations of anti-self-dual Einstein-Yang-Mills pairs}


In this subsection, we prove a refinement of Proposition \ref{prop:infinitesimal} for anti-self-dual Einstein-Yang-Mills pairs in four dimensions. As we will see in the following, this refinement allows to \emph{decouple} the differential system determining $(h^o,a)$ from the differential condition satisfied by $\mathrm{Tr}_g(h)$.

\begin{theorem}
\label{thm:deformationsEYMASD}
Let $(g,A)$ be an anti-self-dual Einstein-Yang-Mills pair on a principal bundle $P$ over a four-dimensional manifold $M$. An unobstructed pair $(h^o,a) \in T^o_{(g,A)}\Conf(P)$ is a completable essential deformation of $(g,A)$ if and only if:
\begin{equation}\label{eq:deformationshgeneralASD}
\begin{array}{l}
 \frac{1}{2}\nabla^{g\ast}\nabla^g h^o   -  \mathrm{R}^g_o (h^o) - 2 \delta^{g}\nabla^{g\ast} h^o - \frac{1}{6} (\nabla^{g\ast}\nabla^{g\ast}h^o) g + \frac{1}{2}\delta^g (a \lrcorner_g^{\frc} F_A) \\ 
 - \kappa \Big( F_A\circ_{h^o,\frc} F_A + \frac{1}{2} \vert F_A \vert_{g,\frc}^2 h^o - 2 F_A \circ_{g,\frc} \dd_A a + \langle F_A , \dd_A a\rangle_{g,\frc} g \Big) = 0 \,,
\end{array}
\end{equation}
\begin{equation}
\dd_A^{g\ast}\dd_A a - a\lrcorner_g^{\frg} F_A + \dd^{g\ast}_A (F_A)^g_{h^o}  = 0\, , \qquad    \dd^{g\ast}_A a = 0\,. \label{eq:deformationsAgeneralASD}
\end{equation}

\noindent
If that is the case, the completed essential deformation $(h,a)$ satisfies $h = \displaystyle\frac{f}4 g + h^o$, where $\dd f = 4\nabla^{g\ast} h^o - 2 \, a\lrcorner_g^{\frc} F_A $.
\end{theorem}

\begin{proof}
By Proposition \ref{prop:infinitesimal}, a pair $(h,a)$ is an essential deformation of an anti-self-dual Einstein-Yang-Mills pair $(g,A)$ in four dimensions if and only if the following differential system holds:

\begin{equation}\label{eq:inf14d}
\begin{array}{l}
  \frac{1}{2} \nabla^{g\ast}\nabla^g h - \mathrm{R}^g_o (h) - \delta^{g}\nabla^{g\ast} h - \frac{1}{2}\nabla^g \dd\mathrm{Tr}_g(h)  \\
  -  \kappa \Big(F_A \circ_{h,\frc} F_A + h\circ_g (F_A\circ_{g,\frc} F_A) - 2 F_A \circ_{g,\frc} \dd_A a  - \langle F_A , \dd_A a\rangle_{g,\frc} g + \frac{1}{2} g(F_A\circ_{g,\frc} F_A , h) g\Big)  = 0 \,,
\end{array}
\end{equation}
\begin{equation}
    \dd_A^{g\ast}\dd_A a - a\lrcorner_g^{\frg} F_A - \frac{1}{2} \dd\mathrm{Tr}_g(h)\lrcorner_g  F_A + \dd^{g\ast}_A (F_A)^g_h  = 0\, , \qquad 2 \nabla^{g\ast}h = a \lrcorner_g^{\frc} F_A\, , \qquad  \dd^{g\ast}_A a = 0\,.\label{eq:inf24d}
\end{equation}

\noindent
A direct computation shows that:
\begin{equation*}
\dd^{g\ast}_A (F_A)^g_h = \dd^{g\ast}_A (F_A)^g_{h^o} + \frac{1}{2} \dd\mathrm{Tr}_g(h)\lrcorner_g  F_A\,,
\end{equation*}

\noindent
and therefore the first equation in \eqref{eq:inf24d} becomes:
\begin{equation*}
\dd_A^{g\ast}\dd_A a - a\lrcorner_g^{\frg} F_A - \frac{1}{2} \dd\mathrm{Tr}_g(h)\lrcorner_g  F_A + \dd^{g\ast}_A (F_A)^g_h  = \dd_A^{g\ast}\dd_A a - a\lrcorner_g^{\frg} F_A + \dd^{g\ast}_A (F_A)^g_{h^o} = 0.
\end{equation*}

\noindent
In particular, the trace of $h$ decouples from this equation. On the other hand, the slice condition  $2 \nabla^{g\ast}h = a \lrcorner_g^{\frc} F_A$ is equivalent to:
\begin{equation}
\label{eq:dh}
\frac{1}{2} \dd\mathrm{Tr}_g(h) = 2 \nabla^{g\ast} h^o - a \lrcorner_g^{\frc} F_A\,.
\end{equation}

\noindent
This equation isolates the differential of $\mathrm{Tr}_g(h)$ in terms of $h^o$ and $a$. Substituting now:
\begin{equation*}
h = h^o + \frac{\mathrm{Tr}_g(h)}{4} g
\end{equation*}

\noindent
in \eqref{eq:inf14d} it can be checked that all the terms proportional to $\mathrm{Tr}_g(h)$ drop out and only terms involving derivatives of $\mathrm{Tr}_g(h)$ remain. Using  \eqref{eq:dh} as well as Lemma \ref{lemma:nablanablah} in  \eqref{eq:inf14d}, we substitute the exterior derivative of $\mathrm{Tr}_g(h)$ in terms of $h^o$, $a$ and their derivatives, obtaining Equation \eqref{eq:deformationshgeneralASD}. Conversely, let $(h^o,a) \in T^o_{(g,A)}\Conf(P)$ be an unobstructed pair satisfying the differential system given in  \eqref{eq:deformationshgeneralASD} and \eqref{eq:deformationsAgeneralASD}. Since $(h^o,a)$ is unobstructed, there exists a function $f\in C^{\infty}(M)$ satisfying:
\begin{equation*}
\frac{1}{2} \dd f =2 \nabla^{g\ast} h^o - a \lrcorner_g^{\frc} F_A
\end{equation*}

\noindent
Defining:
\begin{equation*}
    h = \frac{f}{4} g + h^o
\end{equation*}

\noindent
and using this equation to substitute $h^o$ in terms of $h$ and $f$ in Equations \eqref{eq:deformationshgeneralASD} and \eqref{eq:deformationsAgeneralASD}, we obtain the differential system \eqref{eq:inf14d} and \eqref{eq:inf24d} and hence we conclude.
\end{proof}


\subsection{Anti-self-dual infinitesimal deformations}


We consider the following smooth map of manifolds:
\begin{eqnarray*}
\cE_o := (\cE^o_1,\cE^o_2) \colon \Conf(P) &\to& \Gamma(T^{\ast}M\odot T^{\ast}M) \times \Omega^2(M,\frG_P)\,,\\
 (g,A) &\mapsto& \left(\cE_1^o(g,A) := \mathrm{Ric}^g\, ,
 \cE_2^o(g,A) := \ast_g F_{A} + F_A\right) .
\end{eqnarray*}

\noindent
We will refer to the differential system $\cE_o(g,A) = 0$ as the \emph{anti-self-dual Einstein-Yang-Mills system}. The space of solutions to the anti-self-dual Einstein-Yang-Mills system is given by $\cE^{-1}_o(0)\subset \Conf(P)$, which we consider to be endowed with the subspace topology. The automorphism group $\Aut(P)$ acts on $\Conf(P)$ equivariantly with respect to $\cE_o$ whence $\Aut(P)$ preserves the solution space $\cE^{-1}_o(0)\subset \Conf(P)$. The quotient:
\begin{equation*}
\mathfrak{M}_o(P,\frc) =  \cE^{-1}_{o}(0)/\Aut(P)\, ,
\end{equation*}

\noindent
equipped with quotient topology is the moduli space of \emph{anti-self-dual Einstein-Yang-Mills pairs}. 

\begin{lemma}
\label{lemma:differentialASD}
The differential of $\cE_o  \colon \Conf(P)  \to  \Gamma(T^{\ast}M\odot T^{\ast}M) \times \Omega^2(M,\frG_P)$ at an anti-self-dual Einstein-Yang-Mills pair $(g,A)$ is given by:
\begin{align*}
\dd_{(g,A)}\cE^o_1(h,a) & = \frac{1}{2}\nabla^{g\ast}\nabla^g h -  \mathrm{R}^g_o (h) - \delta^{g}\nabla^{g\ast} h - \frac{1}{2}\nabla^g \dd\mathrm{Tr}_g(h)\,,  \\
\dd_{(g,A)}\cE^o_2(h,a) & =   \ast_g\dd_A a + \dd_A a  + \ast_g (F_A)^g_{h^o}\,,
\end{align*}

\noindent
where $h = \displaystyle \mathrm{Tr}_g(h) g+h^o$.
\end{lemma}

\begin{proof}
The expression for $\dd_{(g,A)}\cE^o_1(h,a)$ follows directly from Lemma \ref{lemma:variationRs} after imposing $\mathrm{Ric}^g = 0$. To obtain $\dd_{(g,A)}\cE^o_2(h,a)$ we compute:
\begin{equation*}
\dd_{(g,A)}\cE^o_2(h,a) = \ast_g \dd_A a + \dd_A a + \frac{1}{2} \mathrm{Tr}_g(h) \ast_g F_A + \ast_g (F_A)^g_h = \ast_g \dd_A a + \dd_A a + \ast_g (F_A)^g_{h^o}\,,
\end{equation*}

\noindent
and we conclude.
\end{proof}

\noindent
The vector space of \emph{essential deformations} of an anti-self-dual Einstein-Yang-Mills pair as an anti-self-dual Einstein-Yang-Mills pair, in contrast to as a general Einstein-Yang-Mills pair, is defined as follows:
\begin{equation*}
\mathbb{E}_{(g,A)} := \left\{ (h,a) \in T_{(g,A)}\Conf(P) \,\, \vert\,\, \dd_{(g,A)}\cE^o(h,a) = 0,  \,\, (\dd_e\Phi_{(g,A)})^{\ast}(h,a) = 0\right\}\,.
\end{equation*}

\noindent
Essential deformations of anti-self-dual Einstein-Yang-Mills pairs as anti-self-dual Einstein-Yang-Mills pairs are characterized by the following proposition.

\begin{proposition}
\label{prop:completionRicciFlat}
A pair $(h^o,a)\in T^o_{(g,A)}\Conf(P)$ can be completed into an essential deformation $(h,a)\in \mathbb{E}_{(g,A)}$ of $(g,A)$ as an anti-self-dual Einstein-Yang-Mills if and only if:
\begin{align}
&  \frac{1}{2}\nabla^{g\ast}\nabla^g h^o   -  \mathrm{R}^g_o (h^o) - 2 \delta^{g}\nabla^{g\ast} h^o - \frac{1}{6} (\nabla^{g\ast}\nabla^{g\ast}h^o ) g + \frac{1}{2}\delta^g (a \lrcorner_g^{\frc} F_A) = 0 \,,\label{eq:RicciFlatEinsteinDeformations}\\
& \ast_g\dd_A a + \dd_A a  + \ast_g (F_A)^g_{h^o} = 0\, , \qquad    \dd^{g\ast}_A a = 0\,. \label{eq:YMDeformations}
\end{align}

\noindent
In particular, $\langle F_A , (F_A)^g_{h^o}\rangle_{g,\frc} = 0$ and $\ast_g (F_A)^g_{h^o} =(F_A)^g_{h^o}$
\end{proposition}

\begin{remark}
We have the following explicit formula for $(F_A)^g_{h^o} \in \Omega^2(M,\frG_P)$:
\begin{equation*}
(F_A)^g_{h^o}(v_1 , v_2) = -F_A(h^o(v_1),v_2) + F_A(h^o(v_2),v_1)\,,
\end{equation*}

\noindent
for every $v_1,v_2\in\mathfrak{X}(M)$. Using this equation it can be checked directly that $(F_A)^g_{h^o}$ is indeed self-dual.
\end{remark}

\begin{proof}
Follows from Lemma \ref{lemma:differentialASD} and Theorem \ref{thm:deformationsEYMASD} after formally setting $\kappa = 0$.
\end{proof}

\noindent
Since every anti-self-dual Einstein-Yang-Mills pair is automatically Einstein-Yang-Mills, it follows that every essential deformation of $(g,A)$ as an anti-self-dual Einstein-Yang-Mills pair must be an essential deformation of $(g,A)$ as an Einstein-Yang-Mills pair. Nonetheless, it is an instructive exercise to verify this explicitly.

\begin{proposition}
\label{prop:AYMEssential}
We have a natural inclusion $\mathbb{E}_{(g,A)}\subset \mathbb{H}^1_{(g,A)}$.
\end{proposition}

\begin{proof}
Let $(h,a) \in \mathbb{E}_{(g,A)}$. The pair $(h,a)$ satisfies  \eqref{eq:deformationshgeneralASD} if and only if:
\begin{equation}
\label{eq:conditionh}
\langle F_A , \dd_A a\rangle_{g,\frc} g - 2 F_A \circ_{g,\frc} \dd_A a + F_A\circ_{h^o,\frc} F_A + \frac{1}{2} \vert F_A \vert_{g,\frc}^2 h^o  = 0.
\end{equation}

\noindent
We compute:
\begin{align*}
2 F_A \circ_{g,\frc} \dd_A a &= 2 F^{-}_A \circ_{g,\frc} (\dd_A a)^{-} + 2 F^{-}_A \circ_{g,\frc} (\dd_A a)^{+} = \langle F_A,\dd_A a\rangle_{g,\frc} g + 2 F_A \circ_{g,\frc} (\dd_A a)^{+} \\
& = \langle F_A , \dd_A a\rangle_{g,\frc} g - F_A \circ_{g,\frc} (F_A)^g_{h^o}\,.
\end{align*}

\noindent
On the other hand, we have:
\begin{eqnarray*}
 - (F_A \circ_{g,\frc}  (F_A)^g_{h^o})(v_1,v_2) = \frac{1}{2} \langle F_A(v_1) , F_A (h^o(v_2)) \rangle_{g,\frc} + \frac{1}{2} \langle F_A(v_2) , F_A (h^o(v_1)) \rangle_{g,\frc}  + (F_A \circ_{h^o,\frc} F_A) (v_1,v_2)\,.
\end{eqnarray*}

\noindent
Furthermore:
\begin{equation*}
 \langle F_A(v_1) , F_A (h^o(v_2)) \rangle_{g,\frc} +   \langle F_A(v_2) , F_A (h^o(v_1)) \rangle_{g,\frc} = \vert F_A \vert^2_{g,\frc} h^o\,,
\end{equation*}

\noindent
which finally gives:
\begin{equation*}
2 F_A \circ_{g,\frc} \dd_A a =  \langle F_A , \dd_A a\rangle_{g,\frc} g+ \frac{1}{2} \vert F_A \vert^2_{g,\frc} h^o + F_A \circ_{h^o,\frc} F_A\,,
\end{equation*}

\noindent
whence \eqref{eq:conditionh} is satisfied. Applying now the differential operator $-\ast_g \dd_A$ to the first equation in \eqref{eq:YMDeformations} we obtain:
\begin{equation*}
\dd^{g\ast}_A\dd_A a -\ast_g \dd^2_A a  + \dd_A^{g\ast} (F_A)^g_{h^o} = 0\,.
\end{equation*}

\noindent
We compute:
\begin{eqnarray*}
& \dd^2_A a  = e_j \wedge \nabla^{g,A}_{e_j} (e_i \wedge \nabla^{g,A}_{e_i} (a^{\Lambda}\otimes \tau_{\Lambda}) )  = e_j\wedge e_i \wedge a^{\Lambda} \otimes \nabla^A_{e_j}\nabla^A_{e_i} \tau_{\Lambda} = \frac{1}{2} e_j\wedge e_i \wedge a^{\Lambda} \otimes  [F_A(e_i,e_j), \tau_{\Lambda}]_{\frG_P} \\
& = F^{\Sigma}_A \wedge a^{\Lambda} \otimes [\tau_{\Sigma},\tau_{\Lambda}]_{\frG_P} \,. 
\end{eqnarray*}

\noindent
This implies:
\begin{equation*}
\ast_g \dd^2_A a = \ast_g (F^{\Sigma}_A \wedge a^{\Lambda})  \otimes [\tau_{\Sigma},\tau_{\Lambda}]_{\frG_P}  = - (a^{\Lambda} \lrcorner_g F_A^{\Sigma}) \otimes [\tau_{\Sigma},\tau_{\Lambda}]_{\frG_P}  = a \lrcorner_g^{\frc} F_A\,.
\end{equation*}

\noindent
Thus, the first equation in \eqref{eq:YMDeformations} is satisfied as well.
\end{proof}

\noindent
By Proposition \ref{prop:RicciFlatASD} we have a natural injection:
\begin{equation*}
\mathbb{I}\colon \mathfrak{M}_o(P,\frc) \hookrightarrow \mathfrak{M}(P,\frc)
\end{equation*}

\noindent
of $\mathfrak{M}_o(P,\frc)$ into the moduli space of Einstein-Yang-Mills pairs. By the structural result given in Theorem \ref{thm:localKuranishi} it is clear that $\mathbb{I}$ is a local injective immersion which is isometric with respect to the natural $L^2$ metric on both $\mathfrak{M}_o(P,\frc)$ and $\mathfrak{M}(P,\frc)$. A natural question arises in this context: Is every infinitesimal deformation of $(g,A)$ as an Einstein-Yang-Mills pair automatically an infinitesimal deformation as an anti-self-dual Einstein-Yang-Mills pair? A positive answer to this question would give the Einstein-Yang-Mills analog of the theorem by which every Ricci flat deformation of a Calabi-Yau metric is again Calabi-Yau.

\begin{theorem}
\label{thm:CYcase}
Let $(M,g)$ be a Calabi-Yau two-fold and let $(g,A)$ be an anti-self-dual Einstein-Yang-Mills pair. Then $\mathbb{E}_{(g,A)} = \mathbb{H}^1_{(g,A)}$.
\end{theorem}

\begin{proof}
It is enough to prove that $\mathbb{H}^1_{(g,A)} \subset \mathbb{E}_{(g,A)}$ if $(g,A)$ is an anti-self-dual Einstein-Yang-Mills pair. By Theorem \ref{thm:deformationsEYMASD} and Proposition \ref{prop:completionRicciFlat} a pair $(h^o, a)$ that can be completed into an essential deformation of $(g,A)$ can also be completed into an anti-self-dual essential deformation of $(g,A)$ if and only if:
\begin{align}
&  F_A\circ_{h^o,\frc} F_A + \frac{1}{2} \vert F_A \vert_{g,\frc}^2 h^o - 2 F_A \circ_{g,\frc} \dd_A a + \langle F_A , \dd_A a\rangle_{g,\frc} g  = 0\,, \label{eq:condition1} \\
&  \ast_g\dd_A a + \dd_A a  + (F_A)^g_{h^o} = 0\,. \label{eq:condition2}
\end{align}

\noindent
Let $(h,a)\in \mathbb{H}^1_{(g,A)}$. Then, $(h^o,a)$ satisfies:
\begin{equation*}
\dd_A^{g\ast}\dd_A a - a\lrcorner_g^{\frg} F_A + \dd^{g\ast}_A (F_A)^g_{h^o}  = -\ast_g \dd_A(\ast_g \dd_A a  + \dd_A a  +  (F_A)^g_{h^o}) = 0\,.
\end{equation*}

\noindent
Hence:
\begin{equation*}
\dd_A a + \ast_g \dd_A a  + (F_A)^g_{h^o} = b\,,
\end{equation*}

\noindent
where $b\in \Ker(\dd_A) \subset \Omega^2(M,\frG_P)$. Furthermore, since the left hand side of the previous equation is self-dual, $b$ is also self-dual and in addition $b\in \Ker(\dd^{g\ast}_A) \subset \Omega^2(M,\frG_P)$. Hence $b$ is a self-dual two-form with values in $\frG_P$ harmonic with respect to $g$ and $A$. Similarly to the proof of Proposition \ref{prop:AYMEssential}, we compute:
\begin{align*}
2 F_A \circ_{g,\frc} \dd_A a &= 2 F^{-}_A \circ_{g,\frc} (\dd_A a)^{-}  +2 F^{-}_A \circ_{g,\frc} (\dd_A a)^{+} = \langle F_A , \dd_A a \rangle_{g,\frc} g  +2 F_A \circ_{g,\frc} (b - (F_A)^g_{h^o}) \\
& = 2 F_A \circ_{g,\frc} b + \langle F_A , \dd_A a \rangle_{g,\frc} g + \frac{1}{2} \vert F_A \vert^2_{g,\frc} h^o + F_A \circ_{h^o,\frc} F_A\,.
\end{align*}

\noindent
Therefore, equations \eqref{eq:condition1} and \eqref{eq:condition2} hold if and only if $b=0$. In order to prove that this is indeed the case, we consider the Weitzenböck formula for the Laplacian associated to $\nabla^g$ and $A$ acting on the bundle of two forms with values in the adjoint bundle $\frG_P$. A computation gives:
\begin{eqnarray*}
((\dd^g_A \dd_A^{g\ast} + \dd_A^{g\ast}\dd^g_A) k)(v_1,v_2) = (\nabla^{g,A \ast} \nabla^{g,A} k)(v_1,v_2) - (\cR^{g,A}_{v_1,e_i} k)(e_i,v_2) + (\cR^{g,A}_{v_2,e_i} k)(e_i,v_1)\, , \quad k\in \Omega^2(\frG_P)\,,
\end{eqnarray*}

\noindent
where $\dd^g_A\colon \Omega^2(\frG_P) \to \Omega^3(\frG_P)$ is the exterior covariant derivative determined by $\nabla^g$ and $A$ on $\Omega^2(\frG_P)$, $\dd^{g\ast}_A\colon \Omega^2(\frG_P)\to \Omega^1(\frG_P)$ is its formal adjoint, $\nabla^{g,A}$ is the connection induced by $\nabla^g$ and $A$ on $\Omega^2(\frG_P)$, $(\nabla^{g,A})^{\ast}\colon\Omega^2(\frG_P) \to \Omega^1(\frG_P)$ is its formal adjoint and $\cR^{g,A}\in \Omega^2(\End(\wedge^2 T^{\ast}M\otimes \frG_P))$ is the curvature of $\nabla^{g,A}$. We have the identity:
\begin{eqnarray*}
(\cR^{g,A}_{v_1,v_2} k)(v_3,v_4) = [F_A(v_1,v_2) , k(v_3,v_4)]_{\frG_P} - k(R^g_{v_1,v_2}v_3,v_4) - k(v_3,R^g_{v_1,v_2} v_4)\,,
\end{eqnarray*}

\noindent
for every $v_1, \hdots , v_4 \in \mathfrak{X}(M)$. Using the fact that $g$ is Ricci-flat, we obtain:
\begin{eqnarray*}
& (\cR^{g,A}_{v_1,e_i} b)(e_i,v_2) - (\cR^{g,A}_{v_2,e_i} b)(e_i,v_1) = [F_A(v_1, e_i) , b(e_1,v_2)]_{\frG_P} -  [F_A(v_2, e_i) , b(e_1,v_1)]_{\frG_P} \\ 
& - b(e_i,R^g_{v_1,e_i} v_2) + b(e_i,R^g_{v_2,e_i} v_1)\,.
\end{eqnarray*}

\noindent
Using that $F_A$ is anti-self-dual while $b$ is self-dual it follows that:
\begin{eqnarray*}
[F_A(v_1, e_i) , b(e_1,v_2)]_{\frG_P} -  [F_A(v_2, e_i) , b(e_1,v_1)]_{\frG_P} = 0\,.
\end{eqnarray*}

\noindent
Furthermore:
\begin{eqnarray*}
b(e_i,R^g_{v_1,e_i} v_2) - b(e_i,R^g_{v_2,e_i} v_1) = b(e_i,R^g_{v_1,v_2} e_i) = 0\,,
\end{eqnarray*}

\noindent
where we have used the first Bianchi identity for $R^g$ and the fact that $g$ is anti-self-dual since it is Calabi-Yau. We conclude that: 
\begin{eqnarray*}
((\dd^g_A \dd_A^{g\ast} + \dd_A^{g\ast}\dd^g_A) b)(v_1,v_2) = (\nabla^{g,A \ast} \nabla^{g,A} b)(v_1,v_2) \,,
\end{eqnarray*}

\noindent
and therefore $\nabla^{g,A} b = 0$. As $(M,g)$ is a Calabi-Yau, there is a parallel $(2,0)$-form $\theta$ and an associated parallel Kähler form $\omega$. Hence, the complexification of the bundle of self-dual two forms $\wedge_{+}^2 M$ decomposes as follows:
\begin{equation*}
\wedge^2_+ M \otimes \CC=\wedge^{2,0} M\oplus \underline{\RR}\cdot  \omega\oplus \wedge^{0,2} M
\end{equation*}

\noindent
is trivial with the trivial connection. Let $x_1,x_2,x_3$ be a parallel frame of $\Omega_+^2(M)$. Since the element $b\in \Omega^2_+(\frG_P)$ is parallel with respect to $\nabla^{g,A}$, writing $b=\sum x_a\, b_a$ in terms of sections $b_a\in \Gamma(\frG_P)$ with $a=1,2,3$, we obtain:
\begin{equation*}
0=\nabla^{g,A} b= \sum_{a = 1}^3 x_a \,\nabla^A b_a\,,
\end{equation*}

\noindent
and so $\nabla^A b_a=0$ for every $a=1,2,3$. 
Assume that they are linearly independent (if not, just throw out those that are linearly dependent and work with
the remaining ones).
Let $\frF_P$ be the subbundle generated by
the sections $b_a$, so we have a decomposition
$\frG_P=\frF_P\oplus \frF_P^\perp$. The bundle
$\frF_P$ is trivial with a trivial flat connection
and $b=\sum x_a \, b_a \in \Gamma (\frF_P)$.
Hence $F_A\in \Omega^2(\frF_P^\perp)$.

Let $b_l$  be a frame for $\frF_P^\perp$, and note that $\nabla^A b_l \in \Omega^1 (\frF_P^\perp)$. As $F_A\in \Omega^2(\frF_P^\perp)$, We have that $(F_A)^g_{h_o}\in \Omega^2(\frF_P^\perp)$ and therefore the equation $b = \dd_A a + \ast_g \dd_A a + (F_A)^g_{h_0}$ implies that the  $b_a$-component of $a=\sum \alpha_a\otimes b_a + \sum \alpha_l \otimes b_l$, where $\alpha_a, \alpha_l \in C^{\infty}(M)$, gives  $\sum x_a \otimes b_a  =(\dd\alpha_a + \ast_g\dd \alpha_a) \otimes b_a$, that is, $x_a= \dd \alpha_a + \ast_g \dd \alpha_a$. Recall that $x_a$ is parallel, so in particular it is closed. Hence $\dd^{g\ast} \dd\alpha_a = 0$, from which $\dd\alpha_a = 0$, and hence $x=0$, and thus $b=0$, as required.
\end{proof}


\appendix



\section{Various differential operators in Riemannian geometry}
\label{app:curvatureformulae}


In order to fix notation and conventions, in this section we collect the Riemannian formulas that we will need throughout the main body of the article. We will mostly follow \cite[\S 1]{Besse} with an opposite sign for the definition of the Riemann curvature tensor. Let $(M,g)$ be a compact and oriented Riemannian manifold of dimension $n$ with Riemannian volume form $\nu_g$. We consider the exterior algebra $\wedge T^{\ast}M$ of $M$ as the subspace of the tensor algebra $\cT(M)$ of $M$ determined by the image of the following vector bundle map:
\begin{equation*}
\wedge T^{\ast}M \hookrightarrow \cT(M) \, , \qquad \alpha_1\wedge \cdots \wedge \alpha_r \to \sum_{\sigma \in S_r} (-1)^{\sigma} \alpha_{\sigma(1)}\otimes \cdots \otimes \alpha_{\sigma(r)}\,,
\end{equation*}

\noindent
where $\alpha_1 , \hdots , \alpha_r \in \Omega^1(M)$ and $\sigma \in S_r$ runs over all permutations of the set of $r$ elements. Therefore we consider differential forms as skew-symmetric $(r,0)$ tensors, that is:
\begin{equation*}
(\alpha_1\wedge \cdots \wedge \alpha_r) (v_1 , \cdots, v_r) = \sum_{\sigma \in S_r} (-1)^{\sigma}  (\alpha_1\otimes \cdots \otimes \alpha_r) (v_{\sigma(1)} , \cdots, v_{\sigma(r)} ) = \sum_{\sigma \in S_r} (-1)^{\sigma} \alpha_1 (v_{\sigma(1)}) \cdots   \alpha_r (v_{\sigma(r)})\, ,  
\end{equation*}

\noindent
for every $v_1, \hdots v_r \in TM$. Consequently, in our conventions the wedge product of two differential forms is given by the formula:
\begin{equation*}
(\omega_1\wedge \omega_2)(v_1, \hdots , v_{r_1 + r_2}) = \frac{1}{r_1! r_2!} \sum_{\sigma\in S_{r_1+r_2}} (-1)^{\sigma}\omega_1(v_{\sigma(1)},\hdots , v_{\sigma(r_1)})\omega_2(v_{\sigma(r_1+1)},\hdots , v_{\sigma(r_1+r_2)})
\end{equation*}

\noindent
where $\omega_1 \in \wedge^{r_1}T^{\ast}M$, $\omega_2 \in \wedge^{r_2}T^{\ast}M$ and $v_1,\hdots , v_{r_1+r_2} \in TM$. In particular, for every $\omega \in \wedge T^{\ast}M$ we can write:
\begin{equation*}
\omega  = \sum_{i_1,\hdots , i_r}  \frac{1}{r!}\omega_{i_1 , \hdots , i_r} e^{i_1} \wedge \hdots \wedge e^{i_r} = \sum_{i_1 < \hdots < i_r}   \omega_{i_1 , \hdots , i_r} e^{i_1} \wedge \hdots \wedge e^{i_r} = \sum_{i_1,\hdots , i_r} \omega_{i_1 , \hdots , i_r} e^{i_1} \otimes \hdots \otimes e^{i_r}\,,
\end{equation*}

\noindent
where $(e^1,\hdots, e^{n})$ is any given coframe.

\noindent
The metric induced by $g$ on the tensor bundles over $M$ will be denoted again by $g( \cdot ,\cdot)$, whereas the determinant metric determined by $g$ on the exterior algebra bundle of $M$ will be denoted by $\langle\cdot ,\cdot \rangle_g$. In our conventions the Hodge dual operator $\ast_g\colon \wedge T^{\ast} M \to \wedge T^{\ast}M$ on $(M,g)$ is determined by the following equation:
\begin{equation*}
\alpha\wedge \ast_g \beta = \langle \alpha , \beta\rangle_g \nu_g\, ,
\end{equation*}

\noindent
for any pair of differential forms $\alpha$ and $\beta$ on $M$. In particular, if $(e^1,\hdots , e^n)$ is a local orthonormal frame then:
\begin{equation*}
\ast_g(e^{i_1}\wedge \hdots \wedge e^{i_k}) = \frac{1}{(n-k)!} \sum_{i_{k+1}, \hdots , i_n}\epsilon_{i_1\hdots i_k i_{k+1}\hdots i_d}e^{i_{k+1}}\wedge \hdots \wedge e^{i_n}\, .
\end{equation*}

\noindent
where $\epsilon_{i_1\hdots i_k i_{k+1}\hdots i_n}$ is the Levi-Civita symbol in the convention $\epsilon_{1\hdots,d} = 1$. Furthermore, if $\alpha , \beta \in \Omega^r(M)$ we have:
\begin{equation*}
\ast_g^2 \alpha = (-1)^{r(n-r)}\alpha \, , \qquad \langle \alpha , \ast_g\beta \rangle_g = (-1)^{r(n-r)}\langle \ast_g\alpha ,\beta \rangle_g\, .
\end{equation*}

\noindent
In addition, we have:
\begin{equation*}
\langle u\wedge \alpha , \beta \rangle_g = \langle \alpha , \iota_{u^{\sharp}}\beta\rangle_g \,,
\end{equation*}

\noindent
for any $u\in \Omega^1(M)$, $\alpha \in \Omega^r(M)$ and $\beta \in \Omega^{r+1}(M)$, where the superscript $\sharp$ denotes musical isomorphism with respect to $g$. The exterior derivative of an $r$-form $\alpha \in \Omega^r(M)$ is defined through the following formula:
\begin{equation*}
\dd \alpha (v_0,\hdots ,v_r) = \sum_{i=0}^{r} (-1)^i v_i (\alpha(v_0, \hdots , \hat{v}_i, \hdots , v_r)) + \sum_{0\leq i< j \leq k} (-1)^{i+j}\alpha([v_i,v_j], v_0,\hdots, \hat{v}_i, \hdots , \hat{v}_j, \hdots , v_r)\,,
\end{equation*}

\noindent
where $v_0 , \hdots, v_r \in \mathfrak{X}(M)$ and the symbol \emph{hat} removes the underlying element. Therefore, in local coordinates, we have:
\begin{equation*}
\dd\alpha = \frac{1}{r!} \partial_{i_0}\alpha_{i_1 , \hdots , i_r} \dd x^{i_0} \wedge \hdots \wedge \dd x^{i_r}\,,
\end{equation*}

\noindent
where we are using the Einstein summation convention. Consequently, the local components of $\dd\alpha$ are given by:
\begin{equation*}
(\dd\alpha)_{i_0 , \hdots , i_r} = (r+1) \partial_{[i_0}\alpha_{i_1 , \hdots , i_r]}\,,
\end{equation*}

\noindent
where $[i_0,\hdots, i_r]$ denotes the skew-symmetrization of $(i_0,\hdots , i_r)$ with the following normalization factor:
\begin{equation*}
 \partial_{[i_0}\alpha_{i_1 , \hdots , i_r]} = \frac{1}{(r+1)!}\sum_{\sigma\in S_{r+1}}  (-1)^\sigma
 \partial_{\sigma(i_0)}\alpha_{\sigma(i_1) , \hdots , \sigma(i_r)}\, . 
\end{equation*}

\noindent
We define the Lie derivative of an $(r,0)$-tensor $\tau \in \Gamma(T^{\ast}M^{\otimes^r})$ by the following expression:
\begin{equation*}
(\cL_v \tau) (v_1,\hdots , v_r) = v \cdot  \tau (v_1,\hdots , v_r) - \sum_{i=1}^{r} \tau (v_1,\hdots , [v,v_i] , \hdots , v_r)\, .
\end{equation*}

\noindent
Let $\left\{f_t\right\}_{t\in\mathbb{R}}$ be the one-parameter group of diffeomorphisms generated by $v\in \mathfrak{X}(M)$. Then, we can equivalently write:
\begin{equation*}
\cL_v \tau = \frac{\dd}{\dd t} f^{\ast}_t \tau\vert_{t=0}\, . 
\end{equation*}

\noindent
for the Lie derivative of $\alpha$ along $v$, where $f^{\ast}_t \tau$ denotes the pullback of $\tau$ by the diffeomorphism $f_t \colon M\to M$, $t\in \mathbb{R}$. In particular, for an $r$-form $\alpha \in \Omega^r(M)$ we have the \emph{Cartan formula}:
\begin{equation*}
\cL_v \alpha = \dd\iota_v\alpha + \iota_v \dd \alpha\, .
\end{equation*}

\noindent
The Levi-Civita connection on $(M,g)$ will be denoted by $\nabla^g$. In terms of $\nabla^g$ the exterior differential of $\alpha \in \Omega^r(M)$ is given by skew-symmetrization in all entries:
\begin{equation*}
\dd \alpha(v_0,\hdots , v_r) = \sum_{i=0}^{r} (-1)^i (\nabla^g_{v_i}\alpha)(v_0 , \hdots , \hat{v}_i , \hdots , v_r) 
\end{equation*}

\noindent
or, equivalently, by:
\begin{equation*}
\dd \alpha = \sum_{i=1}^n e^i\wedge \nabla^g_{e_i} \alpha
\end{equation*}

\noindent
where $(e_1,\hdots , e_n)$ is a local orthonormal frame and $(e^1,\hdots , e^n)$ is its dual coframe. The formal adjoint $\dd^{g\ast}
\colon \Omega^{r+1}(M) \to \Omega^r(M)$ of the exterior derivative with respect to the $L^2$ norm defined by $g$ is given by:
\begin{equation*}
\dd^{g\ast}\colon \Omega^{r+1}(M) \to \Omega^r(M)\, , \qquad \alpha \mapsto - (-1)^{nr}\ast_g\dd\ast_g \alpha \, .
\end{equation*}

\noindent
In particular, we have:
\begin{equation*}
\dd^{g\ast}\alpha = -\sum_{i=1}^n \iota_{e_i} \nabla^g_{e_i} \alpha
\end{equation*}

\noindent
in terms of a local orthonormal frame $(e_1,\hdots , e_n)$ and its dual $(e^1,\hdots , e^n)$.

\noindent
Consider now the Levi-Civita connection as a differential operator:
\begin{equation*}
\nabla^g \colon \Gamma(T^{\ast}M^{\otimes^r}) \to \Gamma(T^{\ast}M^{\otimes^{r+1}})\, , \qquad \tau \mapsto \nabla^g\tau\, .
\end{equation*}

\noindent
The formal adjoint of this differential operator with respect to the $L^2$ norm defined by $g$ can be computed to be:
\begin{equation*}
(\nabla^{g\ast}\tau) (v_1,\hdots , v_r) = - \sum_{i=1}^n (\nabla^g_{e_i} \tau)(e_i,v_1,\hdots , v_r)\, ,\qquad v_1, \hdots , v_r  \in TM\, ,
\end{equation*}

\noindent
where $\tau$ is an $(r,0)$-tensor and $(e_1,\hdots , e_n)$ is a local orthornormal frame. We will refer to $\nabla^{g\ast}$ as the \emph{divergence operator} of $g$. Equivalently, we can write $\nabla^{g\ast}$ in terms of the following trace with respect to $g$:
\begin{equation*}
(\nabla^{g\ast}\tau) (v_1,\hdots ,v_r) =  - \mathrm{Tr}_g((\nabla^g\tau)(\,\cdot\, , v_1,\hdots , v_r))\, , \qquad v_1 , \hdots , v_r\in TM\, .
\end{equation*}

\noindent
In particular, the divergence of a tensor $\tau \in \Gamma(T^{\ast}M^{\otimes r+1})$ is given in local components by:
\begin{equation*}
(\nabla^{g\ast} \tau)_{i_1 , \hdots , i_r} = - (\nabla^{g})^{i_0}\tau_{i_0, i_1,\hdots , i_r}\, .
\end{equation*}

\noindent
The divergence operator restricts to a differential operator between differential forms which we denote for simplicity by the same symbol. For an $(r+1)$-form $\alpha \in \Omega^{r+1}(M)$ we have the following relation between the coadjoint and divergence operators:
\begin{equation*}
\nabla^{g\ast} \alpha = \dd^{g\ast}\alpha    \, , \qquad \alpha \in \Omega^{r+1}(M)\,,
\end{equation*}

\noindent
whence the restriction of $\nabla^{g\ast}$ to the differential forms is precisely the formal adjoint of the exterior derivative. Similarly, the divergence operator defines by restriction a map:
\begin{equation*}
\nabla^{g\ast}\colon \Gamma(T^{\ast}M^{\odot^{r+1}}) \to \Gamma(T^{\ast}M^{\odot^{r}})
\end{equation*}

\noindent
between symmetric tensors, which we denote again by the same symbol for ease of notation. On the other hand, in the same way that the skew-symmetrization of the Levi-Civita connection $\nabla^g$ defines a natural differential operator between differential forms, namely the exterior derivative, the symmetrization $\delta^g$ of $\nabla^g$ defines a differential operator between symmetric tensors given by:
\begin{equation*}
\delta^g\colon \Gamma(T^{\ast}M^{\odot^{r}}) \to \Gamma(T^{\ast}M^{\odot^{r+1}})\, ,\quad \tau\mapsto	(\delta^g \tau) (v_0,\hdots , v_r) = \frac{1}{r+1} \sum_{i=0}^{r}  (\nabla^g_{v_i}\tau)(v_0 , \hdots , \hat{v}_i , \hdots , v_r) \,.
\end{equation*}

\noindent
It can be checked that as defined above the formal adjoint of $\delta^g$ is precisely the restriction of $\nabla^{g\ast}$ to the symmetric tensors, that is:  
\begin{equation*}
\delta^{g\ast} = \nabla^{g\ast}\vert_{\Gamma(T^{\ast}M^{\odot^{r+1}})}\colon \Gamma(T^{\ast}M^{\odot^{r+1}}) \to \Gamma(T^{\ast}M^{\odot^r})\, .
\end{equation*}

\noindent
Using the exterior derivative $\dd$ and its adjoint operator $\dd^{g\ast}$, we define the Laplacian $\Delta_g$ on differential forms:
\begin{equation*}
\Delta_g \colon \Omega^r(M) \to \Omega^r(M) \, , \quad \alpha \mapsto \dd \dd^{g\ast}\alpha + \dd^{g\ast}\dd\alpha\, ,
\end{equation*}

\noindent
which in this form is sometimes called the \emph{Hodge-de-Rham Laplacian} or the \emph{positive} Laplacian. When acting on functions $f\in\cC^{\infty}(M)$ we have the identities:
\begin{equation*}
\Delta_g f = \dd^{g\ast}\dd f = -\mathrm{Tr}_g(\nabla^g\dd f) = (\nabla^g)^{\ast}\dd f\, .
\end{equation*}

\noindent
In our conventions the Riemannian curvature $\mathrm{R}^g$ of $g$ is given by:
\begin{equation*}
\mathrm{R}^g_{v_1,v_2} v_3 = \nabla^g_{v_1}\nabla^g_{v_2}v_3 - \nabla^g_{v_2}\nabla^g_{v_1}v_3 - \nabla^g_{[v_1,v_2]}v_3 \, , \qquad  v_1, v_2, v_3 \in \mathfrak{X}(M)\, .
\end{equation*}

\noindent
The Ricci tensor is in turn obtained from $\mathrm{R}^g$ by taking the following trace:
\begin{equation*}
\mathrm{Ric}^g(v_1,v_2) = \mathrm{Tr}(v_3\mapsto \mathrm{R}^g(v_3,v_1)v_2)\, .
\end{equation*}

\noindent
Therefore, in local coordinates, we have:
\begin{equation*}
\Ric^g(\partial_i , \partial_j) = \sum_{k=1}^n\mathrm{R}_{\,\,kij}^{g\,\,\,\,\,\, k}\, .
\end{equation*}

\noindent
Finally, the scalar curvature of $g$ is given by the trace of the Ricci tensor as defined above. By \emph{lowering} one index with the metric $g$, we will usually consider $\mathrm{R}^g$ as a section $\mathrm{R}^g \in \Gamma(\wedge^2 T^{\ast}M \otimes \wedge^2 T^{\ast}M)$.

\noindent
For further reference, given $v_1 , v_2 \in \mathfrak{X}(M)$ fixed, we define the following differential operator in terms of the Levi-Civita connection:
\begin{equation*}
D_{v_1, v_2}\colon \Met(M) \to \mathfrak{X}(M)\qquad g\mapsto\nabla^g_{v_1}v_2 \, .
\end{equation*}

\noindent
Its differential at a metric $g\in \Met(M)$ defines a map
\begin{equation*}
\dd_g D_{v_1, v_2}\colon \Gamma(T^{\ast}M\odot T^{\ast}M) \to \mathfrak{X}(M)
\end{equation*}

\noindent
given by:
\begin{equation*}
g(\dd_g D_{v_1, v_2} h, v_3) =\frac{1}{2}\big( (\nabla^g_{v_1}h)(v_2,v_3) + (\nabla^g_{v_2}h)(v_1,v_3) - (\nabla^g_{v_3}h)(v_1,v_2)\big), 
\end{equation*}

\noindent
for every $h\in \Gamma(T^{\ast}M\odot T^{\ast}M)$ and $v_3 \in \mathfrak{X}(M)$. In the following we consider the Ricci and scalar curvatures as smooth maps of Fréchet manifolds:
\begin{alignat*}{3}
\mathrm{Ric}\colon \Met(M) &\to \Gamma(T^{\ast}M\odot T^{\ast}M)\, ,
&\qquad g &\mapsto \mathrm{Ric}^g\, ,\\  
s\colon \Met(M) &\to C^{\infty}(M)\, , & \qquad g &\mapsto s^g\, .
\end{alignat*}

\noindent
Note that the Ricci tensor of $g$ satisfies the well-known identity:
\begin{equation*}
\nabla^{g\ast}\Ric^g = -\frac{1}{2} \dd s^g\, .
\end{equation*}

\noindent
We introduce the \emph{Lichnerowicz Laplacian} $\Delta^g_L$ associated to a Riemannian metric $g\in \Met(M)$ as follows:
\begin{eqnarray*}
 \Delta^g_L \colon \Gamma(T^{\ast}M\odot T^{\ast}M) &\to &\Gamma(T^{\ast}M\odot T^{\ast}M)\\
 h &\mapsto& \Delta^g_L h =  \nabla^{g\ast}\nabla^g h +  2\, h\circ_g \Ric^g   - 2 \mathrm{R}^g_o (h)
\end{eqnarray*}

\noindent
where:
\begin{equation*}
h\circ_g \Ric^g(v_1,v_2) = \frac{1}{2}\big(g(\iota_{v_1}h,\iota_{v_2}\Ric^g) + g(\iota_{v_2}h,\iota_{v_1}\Ric^g)\big)\, ,
\end{equation*}

\noindent
for every $v_1,v_2 \in TM$. Furthermore, we have defined $\mathrm{R}_o^g(h)\in \Gamma(T^{\ast}M\odot T^{\ast}M)$ to be the symmetric tensor determined by:
\begin{equation*}
\mathrm{R}^g_o (h)(v_1,v_2) = g(\mathrm{R}^g(\,\cdot\,, v_1, v_2 , \,\cdot\,),h)\, ,
\end{equation*}

\noindent
for every $v_1,v_2 \in TM$. In terms of an orthonormal frame $(e_1 , \hdots , e_n)$ we can equivalently write:
\begin{equation*}
\mathrm{R}^g_o (h)(v_1,v_2) = \sum_{i=1}^n h(\mathrm{R}^g(e_i,v_1)v_2 , e_i)\, .
\end{equation*}

\noindent
In particular, the local components of $\mathrm{R}^g(h)$ are given simply by $\mathrm{R}^g(h)_{jk} = \mathrm{R}^g_{\,\,ijkl} h^{il}$.

\noindent
The differentials of the operators $\Ric$ and $s$ at a point $g$ will be denoted respectively by:
\begin{eqnarray*}
\dd_g \mathrm{Ric}\colon\Gamma(T^{\ast}M\odot T^{\ast}M) \to \Gamma(T^{\ast}M\odot T^{\ast}M)\, , \quad \dd_g s\colon \Gamma(T^{\ast}M\odot T^{\ast}M) \to C^{\infty}(M)\, .
\end{eqnarray*}

\noindent
We have the following classical result.

\begin{lemma}\cite[\S 1]{Besse}
\label{lemma:variationRs}
The differentials of the Ricci and scalar curvature operators at $g$ are respectively given by:
\begin{align*}
  \dd_g\Ric(h) &= \frac{1}{2} \Delta_L h - \delta^{g}\nabla^{g\ast} h - \frac{1}{2}\nabla^g \dd\mathrm{Tr}_g(h) \,, \\ 
  \dd_g s (h) &=  \Delta_g \tr_g(h) + \nabla^{g\ast}\nabla^{g\ast} h- g(h,\Ric^g) \,,
\end{align*}

\noindent
for every $h\in \Gamma(T^{\ast}M\odot T^{\ast}M)$.
\end{lemma}

\noindent
We also define the Einstein curvature operator as the following smooth map:
\begin{equation*}
\mathbb{G}\colon \Met(M) \to \Gamma(T^{\ast}M\odot T^{\ast}M)\, , \qquad g\mapsto \mathbb{G}^g = \Ric^g - \frac{1}{2} s^g g
\end{equation*}

\noindent
in terms of the Einstein tensor of $g$. The previous lemma immediately implies the following result.

\begin{corollary}
\label{cor:variationEinstein}
The differential of the Einstein curvature operator at $g$ is given by:
\begin{equation*}
\dd_g\mathbb{G}(h) = \frac{1}{2} \Delta_L h - \delta^{g}\nabla^{g\ast} h - \frac{1}{2}\nabla^g \dd\mathrm{Tr}_g(h)  - \frac{1}{2} s^g h - \frac{1}{2} \big(\Delta_g \mathrm{Tr}_g (h) + \nabla^{g\ast}\nabla^{g\ast} h - g(\Ric^g , h) \big) g\,,
\end{equation*}

\noindent
for every $h\in \Gamma(T^{\ast}M\odot T^{\ast}M)$.
\end{corollary}

\noindent
Define now the smooth map of Fréchet manifolds $\nu\colon \Met(M) \to \Omega^n (M)$ by $g\mapsto \nu_g$. Its differential at $g\in \Met(M)$ is given by:
\begin{equation}
\label{eq:variationdetg}
\dd_g\nu \colon \Gamma(T^{\ast}M\odot T^{\ast}M) \to \Omega^n (M)\, , \qquad h\mapsto \frac{1}{2}\mathrm{Tr}_g(h) \nu_g
\end{equation}
 
\noindent
We use this expression in Section \ref{sec:EinsteinYM} to compute the critical points of the Einstein-Yang-Mills functional.

\noindent
Given $\alpha \in \Omega^r(M)$, consider now the smooth map of Fréchet manifolds:
\begin{equation*}
\star_{\alpha}\colon \Met(M) \to \Omega^{n-r}(M)\, , \quad g\mapsto \ast_g\alpha\, .
\end{equation*}

\noindent
Its differential at $g\in \Met(M)$ defines a linear map:
\begin{equation*}
\dd_g\star_{\alpha} \colon \Gamma(T^{\ast}M\odot T^{\ast}M)\to \Omega^{n-r}(M)
\end{equation*}

\noindent
given by:
\begin{equation}
\label{eq:variationvolumeform}
(\dd_g\star_{\alpha})(h) = \frac{1}{2} \mathrm{Tr}_g(h) \ast_g\alpha + \ast_g (\alpha^g_h)\, , \qquad h\in \Gamma(T^{\ast}M\odot T^{\ast}M)\,,
\end{equation}

\noindent
where $(-)^g_h\colon \Omega^r(M) \to \Omega^r(M)$ is the linear operation defined as $(\alpha)^g_h = 0$ if $r=0$ and otherwise:
\begin{equation}
\label{eq:lineagh}
(\alpha)^g_h(v_1, \hdots, v_r) = - \sum_{i=1}^r (-1)^{r+i} g(h(v_i),\alpha(v_1,\hdots ,\hat{v}_i, \hdots , v_r)),
\end{equation}

\noindent
where $v_1,\hdots ,v_r\in \mathfrak{X}(M)$. Equivalently:
\begin{equation*}
(\alpha)^g_h = - \sum_{k=1}^n h(e_k)\wedge \iota_{e_k}\alpha 
\end{equation*}

\noindent
in terms of any orthonormal frame $(e_1,\hdots , e_n)$.

\phantomsection
\bibliographystyle{JHEP}


\end{document}